\documentclass[12pt]{article}
\usepackage{setspace}
\usepackage{mathtools}
\usepackage[utf8]{inputenc}
\usepackage{url}

\usepackage[square,sort,comma,numbers]{natbib}

\usepackage{graphicx,graphics}
\usepackage{subcaption}
\usepackage{supertabular}
\usepackage{amsmath,amscd,amsfonts,amssymb}
\usepackage{fancyhdr}
\usepackage{color}
\usepackage{verbatim}

\usepackage{array,supertabular}
\usepackage{float}
\usepackage{rotating}
\usepackage{lscape}

\usepackage{booktabs,caption,fixltx2e}
\usepackage[flushleft]{threeparttable}
\usepackage{amsthm}

\usepackage{dcolumn}
\newcolumntype{L}{D{.}{.}{2,5}}

\usepackage{chngpage} 

\newtheorem{theorem}{Theorem}[section]
\newtheorem{lemma}[theorem]{Lemma}
\newtheorem{cor}[theorem]{Corollary}
\newtheorem{proposition}[theorem]{Proposition}
\newtheorem{corollary}[theorem]{Corollary}

\newtheorem{assumption}{Assumption}
\newtheorem{example}{Example}
\newtheorem{remark}{Remark}

\newenvironment{definition}[1][Definition.]{\begin{trivlist}
\item[\hskip \labelsep {\bfseries #1}]}{\end{trivlist}}
\def\mds{\medskip}
\def\Rb{{\mathbb R}}
\def\Pc{{\mathcal P}}
\def\Fc{{\mathcal P}}
\def\Nc{{\mathcal N}}

\long\def\symbolfootnote[#1]#2{\begingroup%
\def\thefootnote{\fnsymbol{footnote}}\footnotetext[#1]{#2}\footnotemark[#1]\endgroup}

\makeatletter
\@addtoreset{equation}{section}

\makeatother

\newcounter{Fig}[figure]

\newcounter{Tab}[table]

  {\refstepcounter{Tab}
   \addtocounter{table}{1}
   \samepage\vspace{0.2cm}
   \centerline{#1} \nobreak
   \begin{verse}{Table \thesection.\arabic{Tab}:}
  }%
  {\end{verse} \vspace{0.2cm}}

\relax

\newcommand{\XX}{\mbox{\boldmath $X$}}
\newcommand{\VV}{\mbox{\boldmath $V$}}
\newcommand{\YY}{\mbox{\boldmath $Y$}}
\newcommand{\ZZ}{\mbox{\boldmath $Z$}}

\newcommand{\BB}{\mbox{\boldmath $B$}}

\newcommand{\ee}{\mbox{\boldmath $e$}}
\newcommand{\xx}{\mbox{\boldmath $x$}}

\newcommand{\ttt}{\mbox{\boldmath $t$}}
\newcommand{\bbb}{\mbox{\boldmath $b$}}

\newcommand{\WW}{\mbox{\boldmath $W$}}

\newcommand{\cc}{\mbox{\boldmath $c$}}
\newcommand{\dd}{\mbox{\boldmath $d$}}

\newcommand{\pp}{\mbox{\boldmath $p$}}

\newcommand{\uu}{\mbox{\boldmath $u$}}
\newcommand{\vv}{\mbox{\boldmath $v$}}
\newcommand{\ww}{\mbox{\boldmath $w$}}

\newcommand{\zz}{\mbox{\boldmath $z$}}

\newcommand{\UU}{\mbox{\boldmath $U$}}

\def \Cb{{\mathbb C}}
\def \Db{{\mathbb D}}
\def \Eb{{\mathbb E}}

\def \Gb{{\mathbb G}}
\def \Hb{{\mathbb H}}
\def \Kb{{\mathbb K}}
\def \Lb{{\mathbb L}}
\def \Mb{{\mathbb M}}
\def \Nb{{\mathbb N}}
\def \Pb{{\mathbb P}}
\def \Rb{{\mathbb R}}

\def \Zb{{\mathbb Z}}

\def \Ac{{\mathcal A}}

\def \Bc{{\mathcal B}}
\def \Gc{{\mathcal G}}
\def \Zc{{\mathcal Z}}
\def \Ic{{\mathcal I}}
\def \Jc{{\mathcal J}}
\def \Wc{{\mathcal W}}
\def \Uc{{\mathcal U}}

\def \Vc{{\mathcal V}}
\def \Fc{{\mathcal F}}
\def \Cc{{\mathcal C}}

\def \Nc{{\mathcal N}}
\def \Pc{{\mathcal P}}

\def \Xc{{\mathcal X}}

\newcommand{\bqa}{\begin{eqnarray*}}
\newcommand{\eqa}{\end{eqnarray*}}
\newcommand{\bqan}{\begin{eqnarray}}
\newcommand{\eqan}{\end{eqnarray}}
\newcommand{\bqt}{\begin{quote}}
\newcommand{\eqt}{\end{quote}}
\newcommand{\bt}{\begin{tabbing}}
\newcommand{\et}{\end{tabbing}}
\newcommand{\bit}{\begin{itemize}}
\newcommand{\eit}{\end{itemize}}
\newcommand{\ben}{\begin{enumerate}}
\newcommand{\een}{\end{enumerate}}
\newcommand{\beq}{\begin{equation}}
\newcommand{\eeq}{\end{equation}}
\newcommand{\beqw}{\begin{equation*}}
\newcommand{\eeqw}{\end{equation*}}
\newcommand{\bdefi}{\begin{definition}}
\newcommand{\edefi}{\end{definition}}
\newcommand{\bpro}{\begin{proposition}}
\newcommand{\epro}{\end{proposition}}
\newcommand{\blem}{\begin{lemma}}
\newcommand{\elem}{\end{lemma}}
\newcommand{\bco}{\begin{corollary}}
\newcommand{\eco}{\end{corollary}}
\newcommand{\bdes}{\begin{description}}
\newcommand{\edes}{\end{description}}
\newcommand{\bre}{\begin{remark}}
\newcommand{\ere}{\end{remark}}

\newcommand{\eps}{\epsilon}

\def\mds{\medskip}
\def\1{{\mathbf 1}}
\def\0{{\mathbf 0}}

\begin{document}

\title{Sparse M-estimators in semi-parametric copula models}

\author{Benjamin Poignard \and Jean-David Fermanian}

\author{Jean-David Fermanian\footnote{Ensae-Crest, 5 avenue Henry le Chatelier, 91129 Palaiseau, France. E-mail address: jean-david.fermanian@ensae.fr}, Benjamin Poignard\footnote{Osaka University, Graduate School of Economics, 1-7, Machikaneyama, Toyonaka-Shi, Osaka-Fu, 560-0043, Japan. E-mail address: bpoignard@econ.osaka-u.ac.jp. Jointly affiliated at RIKEN Center for Advanced Intelligence Project (AIP), and CREST-LFA.}}

\date{\today}

\maketitle

\begin{abstract}
We study the large sample properties of sparse M-estimators in the presence of pseudo-observations. Our framework covers a broad class of semi-parametric copula models, for which the marginal distributions are unknown and replaced by their empirical counterparts. It is well known that the latter modification significantly alters the limiting laws compared to usual M-estimation. We establish the consistency and the asymptotic normality of our sparse penalized M-estimator and we prove the asymptotic oracle property with pseudo-observations, possibly in the case when the number of parameters is diverging. Our framework allows to manage copula-based loss functions that are potentially unbounded. Additionally, we state the weak limit of multivariate rank statistics for an arbitrary dimension and the weak convergence of empirical copula processes indexed by maps. We apply our inference method to Canonical Maximum Likelihood losses with Gaussian copulas, mixtures of copulas or conditional copulas. \textcolor{black}{The theoretical results are illustrated by two numerical experiments}.
\end{abstract}

\medskip

\noindent


\section{Introduction}

In this paper, we consider the parsimonious estimation of copula models within the semi-parametric framework: margins are left unspecified and a parametric copula model is assumed. The sparse assumption is motivated by the model complexity that occurs in copula modelling, where the parameterization may require the estimation of a large number of parameters. For instance, the variance-covariance matrix of a $d$-dimensional Gaussian copula involves the estimation of $d(d-1)/2$ parameters, the components of an unknown correlation matrix; in single-index copula, the underlying conditional copula is parameterized through a link function that depends on a potentially large number of covariates and thus parameters, which may not be all relevant for describing these conditional distributions. Since the seminal work of \cite{fan2001variable}, a significant amount of literature dedicated to sparsity-based M-estimators has been flourishing in a broad range of settings. In contrast, the sparse estimation of copula based M-estimators has benefited from a very limited attention so far. 
In~\cite{cai2014mixture}, the authors considered a mixture of copulas with a joint estimation of the weight parameters and the copula parameters, 
while penalizing the former ones only. However, a strong limitation of their approach is the parametric assumption formulated for the marginals, which greatly simplifies the large sample inference.~\cite{yang2019semiparametric} specified a penalized estimating equation-based estimator for single-index copula models and derived the corresponding large sample properties but assuming known margins. A theory covering the sparse estimation for semi-parametric copulas is an important missing stone in the literature. One of the key difficulties is the treatment of the values close to the boundaries of $[0,1]^d$, where some loss functions potentially ``explode''. 
The latter situation is not pathological. It often occurs for the standard Canonical Maximum Likelihood method or CML (see, e.g.,~\cite{genest1995,shi1995,tsukahara2005}) and many usual copula log-densities, as pointed out in~\cite{segers2012asymptotics} in particular.
Since the seminal works of~\cite{ruymgaart1972,ruymgaart1974}, the large sample properties of the empirical copula process $\widehat{\Cb}_n$ were established by, e.g.,~\cite{fermanian2004}, and such properties were applied to the asymptotic analysis of copula-based maximum likelihood estimators with pseudo-observations. In that case, the empirical copula process is indexed by the likelihood function: see~\cite{tsukahara2005,chen2005pseudo}, who considered some regularity conditions on the likelihood function to manage the values close to the boundaries of $[0,1]^d$. Similar conditions were stated likewise in~\cite{chen2005pseudo,zhang2016,hamori2019}, among others, where some bracketing number conditions on a suitable class of functions are assumed. These works share a similar spirit with~\cite{vaart2007,dehling2014}, who considered a general framework of empirical processes indexed by classes of functions under entropy conditions. Thanks to a general integration by parts formula,~\cite{radulovic2017weak} established the conditions for the weak convergence of the empirical copula process $\int f \text{d}\widehat{\Cb}_n$ indexed by a class of functions $f \in \mathcal{F}$ of bounded variation, the so-called Hardy–Krause variation. Their results do not require explicit entropy conditions on the class of functions.
In the same vein,~\cite{berghaus2017weak} assumed similar regularity conditions on the indexing functions but restricted their analysis to the two-dimensional copula case. It is worth mentioning that the techniques for stating the large sample analysis of semi-parametric copulas amply differ from the fully parametric viewpoint, for which the classical M-estimation theory obviously applies.

The present paper is then motivated by the lack of links between sparsity and semi-parametric copulas. Our asymptotic analysis for sparse M-estimators in the context of semi-parametric copulas builds upon the theoretical frameworks of~\cite{berghaus2017weak} and~\cite{radulovic2017weak}. The contribution of our paper is fourfold: first, we provide the asymptotic theory (consistency, oracle property for variable selection and asymptotic normality in the same spirit as \cite{fan2001variable}) for general penalized semi-parametric copula models, where the margins are estimated by their empirical counterparts. In particular, our setting includes the Canonical Maximum Likelihood method. Second, these asymptotic results are extended for (a sequence of) copula models in large dimensions, a framework that corresponds to the diverging dimension case, as in, e.g., \cite{fan2004nonconcave}.
Third, we prove the asymptotic normality of multivariate-rank statistics for any arbitrary dimension $d\geq 2$, extending Theorem 3.3 of~\cite{berghaus2017weak}. Fourth and finally, we prove the weak convergence of the empirical copula process indexed by functions of bounded variation, extending Theorem 5 of~\cite{radulovic2017weak} to cover the prevailing situation of unbounded copula densities. We emphasize that our theory is not restricted to i.i.d. data and potentially covers the case of dependent observations, as in \cite{radulovic2017weak}. 

\mds

The rest of the paper is organized as follows. Section~\ref{framework} details the framework and fix our notations. The large sample properties of our penalized estimator are provided in Section~\ref{asym_prop}. The situation of conditional copulas is managed in Section~\ref{section_cond_copulas}. Section~\ref{applications} discusses some examples and two simulated experiments to illustrate the relevance of our method. The theoretical results about multivariate rank statistics and empirical copula processes are stated in Appendix \ref{technicalities}. All the proofs, the theoretical results in the case of a diverging number of parameters and additional simulated experiments are provided in the Appendix.

\section{The framework}\label{framework}


\textcolor{black}{This section details the sparse estimation framework for copula models when the marginal distributions are non-parametrically managed. We} consider a sample of $n$ realizations of a random vector $\XX\in \Rb^d$, $\XX:=(X_{1},\ldots,X_{d})$. This sample is denoted as $\Xc_n =(\XX_{1},\ldots,\XX_{n})$. The observations may be dependent or not.
As usual in the copula world, we are more interested in the ``reduced'' random variables $U_{k}=F_{k}(X_{k})$, $k\in \{1,\ldots,d\}$, where
$F_{k}$ denotes the cumulative distribution function (c.d.f.) of $X_{k}$. \textcolor{black}{Throughout this paper, we make the blanket assumption that all $X_{k}$ have continuous marginals. Then,} the variables $U_k$ are uniformly distributed on $[0,1]$ and the joint law of $\UU:=(U_1,\ldots,U_d)$ is the uniquely defined copula of $\XX$ denoted by $C$.
To study the latter copula, it would be tempting to work with the sample $\Uc_n :=(\UU_1,\ldots,\UU_n)$ instead of $\Xc_n$.
Unfortunately, since the marginal c.d.f.s' $X_k$ are unknown in general, \textcolor{black}{this is still the case of $\Uc_n$, and the marginal c.d.f.s'} have to be replaced by consistent estimates.
Therefore, it is common to build a sample of pseudo-observations
$\widehat{\UU}_i= (\widehat{U}_{i,1},\ldots,\widehat{U}_{i,d})$, $i\in\{1,\ldots,n\}$, obtained from the initial sample $\Xc_n$. Here and as usual, set $\widehat{U}_{i,k} = F_{n,k}(X_{i,k})$
for every $i\in\{1,\ldots,n\}$ and every $k\in\{1,\ldots,d\}$, using the $k$-th re-scaled empirical c.d.f.
$ F_{n,k}(s):= (n+1)^{-1} \sum^n_{i=1} \mathbf{1}\{X_{i,k} \leq s\}.$
We will denote by $G_{n,k}$, $k\in\{1,\ldots,d\}$, the empirical c.d.f. of the (unobservable) random variable $U_k$, i.e.
$ G_{n,k}(u):= (n+1)^{-1} \sum^n_{i=1} \mathbf{1}\{F_k(X_{i,k}) \leq u\}.$
The empirical c.d.f. of $\UU$ is $G_n$, i.e.
$ G_{n}(\uu):= (n+1)^{-1} \sum^n_{i=1} \mathbf{1}\{\UU_i \leq \uu\}$ for any $\uu\in [0,1]^d$. We denote by $\alpha_n$ the usual empirical process associated with the sample $(\UU_i)_{i=1,\ldots,n}$, i.e.
$$ \alpha_n(\uu):=\sqrt{n}\big( G_n-C\big)(\uu)=\sqrt{n}\Big\{ \frac{1}{n}\sum_{i=1}^n \1 \big( \UU_i \leq \uu \big) - C(\uu)\Big\}.$$
The natural estimator of the true underlying copula $C$, i.e. the c.d.f. of $\UU$, is the empirical copula map
\begin{equation}
\widehat C_n(\uu) := \frac{1}{n}\overset{n}{\underset{i=1}{\sum}}\mathbf{1}\Big\{F_{n,1}(X_{i,1})\leq u_1,\ldots,F_{n,d}(X_{i,d})\leq u_d\Big\},
\label{empir_copula}
\end{equation}
and the associated empirical copula process is $\widehat \Cb_n :=\sqrt{n}(\widehat C_n - C)$.

\mds

\textcolor{black}{Hereafter, we select a parametric family of copulas $C_\theta$, $\theta \in \Theta\subset \Rb^p$, and we assume it contains the true copula $C$: there exists $\theta_0$ (the ``true value'' of the parameter) s.t. $C=C_{\theta_0}$.}
We want to cover the usual case of semi-parametric dependence models, for which there is an orthogonality condition of the type
$ \Eb\big[ \nabla_\theta \ell(\UU;\theta_0) \big]=0,$
for some family of loss functions $\ell :(0,1)^d \times \Theta\rightarrow \Rb$.
\textcolor{black}{The dimension $p$ of the copula parameter $\theta$ will be fixed hereafter. 
In the Appendix, it will be allowed to tend to the infinity with the sample size $n$.}
The function $\ell$ is usually defined as a \textcolor{black}{quadratic loss} or minus a log-likelihood function.
\textcolor{black}{Note that $\ell$ has not to be defined on the boundaries of $[0,1]^d$ at this stage because the law of $\UU$ was assumed to be continuous. 
Moreover, an important contribution of the paper will be to deal with some maps $\ell$ that cannot be continuously extended on $[0,1]^d$.}


\mds

For the sake of the estimation of $\theta_0$, let us specify a statistical criterion.
Consider a global loss function $\Lb_n$ from $\Theta \times (0,1)^{dn}$ to $\Rb$.
The value $\Lb_n(\theta;\uu_1,\ldots,\uu_n )$ evaluates the quality of the ``fit'' 
given $\UU_i=\uu_i$ for every $i\in\{1,\ldots,n\}$ and under $\Pb_{\theta}$.
Hereafter, we assume there exists a continuous 
function $\ell : \Theta \times (0,1)^{d} \rightarrow \Rb$ such that 
 \begin{equation}
 \Lb_n(\theta;\uu_1,\ldots,\uu_n) := \overset{n}{\underset{i=1}{\sum}} \ell(\theta;\uu_i),
\label{loss_as_a_sum}
 \end{equation}
for every $\theta\in\Theta$ and every $(\uu_1,\ldots,\uu_n)$ in $(0,1)^{dn}$.
As usual for the inference of semiparametric copula models, the empirical loss $\Lb_n(\theta;\Uc_n)$ cannot be calculated since we do not observe the realizations of $\UU$ in practice. Therefore, invoking the ``pseudo-sample'' $\widehat{\Uc}_n := (\widehat{\UU}_1,\ldots,\widehat{\UU}_n)$, the empirical loss $\Lb_n(\theta;\Uc_n)$ will be approximated by $\Lb_n(\theta;\widehat\Uc_n)$, a quantity called ``pseudo-empirical'' loss function.  

\begin{example}
A key example is the Canonical Maximum Likelihood method: the law of $\UU$ (\textcolor{black}{i.e. the copula of $\XX$)} belongs to a parametric family $\Pc:=\{\Pb_\theta, \,\theta \in \Theta\}$ and $C=C_{\theta_0}$. There, for i.i.d. data, $\ell(\uu;\theta)=-\ln c(\uu;\theta)$, minus the log-copula density of $C_\theta$ w.r.t. the Lebesgue measure on $(0,1)^d$.

\end{example}


Now, \textcolor{black}{we assume} that the unknown parameter is sparse and we introduce a penalization term. Our criterion becomes
\begin{equation} \label{obj_crit}
\widehat{\theta} \,\textcolor{black}{\in} \,\underset{\theta \in \Theta}{\arg \; \min} \; \Big\{ \Lb_n(\theta;\widehat{\Uc}_n) + n \overset{p}{\underset{k=1}{\sum}}\pp(\lambda_n,|\theta_k|)\Big\},
\end{equation}
\textcolor{black}{when such a minimizer exists}. Here, \textcolor{black}{$\pp(\lambda_n,x)$, for $x \geq 0$, is a penalty function, where} $\lambda_n\geq 0$ is a tuning parameter which depends on the sample size and enforces a particular type of sparse structure. \textcolor{black}{Throughout the} paper, we will implicitly work under the following assumption. 
\addtocounter{assumption}{-1}
\begin{assumption}
\textcolor{black}{The parameter space $\Theta$ is a borelian subset of $\Rb^p$. 
The function $\theta \mapsto \Eb[\ell(\theta;\UU)]$ is uniquely minimized on $\Theta$ at $\theta = \theta_0$, and an open neighborhood of $\theta_0$ is contained in $\Theta$.}
\label{assump_theta0}
\end{assumption}
\textcolor{black}{Note that only the uniqueness of $\theta_0$ is required but not that of $\widehat{\theta}$. In~\cite{fan2001variable,fan2004nonconcave}, $\Theta$ is assumed to be an open subset. Nonetheless, it may happen that $\theta_0$ belongs to the frontier of $\Theta $. This may be the case for the estimation of weights in mixture models (see Example \ref{ex_copula_mixture} below), e.g. Actually, our results will apply even if $\Theta$ does not contain an open neighborhood of $\theta_0$. Indeed, it is sufficient that there exists a convex open neighborhood of $\theta_0$ in $\Theta$: for some open ball $B_\delta (\theta_0)\subset \Rb^p$ centered at $\theta_0$ and with a radius $\delta>0$ and for every parameter $\theta_1 \in \Theta \cap B_\delta(\theta_0)$, every element of the segment $t\theta_0+(1-t)\theta_1$, $t\in [0,1]$  belongs to $\Theta$ and $\theta_1$ is the limit of a sequence of elements in $B_\delta (\theta_0)\cap \text{Int}(\Theta)$, where $\text{Int}(\Theta)$ is the interior of $\Theta$. Therefore, we will define the partial derivatives of any map $h:\Theta \mapsto \Rb$ at some point on the frontier of $\Theta$ (in particular $\theta_0$) by continuity. For example, when the derivative of $\theta \mapsto \ell(\theta;\uu)$ exists in the interior of $\Theta$, for some $\uu$, it will be defined at $\theta_0$ by $ \nabla_{\theta } \ell(\theta_0;\uu):= \lim_{\theta \rightarrow \theta_0, \theta \in \Theta} \nabla_{\theta } \ell(\theta;\uu)$, assuming the latter limit exists. To lighten the presentation, we will keep Assumption~\ref{assump_theta0} as above. By slightly strengthening some regularity assumptions, the case of $\theta_0$ on the boundary of $\Theta$ will be straightforwardly
 obtained (essentially by imposing the continuity of some derivatives around $\theta_0$). }


\mds


Some well-known penalty functions are the LASSO, \textcolor{black}{where} $\pp(\lambda,\textcolor{black}{x})=\lambda \textcolor{black}{x}$ \textcolor{black}{for every $x \geq 0$}, and the non-convex SCAD and MCP. The SCAD penalty of \cite{fan2001variable} is defined as: \textcolor{black}{for every $x \geq 0$},
\begin{equation*}
\pp(\lambda,\textcolor{black}{x}) = \begin{cases}
\lambda \textcolor{black}{x}, & \text{for} \;  \textcolor{black}{x} \leq \lambda, \\
\frac{1}{2(a_{\text{scad}}-1)}\big(2 a_{\text{scad}}\lambda \textcolor{black}{x}-\textcolor{black}{x}^2-\lambda^2\big), &  \text{for} \; \lambda \textcolor{black}{< x} \leq a_{\text{scad}} \lambda, \\
(a_{\text{scad}}+1)\lambda^2/2, & \text{for} \;  \textcolor{black}{x} > a_{\text{scad}}\lambda,
\end{cases}
\end{equation*}
where $a_{\text{scad}}>2$. The MCP due to~\cite{zhang2010} is defined for $b_{\text{mcp}}>0$ as follows: \textcolor{black}{for every $x\geq 0$}, 
\begin{equation*}
\pp(\lambda,\textcolor{black}{x})  = \big(\lambda \textcolor{black}{x}-\frac{\textcolor{black}{x}^2}{2 b_{\textcolor{black}{\text{mcp}}}}\big) \mathbf{1}\big(\textcolor{black}{x\leq b_{\text{mcp}}\lambda}\big) + \lambda^2\frac{b_{\text{mcp}}}{2} \mathbf{1}
\big(\textcolor{black}{x> b_{\text{mcp}}\lambda}\big).
\end{equation*}

\textcolor{black}{Note that when $a_{\text{scad}} \rightarrow \infty$ (resp. $b_{\text{mcp}} \rightarrow \infty$), the SCAD (resp. MCP) penalty behaves as the LASSO penalty since all coefficients of $\theta$ are equally penalized.}
The idea of sparsity for copulas naturally applies to a broad range of situations. 
In such cases, the parameter value zero usually plays a particular role, possibly after reparameterisation.
This is in line with the usual previously cited penalties. 
\begin{example}\label{ex_zero_copula_coeff}
 Consider a Gaussian copula model in dimension $d>>1$, whose parameter is a correlation matrix $\Sigma $. 
The description of all the underlying dependencies between the components of $\UU$ is a rather painful task. Then, the sparsity of $\Sigma$ becomes a nice property. Indeed, the independence between two components of $\UU$ is equivalent to the nullity of their corresponding coefficients in $\Sigma$. 
\end{example}
\begin{example}\label{ex_copula_mixture}
The inference of mixtures of copulas may justify the application of a penalty function. 
Indeed, consider a set of known $d$-dimensional copulas $C^{(k)}$, $k\in\{1,\ldots,p+1\}$.
In practice, we could try to approximate the true underlying copula $C$ by a 
mixture $\sum^{p+1}_{k=1}\pi_k C^{(k)}, \sum^{p+1}_{k=1} \pi_k=1$, $\pi_k\in [0,1]$ for every $k$.
\textcolor{black}{Here, the underlying parameter is the vector of weights $\theta=(\pi_1,\ldots,\pi_p)$ and $\Theta$ is defined as
$ \normalfont\Theta_{\text{mixt},p}:=\{(\pi_1,\ldots,\pi_{p}); \pi_j\in[0,1],j\in\{1,\ldots,p\}; \sum_{j=1}^{p} \pi_j \leq 1 \} $.}
If a weight is estimated as zero, its corresponding copula does not matter to approximate $C$. The latter model is generally misspecified, but our theory will apply even in this case, interpreting $\widehat\theta$ in~(\ref{obj_crit}) as an estimator of a ``pseudo-true'' value $\theta_0$. If we apply the CML method, $\Lb_n(\theta;\widehat\Uc_n)$ is the log-copula density of $\sum^{p}_{k=1}\theta_k C^{(k)}+(1-\sum_{k=1}^p \theta_k) C^{(p+1)}$. When some or all copulas $C^{(k)}$ depend on unknown parameters that need to be estimated in addition to the weights, the penalty function could also be applied to these copula parameters. 
\end{example}
Dealing with conditional copulas (\cite{fermanian2012time}, e.g.), 
\textcolor{black}{our framework will be slightly modified.}
Now, the law of a random vector $\XX\in \Rb^d$ knowing the vector of covariates $\ZZ=\zz \in \Rb^m$
is given by a parametric conditional copula whose parameter depends on a known map of $\zz$ and is denoted $\theta(\zz;\beta)$, $\beta \in \Rb^q$.
Beside, the law of the margins $X_k$, $k\in \{1,\ldots,d\}$, given $\ZZ=\zz$ are unknown and we assume they do not depend on $\zz$.
In other words, the conditional law of $\XX$ given $\ZZ$ is assumed to be
$$ \Pb(\XX \leq \xx | \ZZ=\zz)=C_{\theta(\zz;\beta)} (F_1(x_1),\ldots,F_d(x_d)\big), \; \xx\in \Rb^d, \zz\in \Rb^m.$$
Therefore, as for the CML method and in the i.i.d. case, an estimator of $\beta$ would be
$$ \widehat \beta \,\textcolor{black}{\in} \,\arg\max_\beta 
\sum_{i=1}^n \ln c_{\theta(\ZZ_i;\beta)} \big(\widehat{U}_{i,1},\ldots,\widehat{U}_{i,d} \big).$$
Surprisingly and to the best of our knowledge, the asymptotic theory of such estimators has apparently not been explicitly stated in the literature until now.
This will be the topic of Section~\ref{section_cond_copulas}.
\begin{example}\label{ex_single_index_sparsity}
Sparsity naturally applies to single-index copulas (see, e.g., \cite{fermanian2018}). The function $\pp(\lambda_n,\cdot)$ is now specified with respect to the underlying $\beta$ parameter. In other words, sparsity refers to the situation where only a (small) 
subset of the $\ZZ$ components is relevant to describe the dependencies between the $\XX$ components given $\ZZ$. 
Consider the conditional Gaussian copulas, as Example 4 of \cite{fermanian2018}. Here, the correlation matrix $\Sigma$ would be a function of $\ZZ=\zz$. It may be rewritten $\Sigma(\zz)=\big[\sin\big(\frac{\pi}{2}\tau_{kl}(\zz^\top\beta)\big)\big]_{1\leq k,l \leq d}$, where $\tau_{kl}(\zz^\top\beta)$ denotes the conditional Kendall's tau of $(X_k,X_l)$ given $\ZZ=\zz$. 
\end{example}

\section{Asymptotic properties}\label{asym_prop}

To prove the asymptotic results, we consider two sets of assumptions: one is related to the loss function; the other one concerns the penalty function. First, define the support of the true parameter as $\Ac:=\big\{k: \theta_{0,k}\neq 0, k = 1,\ldots,p\big\}$. 
We will implicitly assume a sparsity assumption, i.e. the cardinal of $\Ac$ is ''significantly'' smaller than $p$.
\begin{assumption}\label{regularity_assumption}
The map $\theta\mapsto \ell(\theta;\uu)$ is thrice differentiable on $\Theta$, for every $\uu \in (0,1)^d$. 
\textcolor{black}{The parameter $\theta_0$ satisfies the first order condition $\Eb[\nabla_{\theta}\ell(\theta_0;\UU)]=0$.}
Moreover, $\Hb := \Eb[\nabla^2_{\theta\theta^\top}\ell(\theta_0;\UU)]$ and $\Mb := \Eb[\nabla_{\theta}\ell(\theta_0;\UU) \nabla_{\theta^\top}\ell(\theta_0;\UU)]$ exist and are positive definite.
Finally, for every $\eps>0$, there exists a constant $K_\epsilon$ such that
\begin{equation*}
\sup_{\{\theta; \|\theta - \theta_0\|<\eps\}} \;\sup_{j,l,m}\big|\Eb[\partial^3_{\theta_j\theta_l\theta_m} \ell( \theta;\UU)] \big| \leq K_\epsilon.
\end{equation*}
\end{assumption}
\begin{assumption}
\label{cond_reg_copula}
For any $j\in \{1,\ldots,d\}$, the copula partial derivative $\dot C_j(\uu):=\partial C(\uu)/\partial u_j$ exists and is continuous on $V_j:=\{\uu\in [0,1]^d; u_j\in (0,1) \}$. For every couple $(j_1,j_2)\in \{1,\ldots,d\}^2$, the second-order partial derivative $\ddot C_{j_1,j_2}(\uu):=\partial^2 C(\uu)/\partial u_{j_1}\partial u_{j_2}$ exists and is continuous
on $V_{j_1}\cap V_{j_2}$. Moreover, there exists a positive constant $K$ such that
\begin{equation}
 \big|\ddot C_{j_1,j_2}(\uu) \big|\leq K \min \Big( \frac{1}{u_{j_1}(1-u_{j_1})},  \frac{1}{u_{j_2}(1-u_{j_2})} \Big),\; \uu\in V_{j_1}\cap V_{j_2}.  
\label{cond_2nd_der_cop}    
\end{equation}
\end{assumption}
When $\uu$ does not belong to $(0,1)^d$ (i.e. when one of its components is zero or one), we have defined 
$\dot C_j(\uu):=\lim\sup_{t\rightarrow 0} \big\{ C(\uu+ t\ee_j) - C(\uu)\big\} /t$.
It has been pointed out in~\cite{segers2012asymptotics} 
that Condition~\ref{cond_reg_copula} is satisfied for bivariate Gaussian and bivariate extreme-value copulas in particular.
We formally state this property for $d$-dimensional Gaussian copulas in Section \ref{verif_regul_Gaussian_cop} of the Appendix.


\begin{assumption}
\textcolor{black}{For some $\omega$}, the family of maps $\Fc=\Fc_1\cup\Fc_2\cup\Fc_3$ from $(0,1)^d$ to $\Rb$ is $g_\omega$-regular (see Definition~\ref{def_gomega_reg} in Appendix \ref{technicalities}), with
$$\Fc_1:=\{ f:\uu\mapsto \partial_{\theta_k}\ell(\theta_0;\uu); k=1,\ldots,p\},$$
$$\Fc_2:=\{ f:\uu\mapsto \partial^2_{\theta_k,\theta_l}\ell(\theta_0;\uu); k,l=1,\ldots,p\},\;\text{and}$$
$$\Fc_3:=\{ f:\uu\mapsto \partial^3_{\theta_k,\theta_l,\theta_j}\ell(\theta;\uu); k,l,j=1,\ldots,p, \; \|\theta-\theta_0\|<K\},$$
for some constant $K>0$.
\label{assumption_regularity}
\end{assumption}

We will denote by $\partial_2\pp(\lambda,x)$ (resp. $\partial^2_{2,2}\pp(\lambda,x)$) the first order (resp. second order) derivative of $x\mapsto \pp(\lambda,x)$, for any $\lambda$.

\begin{assumption}\label{assumption_regularizer}
Defining
$$a_n := \max_{1\leq j\leq p} \big\{ \partial_2  \pp(\lambda_n,|\theta_{0,j}|),\theta_{0,j}\neq 0 \big\}\;\text{and}
\;b_n := \max_{1\leq j\leq p} \big\{ \partial^2_{2,2} \pp(\lambda_n,|\theta_{0,j}|),\theta_{0,j}\neq 0 \big\},$$
assume that $a_n\rightarrow 0$ and $b_n\rightarrow 0$ when $n\rightarrow \infty$. Moreover, there exist constants $M$ and $\bar M$ such that
$$ | \partial^2_{2,2} \pp(\lambda_n,\theta_1) - \partial^2_{2,2} \pp(\lambda_n,\theta_2) |\leq M |\theta_1 -\theta_2 |,$$
for any real numbers $\theta_1,\theta_2$ such that $\theta_1,\theta_2 > \bar M\lambda_n  $.
\end{assumption}

Assumption~\ref{regularity_assumption} is a standard regularity condition for asymptotically normal M-estimators. Assumptions~\ref{cond_reg_copula} is a smoothness condition on the copula $C$ and is similar to Condition 4.1 of~\cite{segers2012asymptotics} and Condition 2.1 of~\cite{berghaus2017weak}. It ensures that the second-order derivatives with respect to $\uu$ do not explode ``too rapidly'' when $\uu$ approaches the boundaries of $[0,1]^d$. Assumption~\ref{assumption_regularity} is related to the indexing function of the copula process, here by the first, second and third order derivatives of the loss function $\ell(\theta;.)$. \textcolor{black}{The $g_{\omega}$-regularity of these functions ensures that they are of bounded Hardy-Krause variation - similar to assumption F of~\cite{radulovic2017weak} - together with some integrability conditions. Here, $\omega$ is some fixed number in $(0,1/2)$ entering in the definition of 
a weight function $\min_k \min(u_k,1-u_k)^\omega$, as specified in Definition \ref{def_gomega_reg}, point (ii). Such a weight function is related to 
the theory of weighted empirical processes applied in~\cite{berghaus2017weak}.} Assumption~\ref{assumption_regularizer} is dedicated to the regularity of the penalty function, and includes conditions in the same vein as~\cite{fan2004nonconcave}, Assumption 3.1.1. 
\textcolor{black}{Note that the LASSO, the SCAD and the MCP penalties fulfill Assumption~\ref{assumption_regularizer}}. Our first result establishes the existence of a consistent penalized M-estimator.
\begin{theorem}\label{bound_proba}
\textcolor{black}{Suppose Assumptions \ref{oscillation_modulus_assump}-\ref{tail_empir_processes} given in Appendix \ref{technicalities} are satisfied. Let some $$\omega \in \Big(0,\min\big\{\frac{\textcolor{black}{\kappa}_1}{2(1-\textcolor{black}{\kappa}_1)},\frac{\textcolor{black}{\kappa}_2}{2(1-\textcolor{black}{\kappa}_2)},\textcolor{black}{\kappa}_3 - \frac{1}{2}\big\}\Big),$$ 
and suppose Assumptions~\ref{assump_theta0}-\ref{assumption_regularizer} hold \textcolor{black}{for this $\omega$}.} 
Then, there exists a sequence of estimators $\widehat{\theta}$ \textcolor{black}{as defined in (\ref{obj_crit}) which satisfies} 
\begin{equation*}
\|\widehat{\theta}-\theta_0\|_2 = O_p\Big(\ln(\ln n) n^{-1/2} + a_{n}\Big).
\end{equation*}
\end{theorem}
The proof is postponed in Section \ref{proof_results_sparsity} of the Appendix. \textcolor{black}{The factor $\ln(\ln n)$ could be replaced by any sequence that tends to the infinity with $n$, as in Corollary~\ref{cor_GC}. It has been arbitrarily chosen for convenience}. \textcolor{black}{We will apply this rule throughout the article.}

\mds 

We now show that the penalized estimator $\widehat{\theta}$ satisfies the oracle property in the sense of \cite{fan2001variable}: the true support can be recovered asymptotically and the non-zero estimated coefficients are asymptotically normal.
Denote by $\widehat\Ac:=\big\{k: \widehat\theta_{k}\neq 0, k = 1,\ldots,p\big\}$ the (random) support of our estimator.
The following theorem states the main results of this section. It uses some standard notations for concatenation of sub-vectors, as recalled in Appendix \ref{technicalities}. 
\begin{theorem}\label{oracle_property}
In addition to the assumptions of Theorem~\ref{bound_proba},
assume that the \textcolor{black}{penalty function} satisfies $\underset{n \rightarrow \infty}{\lim \, \inf} \; \underset{x \rightarrow 0^+}{\lim \, \inf} \; \lambda^{-1}_n \partial_2\pp(\lambda_n,x) > 0$. Moreover, assume $\lambda_n \rightarrow 0$, $\sqrt{n} \lambda_n (\ln(\ln n))^{-1} \rightarrow \infty$ and $a_n=o(\lambda_n)$.
Then, the consistent estimator $\widehat{\theta}$ of Theorem~\ref{bound_proba} satisfies the two following properties.
\begin{itemize}
    \item[(i)] Support recovery: $\underset{n \rightarrow \infty}{\lim} \; \Pb\left(\widehat{\Ac} = \Ac\right) = 1$. 
    \item[(ii)] Asymptotic normality: if, in addition, the conditions~\ref{cond_wc_empir_process}-\ref{regularity_cond_wc} in Appendix \ref{technicalities} (applied to $\Fc_1$ instead of $\Fc$) are fulfilled and $\sqrt{n}\lambda_n^2=o(1)$, then
    \begin{equation*}
    \sqrt{n}\Big[\Hb_{\Ac\Ac}+\mathbf{B}_n(\theta_0)\Big]\Big\{\big(\widehat{\theta}-\theta_0\big)_{\Ac} + \big[\Hb_{\Ac\Ac}+\mathbf{B}_n(\theta_0)\big]^{-1}\mathbf{A}_n(\theta_0) \Big\} \overset{d}{\underset{n \rightarrow \infty}{\longrightarrow}} \WW,
    \end{equation*}
    where $\Hb_{\Ac\Ac} = \Big[ \Eb\big[  \partial^2_{\theta_k\theta_l}\ell(\theta_0;\UU)\big] \Big]_{k,l \in \Ac}$,
    $\normalfont\mathbf{A}_n(\theta_0) = \big[\partial_{2}\pp(\lambda_n,|\theta_{0,k}|)\text{sgn}(\theta_{0,k})\big]_{k \in \Ac}$,\\
    $\mathbf{B}_n(\theta) =\normalfont \text{diag}\big(\partial^2_{2,2}\pp(\lambda_n,|\theta_{0,k}|),\,k\in\Ac\big)$ and $\WW$ a $|\Ac|$-dimensional centered random vector defined as 
    {\small{\begin{equation*}
    W_j := (-1)^d\int_{(0,1)^d} \Cb(\uu) \, \partial_{\theta_j}\ell(\theta_0;d\uu)+
\sum_{ \substack{I \subset \{1,\ldots,d\} \\ I\neq \emptyset, I\neq \{1,\ldots,d\} } } (-1)^{|I|}
  \int_{(0,1)^{|I|}} \Cb(\uu_{I}:\1_{-I}) \, \partial_{\theta_j}\ell(\theta_0;d\uu_{I};\1_{-I}),
    \end{equation*}}}
    for $j \in \Ac$, with $\Cb(\uu):= \alpha_C(\uu) - \sum_{k=1}^d \dot C_k (\uu)\alpha_C(\1_{-k}:u_k)$ and 
    $\alpha_C$ \textcolor{black}{is the process defined in Assumption~\ref{cond_wc_empir_process}.} 
\end{itemize}
\end{theorem}
\textcolor{black}{The proof is relegated to} Section \ref{proof_results_sparsity} of the Appendix.
\textcolor{black}{The existence and the meaning of the integrals defining $W_j$  comes from Assumption~\ref{regularity_cond_wc}.
The proofs of the two latter theorems rely on a third-order limited expansion of our empirical loss w.r.t. its arguments.
The negligible terms are managed thanks to a ULLN (Corollary~\ref{cor_GC}).  
An integration by parts formula (Theorem~\ref{Th_fondam_multiv_rank_stat}) allows to write the main terms as sums of integrals of the empirical copula process w.r.t. some ``sufficiently regular'' functions. The latter ones are deduced from the derivatives of the loss, justifying the concept of 
$g_\omega$ regularity and Assumption~\ref{assumption_regularity}. 
A weak convergence result of the weighted empirical copula process concludes the asymptotic normality proof.}

\mds 

\textcolor{black}{Note that the previous results apply with dependent observations $(\XX_i)$. Indeed, in Theorem~\ref{oracle_property} (ii), we only require the weak convergence of the process $(\alpha_n)$ to $\alpha_C$.
In the i.i.d. case, $\alpha_C$ is a Brownian Bridge and $\WW$ is a Gaussian random vector. 
The latter assertion is still true when $(\XX_i)$ is a strongly stationary and geometrically alpha-mixing sequence, due to Proposition 4.4 in~\cite{berghaus2017weak}.}
The existence and the meaning of the random variable 
$\int_{(0,1)^{|I|}} \Cb(\uu_{I}:\1_{-I}) \, \partial_{\theta_j}\ell(\theta_0;d\uu_{I};\1_{-I})$ come from Assumption~\ref{regularity_cond_wc}.
Our conditions on $(a_n,\lambda_n)$ allow to use the SCAD or the MCP penalty, but not the LASSO because $a_n=\lambda_n$ in the latter case. \textcolor{black}{In other words, Theorem~\ref{bound_proba} may apply with LASSO but not Theorem~\ref{oracle_property}.}
The fact that LASSO does not yield the oracle property has already been noted in the literature: see \cite{zou2006adaptive} and the references therein.
Actually, consider any penalty such that $\pp(\lambda,t)$ does not depend on $t>0$, when $\lambda$ is sufficiently small, as for the SCAD and MCP cases. 
Then,  
$$\partial_2\pp(\lambda_n,|\theta_{0,j}|)=\partial^2_{2,2}\pp(\lambda_n,|\theta_{0,j}|)=0, \; \; j\in \Ac,$$ 
when $n$ is sufficiently large because $\lambda_n \rightarrow 0$, implying $a_n=b_n=0$. Thus, Assumption~\ref{assumption_regularizer} is satisfied
and $\mathbf{A}_n(\theta_0)=\mathbf{B}_n(\theta_0) = \mathbf{0}$, 
for $n$ large enough. Therefore, for such \textcolor{black}{penalty functions}, the asymptotic law of $\widehat\theta$ becomes a lot simpler.
\begin{cor}
\label{cor_AnBn_null}
Assume we have chosen the SCAD or the MCP penalty.
Then, \textcolor{black}{under the assumptions of Theorem~\ref{oracle_property} (i) and (ii)},
we have
\begin{equation*}
\sqrt{n}(\widehat{\theta}-\theta_0)_{\Ac} \overset{d}{\underset{n\rightarrow \infty}{\longrightarrow}} \Hb^{-1}_{\Ac\Ac} \mathbf{W},
\end{equation*}
where $\WW$ is defined in Theorem~\ref{oracle_property}.
\end{cor}
\textcolor{black}{Theorem \ref{oracle_property} establishes that the non-convex penalization procedure for semi-parametric copula models is an ``oracle'': for a consistent estimator $\widehat{\theta}$, the true sparse support is correctly identified in probability, 
and the non-zero estimated parameters have the same limiting law as if the support of $\theta_0$ were known. 
The limit $\WW$ is in the same vein as the one in Theorem 5 of \cite{radulovic2017weak}. However, in contrast to the latter result which assumes bounded indexing functions on $[0,1]^d$, our framework covers the case of unbounded ones. In particular, our indexing functions $\uu\mapsto \partial_{\theta_j}\ell(\theta_0;\uu)$ are not constrained to be bounded on $(0,1)^d$.}


\mds

\textcolor{black}{
\label{arguments_our_th}
In the particular case of CML, the asymptotic law of our estimator has been stated under a set of regularity assumptions (particularly Assumption~\ref{assumption_regularity}) that significantly differs from those that have been proposed in the literature: see the seminal papers~\cite{tsukahara2005} (Assumption (A.1)) and~\cite{chen2005pseudo} (Assumptions (A.2)-(A.4)). 
These competing sets of assumption are apparently not nested, due to two very different techniques of proofs: a general integration by parts formula in our case, and some old results from~\cite{ruymgaart1972,ruymgaart1974} in the latter cases.
Our set of assumptions should most often be weaker. Indeed, we typically do not require that some expectations as $\Eb\big[ \{U_1(1-U_1)U_2(1-U_2)\}^{-a} \big]$ are finite for some $a>0$ (Assumption (A.3) in~\cite{chen2005pseudo}). 
Moreover and importantly, our results may directly be applied to time series, contrary to the cited papers that are strongly restricted to i.i.d. observations. }

\mds 

\textcolor{black}{
Our proofs rely on the weak convergence of multivariate rank-order statistics: see Theorem~\ref{Th_fondam_multiv_rank_stat} in the Appendix, that is an extension of Theorem 3.3 of~\cite{berghaus2017weak} in an arbitrary dimension. 
Note that some papers have already stated similar results, but under restrictive conditions: Theorem 6 in \cite{fermanian2002weak} assumed the existence of continuous partial derivatives of the true copula on the whole hypercube $[0,1]^d$ (contrary to our Assumption~\ref{cond_reg_copula}). Theorem 2.4 in \cite{ghoudi2004empirical} relies on numerous very technical assumptions
and it is unclear whether the latter result is weaker/stronger than ours. 
Nonetheless, it ``only'' states weak convergence in the space of cadlag functions equipped with the Skorohod topology.}

\mds 

\textcolor{black}{Applying Theorem~\ref{oracle_property} (ii) or its corollary, it is possible to approximate by plug-in the limiting law of $\sqrt{n}(\widehat{\theta}-\theta_0)_{I}$, for every fixed subset $I\subset \widehat\Ac$ and assuming $\widehat\Ac= \Ac$ (an event whose probability tends to one): just replace the unknown quantities with their empirical counterparts. This is obvious concerning the matrices $\mathbf{A}_n(\theta_0)$, $\mathbf{B}_n(\theta_0)$ and $\Hb$. Concerning the Gaussian random vector $\WW$, an estimator of its variance is proposed in Section \ref{asymptoticvarianceZ} of the Appendix.}

\section{Conditional copula models}
\label{section_cond_copulas}

\textcolor{black}{In several situations of interest}, multivariate models are defined through some conditioning laws.
In other words, there exist a random vector of covariates $\ZZ \in \Zc$, for some borelian subset $\Zc$ in $\Rb^m$, and we focus on the laws of $\XX$ given any value of $\ZZ$.  
The natural framework is given by conditional copulas and the associated Sklar's theorem (\cite{fermanian2012time}, e.g.).
Generally speaking, conditional copula models have to manage conditional marginal distributions on one side, and conditional copulas on the other side. 
In our semiparametric approach, we do not specify the former ones. Indeed, this would be a source of complexities due to the need of kernel smoothing or
other non parametric techniques (\cite{fermanian2012time,abegaz2012semiparametric}, among others).
Here, we make the following simplifying assumption.
\begin{assumption}\label{no_dep_margins_covariates}
For every $k\in \{1,\ldots,d\}$, the law of $X_k$ given $\ZZ=\zz$ does not depend on $\zz\in \Zc$.
\end{assumption}
In other words, the conditional margins and the unconditional ones are identical.
Even if the latter assumption may be considered as relatively strong, it is not implausible. For instance, in typical copula-GARCH models (\cite{chen2006estimation}),
the marginal dynamics are filtered in a first stage, and a parametric copula is then postulated between the residuals.
It is well-known that systemic risk measures strongly depend on the economic cycle. Thus, some macroeconomic explanatory variables may have a significant impact on the latter copula. But, concerning the marginal conditional distributions,  this effect could be hidden due to the first-order phenomenon of ``volatility clustering''.  

\mds 

Under Assumption~\ref{no_dep_margins_covariates}, our dependence model of interest will be related to the laws of $\UU$ given $\ZZ=\zz$, $\zz\in \Zc$, by keeping the same definition of $\UU$ as previously. 
Sparsity would then be related to the number of components of $\zz$ that are relevant 
to specify any copula of $\XX$ (or $\UU$, equivalently) given $\ZZ=\zz$. 
Typically, the latter copulas belong to a given parametric $d$-dimensional family
$\{C_\theta; \theta \in \Theta\subset \Rb^p\}$ and the parameter $\theta$ depends on the underlying covariate: 
given $\ZZ=\zz$, the law of $\UU$ is $C_{\theta(\zz;\beta_0)}$, for some known map $\theta:\Zc \times \Bc\rightarrow \Theta$, $\Bc\subset \Rb^q$.
The problem is now to evaluate the true ``new parameter'' $\beta_0$, based on a sample $(\XX_i,\ZZ_i)_{i=1,\ldots,n}$.
Compared to the previous sections, the focus has switched from $(\theta_0,\Theta,p)$ to $(\beta_0,\Bc,q)$.
In particular, we will assume that the new parameter set $\Bc$ satisfies Assumption~\ref{assump_theta0} instead of $\Theta$.

\mds

Under Assumption~\ref{no_dep_margins_covariates}, we define the same pseudo-observations as before, and keep the notation $\widehat\Uc_n$.
In addition, set $\Zc_n:=(\ZZ_1,\ldots,\ZZ_n)$. For example, the parameter $\beta_0$ may naturally be estimated without any penalty by CML as
\begin{equation}
\widetilde \beta \,\textcolor{black}{\in}\, \arg\max_{\beta \in \Bc} \;\Lb_n(\beta;\widehat\Uc_n,\Zc_n),\; \Lb_n(\beta;\widehat\Uc_n,\Zc_n) :=
\sum_{i=1}^n \ln c_{\theta(\ZZ_i;\beta)} \big(\widehat{U}_{i,1},\ldots,\widehat{U}_{i,d}\big).
\label{estim_condit_cop}    
\end{equation}

Under sparsity and with penalties, the results of Sections~\ref{framework} and~\ref{asym_prop} can be adapted to tackle this new problem and even more general ones.
First of all, we need to distinguish the cases of known and/or unknown covariate distributions.

\subsection{The marginal laws of the covariates are known.}\label{cond_copula_marginal_known}

Let us assume the law of $Z_k$ is known, continuous, and denoted as $F_{Z_k}$, $k\in \{1,\ldots,m\}$.
To simplify and w.l.o.g., we can additionally impose that the joint law of $(\UU,\ZZ)$ is a copula.
\begin{assumption}
\label{unif_margins_covariates}
The law of $Z_k$ is uniform between $0$ and $1$, for every $k\in \{1,\ldots,m\}$.
\end{assumption}
If this is not the case, just replace any $\zz$ with $\tilde \zz:=\big(F_{Z_1}(z_1),\ldots,F_{Z_m}(z_m) \big)$.
In the case of conditional copula models, the map $\theta$ would be replaced by $\tilde \theta(\tilde \zz;\beta):=\theta \big( \zz ;\beta\big)$.
To ease the notations, we will not distinguish between $(\theta,Z_k)$ and $(\tilde \theta,\tilde Z_k)$. Extending~(\ref{obj_crit}), we now consider the penalized estimator
\begin{equation} \label{obj_crit_extended}
\widehat{\beta} \,\textcolor{black}{\in}\, \underset{\beta \in \Bc}{\arg \; \min} \; \Big\{ \Lb_n(\beta;\widehat{\Uc}_n,\Zc_n)
+ n \overset{q}{\underset{k=1}{\sum}}\pp(\lambda_n,|\beta_k|)\Big\},
\end{equation}
where $\pp(\lambda_n,.): \Rb \rightarrow \Rb_+$ is a \textcolor{black}{penalty}. Assume the new empirical loss
$\Lb_n(\beta;\uu_1,\ldots,\uu_n,\zz_1,\ldots,\zz_n)$ is associated with a continuous
function $\ell : \Bc \times (0,1)^{d+m} \rightarrow \Rb$ so that we can write
 \begin{equation}
 \Lb_n(\beta;\uu_1,\ldots,\uu_n,\zz_1,\ldots,\zz_n)
 := \overset{n}{\underset{i=1}{\sum}} \ell(\beta;\uu_i;\zz_i).
\label{loss_as_a_sum_extended}
 \end{equation}

Since the margins of $\ZZ$ are uniform by assumption, the joint law of $(\UU,\ZZ)$ is a $d+m$-dimensional copula denoted $D$.
Instead of the empirical copula related to the $\XX_i$, we now focus on an empirical counterpart of the $(\XX,\ZZ)$ copula, i.e.
\begin{equation}
\widehat D_n(\uu,\zz) := \frac{1}{n}\overset{n}{\underset{i=1}{\sum}}\mathbf{1}\big\{F_{n,1}(X_{i,1})\leq u_1,\ldots,F_{n,d}(X_{i,d})\leq u_d,Z_{i,1}\leq z_1,\ldots,Z_{i,m}\leq z_m\big\},
\label{empir_copula_extended}
\end{equation}
for every $\uu\in [0,1]^d$ and $\zz\in [0,1]^m$.
The associated ``empirical copula'' process becomes $\widehat \Db_n :=\sqrt{n}(\widehat D_n - D)$.
Obviously, the weak behavior of $\widehat\Db_n$ is the same as the one of $\Db_n:=\sqrt{n}(D_n-D)$, where
$$ D_n(\uu,\zz) := \frac{1}{n}\overset{n}{\underset{i=1}{\sum}}
\mathbf{1}\big\{X_{i,1}\leq F_{n,1}^-(u_1),\ldots,X_{i,d}\leq F_{n,d}^-(u_d),Z_{i,1}\leq z_1,\ldots,Z_{i,m}\leq z_m\big\}.$$
See Appendix C in~\cite{radulovic2017weak}, for instance.

\mds 

We would like to state some versions of Theorems~\ref{bound_proba} and~\ref{oracle_property} for an estimator given by~(\ref{obj_crit_extended}). 
Obviously, 
$$ \Lb_n(\beta;\widehat{\Uc}_n,\Zc_n) = n\int  \ell(\beta;\uu;\zz)\, \widehat D_n(d\uu,d\zz),  $$
and the limiting law of $\widehat \beta$ will be deduced from the asymptotic behavior of $\widehat D_n$.
Broadly speaking, the two main components to state such results are an integration by parts formula and a
weak convergence result for $\widehat\Db_n$ (or $\Db_n$, equivalently).
The former tool will be guaranteed by applying our Theorem~\ref{Th_fondam_multiv_rank_stat} in the appendix.
And the latter weak convergence result, will be a consequence of the weak convergence of $(\alpha_n)$, 
now the empirical process associated with the sample $(\UU_i,\ZZ_i)_{i=1,\ldots,n}$:
\begin{equation}
 \alpha_n(\uu,\zz):=\sqrt{n}\Big\{ \frac{1}{n}\sum_{i=1}^n \1 \big( \UU_i \leq \uu,\ZZ_i\leq \zz \big) - D(\uu,\zz)\Big\}.
 \label{new_def_alphan}
 \end{equation}
Obviously, when our observations are i.i.d.,
the process $(\alpha_n)$
weakly tends to a $D$-Brownian bridge $\alpha_D$ in $\ell^\infty([0,1]^{d+m})$.
Instead of $\bar\Cb_n$ (see Theorem~\ref{Th_fondam_multiv_rank_stat} in the appendix), the approximated empirical copula process is here
$$\bar \Db_n(\uu,\zz):= \alpha_n(\uu,\zz) -
\sum_{k=1}^d \dot D_k (\uu,\zz)\alpha_n(\1_{-k}:u_k), \;\;(\uu,\zz)\in [0,1]^{d+m},$$
using the notations detailed in our appendix.
Note that the partial derivatives of the copula $D$ have to be considered w.r.t. the first $d$ components only, i.e. the components that correspond to pseudo-observations (and not $Z_k$-type components).
\mds

Now, let us state the new theoretical results related to semi-parametric inference in the presence of pseudo-observations and possible complete observations. 
Since they can be deduced from the previous sections and proofs, we omit the details.
\begin{assumption}\label{regularity_assumption_extended}
The map $\beta\mapsto \ell(\beta;\uu,\zz)$ is thrice differentiable on $\Bc$, for every $(\uu,\zz) \in (0,1)^{d+m}$.
Moreover, $\Hb := \Eb[\nabla^2_{\beta\beta^\top}\ell(\beta_0;\UU,\ZZ)]$ and $\Mb := \Eb[\nabla_{\beta}\ell(\beta_0;\UU,\ZZ) \nabla_{\beta^\top}\ell(\beta_0;\UU,\ZZ)]$ exist and are positive definite.
Finally, for \textcolor{black}{every} $\eps>0$, there exists a constant $K_\epsilon$ such that
$$ \sup_{\{\beta; \|\beta - \beta_0\|<\eps\}} \;\sup_{j,l,m}\big|\Eb[\partial^3_{\beta_j\beta_l\beta_m} \ell( \beta;\UU,\ZZ)] \big| \leq K_\epsilon.$$
\end{assumption}
\begin{assumption}
\label{cond_reg_copula_extended}
For any $1 \leq j\leq d$ and $\epsilon>0$, the copula partial derivative $\dot D_j(\uu,\zz):=\partial D(\uu,\zz)/\partial u_j$ exists and is continuous on $V_j:=\{\uu\in [0,1]^d; u_j\in [\epsilon,1-\epsilon] \}$ uniformly w.r.t. $\zz\in [0,1]^m$. For every couple $(j_1,j_2)\in \{1,\ldots,d\}^2$ and $\zz\in [0,1]^m$, the second-order partial derivative $\ddot D_{j_1,j_2}(\uu,\zz):=\partial^2 D(\uu,\zz)/\partial u_{j_1}\partial u_{j_2}$ exists and is continuous
on $V_{j_1}\cap V_{j_2}$. Moreover, there exists a positive constant $K$ such that
\begin{equation}
 \sup_{\zz\in [0,1]^m}\big|\ddot D_{j_1,j_2}(\uu,\zz) \big|\leq K \min \left( \frac{1}{u_{j_1}(1-u_{j_1})},  \frac{1}{u_{j_2}(1-u_{j_2})} \right),\; \uu\in V_{j_1}\cap V_{j_2}.
\label{cond_2nd_der_cop_extended}
\end{equation}
\end{assumption}

\begin{assumption}
\textcolor{black}{For some positive constants $\omega$ and $K$,} the family of maps $\Fc=\Fc_1\cup\Fc_2\cup\Fc_3$ from $(0,1)^{d+m}$ to $\Rb$ is $g_{\omega,d+m}$-regular (see Definition~\ref{def_gomega_reg} in Appendix \ref{technicalities}), with
$$\Fc_1:=\{ f:(\uu,\zz)\mapsto \partial_{\beta_k}\ell(\beta_0;\uu;\zz); k=1,\ldots,p\},$$
$$\Fc_2:=\{ f:(\uu,\zz)\mapsto \partial^2_{\beta_k,\beta_l}\ell(\beta_0;\uu;\zz); k,l=1,\ldots,p\},$$
$$\Fc_3:=\{ f:(\uu,\zz)\mapsto \partial^3_{\beta_k,\beta_l,\beta_j}\ell(\beta;\uu;\zz); k,l,j=1,\ldots,p, \; \|\beta-\beta_0\|<K\}.$$
\label{assumption_regularity_extended}
\end{assumption}

\begin{theorem}\label{bound_proba_extended}
\textcolor{black}{ 
Suppose Assumptions \ref{oscillation_modulus_assump}-\ref{tail_empir_processes} given in Appendix \ref{technicalities} are satisfied with the process $(\alpha_n)$ defined on $[0,1]^{d+m}$ by~(\ref{new_def_alphan}).
Let some $$\omega \in \Big(0,\min\big\{\frac{\textcolor{black}{\kappa}_1}{2(1-\textcolor{black}{\kappa}_1)},\frac{\textcolor{black}{\kappa}_2}{2(1-\textcolor{black}{\kappa}_2)},\textcolor{black}{\kappa}_3 - \frac{1}{2}\big\}\Big).$$
Suppose Assumptions~\ref{assumption_regularizer}-\ref{assumption_regularity_extended} hold for this $\omega$.}
Then, there exists a sequence $\widehat{\beta}$ defined in (\ref{obj_crit_extended}) that satisfies the bound
\begin{equation*}
\|\widehat{\beta}-\beta_0\|_2 = O_p\Big(\ln(\ln n) n^{-1/2} + a_{n}\Big).
\end{equation*}
\end{theorem}

\begin{theorem}\label{oracle_property_extended}
In addition to the assumptions of Theorem~\ref{bound_proba_extended},
assume that the \textcolor{black}{penalty function} satisfies $\underset{n \rightarrow \infty}{\lim \, \inf} \; \underset{x \rightarrow 0^+}{\lim \, \inf} \; \lambda^{-1}_n \partial_2\pp(\lambda_n,x) > 0$. Moreover, assume $\lambda_n \rightarrow 0$, $\sqrt{n} \lambda_n (\ln(\ln n))^{-1} \rightarrow \infty$ and $a_n=o(\lambda_n)$.
Then, the consistent estimator $\widehat{\beta}$ of Theorem~\ref{bound_proba_extended} satisfies the following properties.
\begin{itemize}
    \item[(i)] Support recovery: $\underset{n \rightarrow \infty}{\lim} \; \Pb(\widehat{\Ac} = \Ac) = 1$, where $\Ac$ and $\widehat{\Ac}$ are related to the support of $\beta_0$.
    \item[(ii)] Asymptotic normality:
    assume the empirical process $(\alpha_n)$ converges weakly in $\ell^\infty([0,1]^{d+m})$
    to a limit process $\alpha_D$ which has continuous sample paths, almost surely.
    If, in addition, the Assumption \ref{regularity_cond_wc} in Appendix \ref{technicalities} (applied to $\Fc_1$ instead of $\Fc$) is fulfilled and $\sqrt{n}\lambda_n^2=o(1)$, then
    \begin{equation*}
    \sqrt{n}\Big[\Hb_{\Ac\Ac}+\mathbf{B}_n(\beta_0)\Big]\Big\{\big(\widehat{\beta}-\beta_0\big)_{\Ac} + \big[\Hb_{\Ac\Ac}+\mathbf{B}_n(\beta_0)\big]^{-1}\mathbf{A}_n(\beta_0) \Big\} \overset{d}{\underset{n \rightarrow \infty}{\longrightarrow}} \WW,
    \end{equation*}
    where $\Hb_{\Ac\Ac} = \Big[ \Eb\big[  \partial^2_{\beta_k\beta_l}\ell(\beta_0;\UU;\ZZ)\big] \Big]_{k,l \in \Ac}$,
    $\normalfont\mathbf{A}_n(\beta_0) = \big[\partial_{2}\pp(\lambda_n,|\beta_{0,k}|)\text{sgn}(\beta_{0,k})\big]_{k \in \Ac}$, $\mathbf{B}_n(\beta) =\normalfont \text{diag}(\partial^2_{2,2}\pp(\lambda_n,|\beta_{0,k}|),\,k\in\Ac)$ and $\WW$ is a $|\Ac|$-dimensional centered random vector defined as
    \begin{eqnarray*}
    \lefteqn{
    W_j := (-1)^{d+m}\int_{(0,1)^{d+m}} \Db(\uu,\zz) \, \partial_{\beta_j}\ell(\beta_0;d\uu,d\zz)   }\\
    &+&
\sum_{ \substack{I \subset \{1,\ldots,d+m\} \\ I\neq \emptyset, I\neq \{1,\ldots,d+m\} } } (-1)^{|I|}
  \int_{(0,1)^{|I|}} \Db\big((\uu,\zz)_{I}:\1_{-I}\big) \, \partial_{\beta_j}\ell(\beta_0;d(\uu,\zz)_{I};\1_{-I}),
    \end{eqnarray*}
    for $j \in \Ac$, with $\Db(\uu,\zz)= \alpha_D(\uu,\zz) - \sum_{k=1}^d \dot D_k (\uu,\zz)\alpha_D(\1_{-k}:u_k)$, $(\uu,\zz)\in [0,1]^{d}\times \Zc$.
\end{itemize}
\end{theorem}

\begin{remark}
In Theorems~\ref{bound_proba_extended} and~\ref{oracle_property_extended}, it has been assumed that some $(\uu,\zz)$-maps are regular w.r.t. $g_{\omega,d+m}$ (Assumption~\ref{assumption_regularity_extended}). Checking the latter property with $g_{\omega,d+m}$ may sometimes be painful, or even impossible.
Beside, this task may have been done for the same parametric copula family in the (usual) unconditional case, using the weights 
$g_{\omega,d}:\uu\mapsto g_{\omega,d}(\uu)$.
Unfortunately, there is no order between $g_{\omega,d+m}(\uu,\zz)$ and $g_{\omega,d}(\uu)$ and the randomness of the covariates matters in the more general situation~(\ref{obj_crit_extended}).
Hopefully, when some regularity properties are available uniformly w.r.t. $\zz\in \Zc$, we can rely on $g_{\omega,d}$ instead of $g_{\omega,d+m}$ and checking regularity properties becomes simpler: see Section~\ref{appl_cond_cop_models} below.
\end{remark}

By a careful inspection of the proof of Proposition 3.1 in~\cite{segers2012asymptotics}, 
the weak convergence of $\Db_n$ can be easily stated under ``minimal assumptions'' in the i.i.d. case.
Since this result is new and of interest per se, it is now precisely stated.
\begin{theorem}
Assume the margins $F_1,\ldots,F_d$ are continuous. 
If, for any $j\in \{1,\ldots,d\}$, any $\zz\in \Zc$ and any $\epsilon>0$, the partial derivative $\dot D_j(\uu,\zz):=\partial D(\uu,\zz)/\partial u_j$ exists and is continuous on
$V_{j,\epsilon}:=\{\uu\in [0,1]^d; u_j\in [\epsilon,1-\epsilon] \}$ uniformly w.r.t. $\zz\in \Zc$, then 
$$\sup_{\uu\in [0,1]^d,\zz\in \Zc} \big| (\Db_n-\bar \Db_n)(\uu,\zz)\big|=o_P(1),$$
when $n\rightarrow \infty$. Moreover, $\Db_n$ weakly tends to $\Db$ in $\ell^\infty([0,1]^{d}\times \Zc)$.
\label{weak_conv_cop_process_extended}
\end{theorem}
Note that Theorem~\ref{weak_conv_cop_process_extended} does not require Assumption~\ref{unif_margins_covariates}, i.e. 
it applies with an arbitrary (possibly discrete) subset set $\Zc$, and even if the marginal laws of the covariates are discontinuous.

\subsection{The marginal laws of the covariates are unknown.}
In this case, the laws $F_{Z_k}$ are still continuous but unknown, and the covariates belong to some arbitrary subset $\Zc\in \Rb^m$.
Introduce the random variables $V_k:=F_{Z_k}(Z_k)$, $k\in \{1,\ldots,m\}$, that are uniformly distributed between zero and one.
Set $\VV:= (V_{1},\ldots,V_{m})$.
We can manage this situation when the loss function is a map of $(\UU,\VV)$, instead of $(\UU,\ZZ)$ as previously:
define pseudo-observations related to the covariates
 $\widehat{\VV}_i:= (\widehat{V}_{i,1},\ldots,\widehat{V}_{i,m})$, where $\widehat{V}_{i,k} := F_{n,Z_k}(Z_{i,k})$
for every $i\in\{1,\ldots,n\}$ and every $k\in\{1,\ldots,m\}$, using the $k$-th re-scaled empirical c.d.f.
$ F_{n,Z_k}(s):= (n+1)^{-1} \sum^n_{i=1} \mathbf{1}\{Z_{i,k} \leq s\}.$
The penalized estimator of interest is here defined as
\begin{equation*}
\label{obj_crit_extended_bis}
\overline{\beta} \,\textcolor{black}{\in}\, \underset{\beta \in \Bc}{\arg \; \min} \; \Big\{ \overset{n}{\underset{i=1}{\sum}} \ell(\beta;\widehat{\UU}_i;\widehat{\VV}_i)
+ n \overset{q}{\underset{k=1}{\sum}}\pp(\lambda_n,|\beta_k|)\Big\}.
\end{equation*}

Thus, we recover the standard situation that has been studied in Section~\ref{asym_prop}. 
All the results of Section~\ref{asym_prop} directly apply, replacing the $d$-dimensional copula $C$ by the $d+m$-dimensional copula $D$, replacing $\widehat{\Uc}_n$ by $\big( \widehat{\UU}_i,\widehat{\VV}_i\big)_{i=1,\ldots,n}$, etc. The limiting law of
$\big(\overline{\beta}-\beta_0\big)_{\Ac}$ will not be the same as in Theorem~\ref{oracle_property_extended} (ii). Indeed, the process $\Db$ has now to be replaced
by $\tilde\Db(\uu,\zz):= \Db(\uu,\zz) - \sum_{k=d+1}^{d+m} \dot D_k (\uu,\zz)\alpha_D(\1_{-k}:z_k)$, due to the additional amount of randomness induced by the ``pseudo-covariates'' $\widehat \VV_{i}$. 

\subsection{Practical considerations}
\label{appl_cond_cop_models}

Now, let us come back to the estimator given by~(\ref{estim_condit_cop}).
The regularity assumptions of 
Theorems~\ref{bound_proba_extended} and~\ref{oracle_property_extended} have to be checked on a case-by-case basis. 
Nonetheless, there are some situations where things become simpler. 
Indeed, assume 
\begin{enumerate}
    \item[(i)] the map $(\zz,\beta)\mapsto \theta(\zz;\beta)$ and all its partial derivatives (up to order $3$ and $m$ w.r.t. the 
    components of $\beta$ and $\zz$ respectively) are
    bounded in $\Vc(\beta_0) \times \Zc$, denoting by $\Vc(\beta_0)$ a neighborhood of the true parameter $\beta_0$ in the space $\Bc$;
    \item[(ii)] the conditions of Theorems~\ref{bound_proba} and~\ref{oracle_property} are satisfied for the CML loss function $\ell(\theta;\uu):=-\ln c_\theta (\uu)$.
    \item[(iii)] the latter conditions may be verified replacing the weight function $g_{\omega,d}(\uu)$ by $\min_j(u_j)$.
\end{enumerate}
Note that, under (i) and (ii), the conditions of Theorems~\ref{bound_proba} and~\ref{oracle_property} are satisfied with the derivatives of the loss functions $(\uu,\beta) \mapsto \ln c_{\theta(\zz;\beta)} (\uu )$, for any fixed $\zz\in [0,1]^m$.
Then, it can be easily checked that Theorems~\ref{bound_proba_extended} and~\ref{oracle_property_extended} apply. 
In particular, the influence of the covariates is ``neutralized'' through (i); moreover, noting that $g_{\omega,d+m}(\uu,\zz)\leq \min_j(u_j)$, (iii) is sufficient to manage the weight functions. 

\mds 

To illustrate, consider the case of Gumbel and Clayton copulas, for which $\beta$ is $(1,+\infty)$ and $\Rb_+^*$ respectively 
(under the usual parameterization of~\cite{nelsen2006introduction}). 
Assume that, given $\ZZ=\zz$, the latter parameters can be rewritten $\theta(\zz;\beta)$ for some $\beta$ in $\Bc$ that satisfies Assumption~\ref{assump_theta0}.
Moreover, assume the ranges of the maps $(\zz,\beta)\mapsto \theta(\zz;\beta)$ from $[0,1]^m \times \Bc$ to $\Theta$
are included into a compact subset of $\Rb$. Then, Theorems~\ref{bound_proba_extended} and~\ref{oracle_property_extended} apply: the associated parameters $\tilde\beta$ defined in~\ref{estim_condit_cop} are consistent and weakly convergent.
See Sections \ref{verif_regul_Gumbel_cop} and \ref{verif_regul_Clayton_cop} of the Appendix for the technical details and for the proof that (i), (ii) and (iii) are indeed satisfied. This justifies the application of our results in the case of single-index models with Clayton/Gumbel copulas (Section~\ref{simulations} below).

\section{Applications}\label{applications}

\subsection{Examples}

\subsubsection{M-criterion for Gaussian copulas}
\label{Mcriterion_Gaussian_cop}
An important application of the latter results is Maximum Likelihood Estimation with pseudo-observations, where we observe a sample $\mathbf{X}_1,\ldots,\mathbf{X}_n$ from a $d$-dimensional distribution whose parametric copula depends on some parameter $\theta \in \Rb^p$. Equipped with pseudo-observations and using the same notations as above, our penalized estimator is defined as
\begin{equation*}
\widehat{\theta} \,\textcolor{black}{\in}\, \underset{\theta \in \Theta}{\arg \; \min} \; \Big\{ -\overset{n}{\underset{i=1}{\sum}} \ln c\big(\widehat{\UU}_{i};\theta\big) + n \overset{p}{\underset{k=1}{\sum}} \pp(\lambda_n,|\theta_k|)\Big\},
\end{equation*}
denoting by $c(\cdot,\theta)$ the copula density.
In the case of Gaussian copula model, the parameter of interest is $\theta = \text{vech}(\Sigma) \in \Rb^{d(d-1)/2}$, where $\Sigma$ is the correlation matrix of such a copula. This yields
\begin{equation}
\widehat{\theta} \,\textcolor{black}{\in}\, \underset{\theta \in \Theta}{\arg \; \min} \; \Big\{ \overset{n}{\underset{i=1}{\sum}} \ell(\theta;\widehat{\UU}_{i}) + n \overset{p}{\underset{k=1}{\sum}} \pp(\lambda_n,|\theta_k|)\Big\}, \; \ell(\theta;\widehat{\UU}_{i}) = \frac{1}{2}\ln(|\Sigma|) + \frac{1}{2}\mathbf{Z}^\top_{ni} \Sigma^{-1}\mathbf{Z}_{ni},
\label{estim_theta_hat_CG}
\end{equation}
with $\mathbf{Z}_{ni} = \big(\Phi^{-1}(\widehat{U}_{i1}),\ldots,\Phi^{-1}(\widehat{U}_{id})\big)^\top$, $i\in\{1,\ldots,n\}$.
Note that the $\mathbf{Z}_{ni}$ are approximated realizations of the Gaussian random vector 
$\mathbf{Z}: = (\Phi^{-1}(U_{1}),\ldots,\Phi^{-1}(U_{d}))^\top$.
\textcolor{black}{The Gaussian copula exhibits discontinuous partial derivatives at the boundary of $[0,1]^d$: see ~\cite{segers2012asymptotics}, Example 5.1.}
We have seen in Section~\ref{asym_prop} that $\widehat{\theta}$ is asymptotically normally distributed
under suitable regularity conditions. 
In Section \ref{verif_regul_Gaussian_cop} of the Appendix, we check that all the conditions are satisfied so that Theorems~\ref{bound_proba} and~\ref{oracle_property} can be applied to Gaussian copula models and $\widehat{\theta}$ given by~(\ref{estim_theta_hat_CG}). Interestingly, the associated limiting law in Theorem~\ref{oracle_property} is simply 
$ \WW := (-1)^d\int_{(0,1)^d} \Cb(\uu) \, \nabla_{\theta}\ell(\theta_0;d\uu)$.

\mds 

\textcolor{black}{The estimation of $\Sigma$ can be carried out using the least squares loss $\ell_{\text{LS}}(\theta;\widehat{\UU}_i) := \text{tr}((\mathbf{Z}_{ni}\mathbf{Z}_{ni}^\top - \Sigma)^2)$. In Section \ref{verif_regul_Gaussian_cop} of the Appendix, we verify that the latter loss satisfies the regularity conditions of Theorems~\ref{bound_proba} and~\ref{oracle_property}. Our simulation experiments on sparse Gaussian copula will be based on both the Gaussian loss and the least squares loss. Set $\widehat{S} := n^{-1}\sum^n_{i=1}\mathbf{Z}_{ni}\mathbf{Z}_{ni}^\top$, that approximates $\Sigma$. 
Interestingly, our empirical loss is equal to $n\ln(|\Sigma|)/2+ n\text{tr}(\Sigma^{-1}\widehat{S})/2$ and $n\;\text{tr}((\widehat{S} - \Sigma)^2)$ for the Gaussian CML and least squares cases respectively, apart from some constant terms that do not depend on $\Sigma$. Indeed, for the least squares loss, we have
$$ n\;\text{tr}((\widehat{S} - \Sigma)^2) 
= n \; \text{tr}(\widehat{S}^\top \widehat{S})-\text{tr}(\sum_{i=1}^n\mathbf{Z}_{ni}\mathbf{Z}_{ni}^\top \Sigma-\Sigma^\top \sum_{i=1}^n\mathbf{Z}_{ni}\mathbf{Z}_{ni}^\top)+n\;\text{tr}(\Sigma^\top\Sigma),$$
which is equal to $\sum^n_{i=1}\text{tr}((\mathbf{Z}_{ni}\mathbf{Z}_{ni}^\top - \Sigma)^2)$ plus 
some constant terms that do not depend on $\Sigma$.
In our simulation experiment for sparse Gaussian copulas, the implemented code relies on 
$\ln(|\Sigma|)+\text{tr}(\Sigma^{-1}\widehat{S})$ and/or $\text{tr}((\widehat{S} - \Sigma)^2)$ intensively, 
some quantities that can be quickly calculated through some matrix manipulations, even when $n>>1$.
}

\subsubsection{M-criterion for mixtures of copulas}

Mixing parametric copulas is a flexible way to build richly parameterized copulas. More precisely, a mixture based copula $C$ is usually specified by its density $c(\uu) = \sum^q_{k=1}\pi_k c_k(\uu;\gamma_k)$, built from the family of copula densities  $\{c_k(\uu;\gamma_k), k=1,\ldots,q\}$. Each of the copula density $c_k(\uu;\cdot)$ depends on a vector of parameters $\gamma_k\in \Theta_k$, and $(\pi_k)_{k=1,\ldots,q}$ are some unknown weights satisfying $\pi_k\in [0,1]$, $k\in\{1,\ldots,q\}$, and
$\sum^q_{k=1}\pi_k = 1$. The parameter of interest is $\theta = \big(\pi_1,\ldots,\pi_{q-1},\gamma^\top_1,\ldots,\gamma^\top_q\big)^\top \in \Theta$, with \textcolor{black}{$\Theta:=\Theta_{\text{mixt},q-1}\times \Theta_1\times \cdots \Theta_q$, 
with the notations of Example~\ref{ex_copula_mixture}}. Let $p$ be the dimension of any parameter $\theta$.  
Then, with our CML criterion with pseudo-observations, an estimator of the true $\theta_0$ is defined as
\begin{equation*}
\widehat{\theta}_{\text{mixt}} \,\textcolor{black}{\in} \,\underset{\theta \in \Theta}{\arg \; \min} \; \Big\{ \overset{n}{\underset{i=1}{\sum}} \ell(\theta;\widehat{\UU}_{i}) + n \overset{p}{\underset{k=1}{\sum}} \pp(\lambda_n,|\theta_k|)\Big\}, \; \ell(\theta;\widehat{\UU}_{i}) = -\ln\big(\overset{q}{\underset{k=1}{\sum}}\pi_kc_k(\widehat{U}_{i1},\ldots,\widehat{U}_{id};\gamma_{k})\big).
\end{equation*}
Such a procedure fosters sparsity among $(\pi_k)_{k=1,\ldots,q}$ and among $\gamma_1,\ldots,\gamma_q$: when $\widehat{\pi}_k \neq 0$, then the corresponding copula parameter $\widehat{\gamma}_k$ can potentially be sparse. The latter criterion is similar to Criterion (3) in~\cite{cai2014mixture}; however, these authors treat the marginals as known quantities, which significantly simplifies their large sample analysis.

\mds

\textcolor{black}{Assume that all parametric copula families $\{C_k(\uu;\gamma_k), \gamma_k\in \Theta_k\}$, $k\in \{1,\ldots,q\}$, satisfy the regularity 
conditions to apply our Theorems~\ref{bound_proba} and~\ref{oracle_property}. Unfortunately, in general, this does not imply their mixture model will satisfy all these properties, in particular the $g_\omega$-regularity. Therefore, it will be necessary to do this task on a case-by-case basis.}
\textcolor{black}{
Nonetheless, in the particular case of mixtures of Gaussian copulas, our regularity conditions are satisfied by considering the least squares loss function, as in Section~\ref{Mcriterion_Gaussian_cop}. 
Indeed, for this model, the variance-covariance matrix of $\ZZ$ is 
$\Sigma:=\sum_{k=1}^q \pi_k \Sigma_k$, denoting by $\Sigma_k$ the correlation matrix of parameters that is associated with $c_k$, $k\in \{1,\ldots,q\}$. Thus, the same arguments as for the Gaussian copula with the $\ell_{LS}$ loss can easily be invoked. Alternatively, choosing the log-likelihood (CML) loss induces more difficulties. Nonetheless, it can be proved that our regularity conditions apply, at least when all the (true) weights are strictly positive.
See Section \ref{verif_regul_Gaussian_cop} of the Appendix for details.
}

\subsection{Simulated experiments}\label{simulations}

\textcolor{black}{In this section, we assess the finite sample performances of our penalization procedure in the presence of pseudo-observations. To do so, we carry out simulated experiments for the sparse Gaussian copula model and some sparse conditional copulas. These experiments are meant to illustrate the ability of the penalization procedure to correctly identify the zero entries of the copula parameter with non-parametric marginals. First, let us briefly discuss the implementation procedure and the choice of $\lambda_n$.}

\subsubsection{Implementation and selection of $\lambda_n$}\label{cv_procedure}

\textcolor{black}{All the experiments are implemented in Matlab code and run on a Mac-OS Apple M1 Ultra with 20 cores and 128 GB Memory. A gradient descent type algorithm is implemented to solve the penalized Gaussian copula problem, a situation where closed-form gradient formulas can directly be applied. We employed the numerical optimization routine \emph{fmincon} of Matlab to find the estimated parameter for sparse conditional copulas \footnote{\textcolor{black}{The code for replication is available at \url{https://github.com/Benjamin-Poignard/sparse-copula}}}. The tuning parameter $\lambda_n$ controls the model complexity and must be calibrated for each penalty function. To do so, we employ a $5$-fold cross-validation procedure, in the same spirit as in Section 7.10 of \cite{Hastie2009}. To be specific, we divide the data into $5$ disjoint subgroups of roughly the same size, the so-called folds. Denote the indices of that observations that belong to the $k$-th fold by $T_k$, $k\in \{1,\ldots,5 \}$ and the size of the $k$-th fold by $n_k$. The $5$-fold cross-validation score is defined as
\begin{equation}\label{CV_score}
\text{CV}(\lambda_n) := \overset{5}{\underset{k=1}{\sum}}\Big\{\underset{i \in T_k}{\sum}\ell(\widehat{\theta}_{-k}(\lambda_n);\widehat{U}_i)\Big\},
\end{equation}
where $\sum_{i \in T_k}\ell(\widehat{\theta}_{-k}(\lambda_n);\widehat{U}_i)$ is the non-penalized loss associated to the copula model and evaluated over the $k$-th fold $T_k$ of size $n_k$, which serves as the test set, and $\widehat{\theta}_{-k}(\lambda_n)$ is our penalized estimator of the latter copula parameter based on the sample $(\cup^5_{j=1}T_j)\setminus T_k$ - the training set - using $\lambda_n$ as the tuning parameter. The optimal tuning parameter $\lambda^\ast_n$ is then selected according to $\lambda^\ast_n \,\textcolor{black}{\in}\, \arg \; \min_{\lambda_n}\; \text{CV}(\lambda_n)$. Then, $\lambda^\ast_n$ is used to obtain the final estimate of $\theta_0$ over the whole sample. Here, the minimization of the cross-validation score is performed over $\{c\sqrt{\log(\text{dim})/n}\}$, where $c$ is a grid (its size is user-specified) of $91$ values set as $0.01,0.05,0.1,0.15,\ldots,4.5$ and $\text{dim}$ the number of parameters to estimate. The choice of the rate $\sqrt{\log(\text{dim})/n}$ is standard in the sparse literature for M-estimators: see, e.g., Chapter 6 of \cite{Buhlmann2011} for the LASSO and \cite{loh2017} for non-convex penalization methods. }

\subsubsection{Sparse Gaussian copula models}\label{gaussian_empiric}

\textcolor{black}{Our first application concerns the Gaussian copula. Here, sparsity is specified with respect to the variance-covariance matrix $\Sigma \in \Rb^{d\times d}$. Its diagonal elements are equal to one and its off-diagonal elements are $\theta_{kl} \in (-1,1), 1 \leq k,l\leq d, k \neq l$, so that the number of distinct free parameters is $d(d-1)/2$. The parameter $\theta$ is defined as the column vector of the $\Sigma$ components located strictly below the main diagonal. Thus, $\Sigma$ can be considered as a function of $\theta$: $\Sigma=\Sigma(\theta)$. We still denote by $\Ac = \{k: \theta_{0,k} \neq 0, k=1,\ldots,d(d-1)/2\}$ the true sparse support, where $\theta_0$ is sparse when some components of $\UU$ are independent. Our simulated experiment can be summarized as follows: we simulate a sparse true $\theta_0$, generate the $\UU_i$ from the corresponding Gaussian copula with parameter $\theta_0$ for a given sample size $n$, and calculate $\widehat{\theta}$ by minimizing our penalized procedure based on the pseudo-sample; this procedure is repeated for two hundred independent batches.} 

\mds

\textcolor{black}{To be more specific, a sparse $\theta_0$ is randomly drawn for each batch as detailed below. 
Then, we generate a sample of $n$ vectors $\UU_i$ as follows: we draw $\UU_i=(U_{i,1},\ldots,U_{i,d})$, $i\in\{ 1,\ldots,n\}$, from a Gaussian copula with parameter $\theta_0$; then we consider their rank-based transformation to obtain a non-parametric estimator of their marginal distribution, providing the pseudo-observations $\widehat{\UU}_i = (\widehat{U}_{i,1},\ldots,\widehat{U}_{i,d})$ that enter the loss function.
Then, we solve~(\ref{obj_crit}).
The non-penalized loss is the Gaussian log-likelihood, as defined in~(\ref{estim_theta_hat_CG}).
Alternatively, \textcolor{black}{in~(\ref{obj_crit})}, we consider the least squares criterion for which 
$\ell(\theta;\widehat{U}_i)=\text{tr}((\ZZ_{ni}\ZZ^\top_{ni}-\Sigma)^2)$, where $\ZZ_{ni} := (\Phi^{-1}(\widehat{U}_{i,1}),\ldots,\Phi^{-1}(\widehat{U}_{i,d}))^\top$. \textcolor{black}{In both cases, the penalized problem is solved by a gradient descent algorithm based on the updating formulas of Section 4.2 in~\cite{loh2015}, where the initial value is set as $\widehat{S} := n^{-1}\sum^n_{i=1}\mathbf{Z}_{ni}\mathbf{Z}_{ni}^\top$}. The score for cross-validation purpose is defined in (\ref{CV_score}) with the Gaussian loss or the least squares loss. 
Concerning $\widehat{\theta}$, we apply the SCAD, MCP and LASSO penalties. The non-convex SCAD and MCP ones require the calibration of $a_{\text{scad}}$ and $b_{\text{mcp}}$, respectively. We select $a_{\text{scad}}=3.7$, a ``reference'' value identified as optimal in \cite{fan2001variable} by cross-validated experiments. In the MCP case, the ``reference'' parameter is set as $b_{\text{mcp}}=3.5$, following \cite{loh2015}. 
We investigate the sensitivity of these non-convex procedures with respect to their parameters $a_{\text{scad}}, b_{\text{mcp}}$. 
In particular, our results are also detailed with the values $a_{\text{scad}}=40$ and $b_{\text{mcp}}=40$. 
This case corresponds to ``large'' $a_{\text{scad}}$ and $b_{\text{mcp}}$, for which the corresponding \textcolor{black}{penalty functions} tend to the LASSO \textcolor{black}{penalty}.
}

\mds

\textcolor{black}{We consider the dimensions $d\in\{10, 20\}$, so that the dimension of $\theta_0$ is $p=d(d-1)/2 = 45$ and $190$, respectively. The cardinality of the true support $\Ac$ is set arbitrarily as $|\Ac|=7$ (resp. $|\Ac| = 19$) when $d=10$ (resp. $d=20$), so that the percentage of zero coefficients of $\theta_0$ is approximately $85\%$ (resp. $90\%$). As for the non-zero coefficients of $\theta_0$, for each batch, they are generated from the uniform distribution $\Uc([-0.7,-0.05]\cup[0.05,0.7])$, thus ensuring the minimum signal strength $\min_{k\in \Ac}|\theta_{0,k}|\geq 0.05$. As for the sample size, we consider $n\in\{500,1000\}$. Note that, for each batch, the number of zero coefficients of $\theta_0$ remains unchanged but their locations may be different.}

\mds

We report the variable selection performance through the percentage of zero coefficients in $\theta_0$ that are correctly estimated, denoted by $\text{C1}$, and the percentage of non-zero coefficients in $\theta_0$ correctly identified as such, denoted by $\text{C2}$. The mean squared error (MSE), defined as $\|\widehat\theta-\theta_0\|^2_2$, is reported as an estimation accuracy measure. These metrics are averaged over the two hundred batches and reported in Table~\ref{support_gaussian_copula}. For clarifying the reading of the figures, the first entry \textcolor{black}{$84.70$} in the column ``Gaussian'' represents the percentage of the true zero coefficients correctly identified by the estimator deduced from the Gaussian loss function, with SCAD penalization when $a_{\text{scad}}=3.7$, for a sample size $n=500$, a dimension $d=10$, and averaged over two hundred batches; in the last MSE line, the value \textcolor{black}{$0.0625$} in the column ``Least Squares'' represents the MSE of the estimator deduced from the least squares loss function with MCP penalization when $b_{\text{mcp}}=40$, for a sample size $n=1000$ and $d=20$, \textcolor{black}{and averaged over two hundred batches}. Our results highlight that, for a given loss, the SCAD/MCP-based penalization procedures provide better results in terms of support recovery compared to the LASSO for our reference values of $a_{\text{scad}}, b_{\text{mcp}}$.
Furthermore, the SCAD and even more the MCP-based estimators with the Gaussian loss provide better recovery performances compared to the least squares loss. This is particularly true with the indicator $\text{C1}$, i.e., for the sake of identifying the zero coefficients. Interestingly, large $a_{\text{scad}},b_{\text{mcp}}$ values worsen the recovery ability. Indeed, such large values result in a LASSO-type behavior, which is biased so that small $\lambda_n$ will tend to be selected. 
Moreover, for any given penalty function, the Gaussian loss-based MSE's are always lower than the least squares loss-based MSE's, which suggest that the estimator deduced from the former loss is more efficient than the estimator obtained from the latter loss. Furthermore, larger $a_{\text{scad}},b_{\text{mcp}}$ values worsen the MSE performances: when $a_{\text{scad}}=b_{\text{mcp}}=40$, for a given loss function, the MSE's of SCAD/MCP are close to the LASSO. In Section \ref{additional_experiment_gaussian_cop} of the Appendix, we investigate further the sensitivity of the performances of the SCAD and MCP-based estimators with respect to $a_{\text{scad}}$ and $b_{\text{mcp}}$. 
    
\begin{table}[h!]\centering\caption{Model selection and accuracy, based on 200 replications. For each penalized loss, the results are reported according to the order SCAD, MCP and then LASSO. $\text{C1},\text{C2}$ are expressed in percentage, and larger numbers are better; for each MSE metric, smaller
numbers are better.\label{support_gaussian_copula}}
{\color{black}\scalebox{0.95}{\begin{tabular}{c c c c c}\hline \hline $(n,d,a_{\text{scad}},b_{\text{mcp}})$ &  & Truth & Gaussian & Least Squares \\
\hline\hline
& & & & \\

$(500,10,3.7,3.5)$ & $\text{C1}$ & $100$ &  $ 84.70 \,-\, 89.97 \,-\, 78.36 $ & $ 80.22 \,-\, 86.26 \,-\, 73.76 $ \\
 & $\text{C2}$ & $100$ & $ 96.14 \,-\, 94.64 \,-\, 97.07 $  & $ 95.57 \,-\, 94.43 \,-\, 96.07 $\\
 & MSE & & $ 0.0190  \,-\, 0.0189  \,-\, 0.0238 $  &  $ 0.0293 \,-\, 0.0281 \,-\, 0.0421 $ \\

& & & & \\

$(1000,10,3.7,3.5)$ & $\text{C1}$ & $100$ &  $ 86.55 \,-\, 91.29 \,-\, 77.53 $ &  $ 80.37 \,-\, 86.90 \,-\, 69.17 $ \\
 & $\text{C2}$ & $100$ & $ 97.93 \,-\,  97.50  \,-\, 98.64 $ &  $ 98.14 \,-\, 97.64 \,-\, 98.71 $  \\
 & MSE & & $ 0.0081 \,-\, 0.0079 \,-\, 0.0117 $ & $ 0.0131 \,-\, 0.0124 \,-\, 0.0213 $  \\

& & & & \\
    
$(500,10,40,40)$ & $\text{C1}$ & $100$ & $ 79.54 \,-\, 80.01 \,-\, 78.36 $ &  $ 74.30 \,-\, 75.34 \,-\, 73.76 $ \\
 & $\text{C2}$ & $100$ & $ 97.00 \,-\, 96.79 \,-\, 97.07 $  &  $ 96.00 \,-\, 95.93 \,-\, 96.07 $ \\
 & MSE & & $ 0.0228 \,-\, 0.0227 \,-\, 0.0238 $ & $ 0.0405 \,-\, 0.0403 \,-\, 0.0421 $ \\

& & & & \\

$(1000,10,40,40)$ & $\text{C1}$ & $100$ & $ 78.97  \,-\, 79.86 \,-\, 77.53 $ & $70.49 \,-\, 71.54 \,-\, 69.17 $  \\
 & $\text{C2}$ & $100$ & $ 98.64 \,-\, 98.57 \,-\, 98.64 $ & $ 98.64 \,-\, 98.57 \,-\, 98.71 $  \\
 & MSE & & $ 0.0110 \,-\, 0.0110 \,-\,  0.0117 $  &  $ 0.0199 \,-\, 0.0198 \,-\, 0.0213 $ \\

& & & & \\

\hline
& & & & \\
    
$(500,20,3.7,3.5)$ & $\text{C1}$ & $100$ & $ 88.08 \,-\, 92.71 \,-\, 85.63 $ & $ 84.13 \,-\, 89.91 \,-\, 81.14 $  \\
 & $\text{C2}$ & $100$ &  $ 94.55 \,-\, 92.84 \,-\, 95.40 $  & $ 94.05 \,-\, 92.55 \,-\, 94.82 $ \\
 & MSE & & $ 0.0602 \,-\, 0.0571 \,-\, 0.0690 $ &  $ 0.0937 \,-\, 0.0884  \,-\, 0.1272 $\\
 
 & & & & \\
    
 $(1000,20,3.7,3.5)$ & $\text{C1}$ & $100$ &  $ 89.28 \,-\, 93.46 \,-\, 84.65 $ & $ 86.16 \,-\, 91.77 \,-\, 80.33 $  \\
 & $\text{C2}$ & $100$ & $ 97.92 \,-\,  97.08 \,-\, 98.21 $  & $ 97.11 \,-\, 96.13 \,-\, 97.74 $ \\
 & MSE & & $ 0.0252 \,-\, 0.0239 \,-\, 0.0337 $ & $ 0.0433 \,-\, 0.0406 \,-\, 0.0664 $ \\
 
 & & & & \\

 $(500,20,40,40)$ & $\text{C1}$ & $100$ &  $85.85 \,-\,  86.51 \,-\, 85.63 $ & $ 81.29 \,-\, 82.14 \,-\, 81.14 $ \\
 & $\text{C2}$ & $100$ & $ 95.37 \,-\, 95.26 \,-\, 95.40 $  & $ 94.79 \,-\, 94.58 \,-\, 94.82 $  \\
 & MSE & & $ 0.0673 \,-\,  0.0668 \,-\, 0.0690 $ & $ 0.1234 \,-\, 0.1224 \,-\, 0.1272 $ \\
 
 & & & & \\
    
 $(1000,20,40,40)$ & $\text{C1}$ & $100$ &  $ 85.51  \,-\, 86.11 \,-\, 84.65 $ &  $ 80.73 \,-\, 81.74 \,-\, 80.33 $\\
 & $\text{C2}$ & $100$ & $ 98.08 \,-\, 98.05 \,-\, 98.21 $ & $ 97.74 \,-\, 97.55 \,-\, 97.74 $ \\
 & MSE & & $ 0.0322 \,-\, 0.0319 \,-\, 0.0337 $ & $ 0.0631 \,-\, 0.0625 \,-\, 0.0664 $  \\
 
 & & & & \\
 
\hline
\hline\end{tabular}}}
\end{table}

\subsubsection{Conditional copulas}\label{cond_empiric}

\textcolor{black}{Our next application is dedicated to the sparse estimation of conditional copula models with known link functions and known marginal laws of the covariates: the experiment is an application of the penalized problem detailed in Subsection \ref{cond_copula_marginal_known}. We specify the law of $\XX \in \Rb^d$, given some covariates $\ZZ \in \Rb^q$, as a parametric copula with parameter $\theta(\ZZ^\top\beta)$, where $\ZZ \in \Rb^q$ and $\beta \in \Rb^q$ ($p=1$ and $m=q$, with our notations of Section~\ref{section_cond_copulas}). 
We assume the marginal distribution of $X_k$, $k\in \{1,\ldots,d\}$, is unknown and does not depend on $\ZZ$. We focus on the Clayton and Gumbel copulas, and restrict ourselves to $d=2$.}

\mds

\textcolor{black}{In the same vein as in the previous application to sparse Gaussian copulas, for each sample size $n$, we draw two hundred independent batches of $n$ vectors $\UU_i$ as follows: in each batch, we simulate a sparse true $\beta_0$; then for every $i\in\{1,\ldots,n\}$, we draw the covariates $Z_{i,k}$, $ k\in\{1,\ldots,q\}$, from a uniform distribution $\Uc([0,1])$, independently of each other; then, for a given $\ZZ_i=(Z_{i,1},\ldots,Z_{i,q})^\top$, we sample $\UU_i=(U_{i,1},U_{i,2})$ from the Clayton/Gumbel copula with parameter $\theta_i:=\theta(\ZZ^\top_i\beta_0)$; we consider their rank-based transformation to obtain the pseudo-observations $\widehat{\UU}_i = (\widehat{U}_{i,1},\widehat{U}_{i,2})$, which are plugged in the penalized criterion. Here, the copula parameters $\theta_i$ are specified in terms of Kendall's tau: for each $i$, define the Kendall's tau $\tau_i:=2\arctan(\ZZ^\top_i\beta_0)/\pi$. Using the mappings of, e.g., \cite{nelsen2006introduction}, set $\theta_i = 2\tau_i/(1-\tau_i)$ for the Clayton copula and $\theta_i = 1/(1-\tau_i)$ for the Gumbel copula. We consider the dimension $q=30$, and set the cardinality of the true support $\Ac = \{k:\beta_{0,k} \neq 0, k = 1,\cdots,q\}$ as $|\Ac| = 3$, so that approximately $90\%$ of the entries of $\beta_0$ are zero coefficients. For each batch, the non-zero entries are simulated from the uniform distribution $\Uc([0.05,1])$, which ensures that the following copula parameter constraints are satisfied: $\theta_i > 0$ (resp. $\theta_i > 1$) for the Clayton (resp. Gumbel) copula. For each batch, the locations of the zero/non-zero entries in $\beta_0$ are arbitrary, but the size of $\Ac$ remains fixed. Finally, we consider the sample size $n \in \{500, 1000\}$. For a given batch, our criterion becomes:
\begin{equation*}
\widehat{\beta} \,\textcolor{black}{\in}\, \underset{\beta \in \Theta}{\arg \; \min} \; \Big\{ \Lb_n(\theta;\widehat{\Uc}_n,\mathcal{Z}_n) + n \overset{q}{\underset{k=1}{\sum}}\pp(\lambda_n,|\theta_k|)\Big\},\;\text{with}\; \Lb_n(\theta;\widehat{\Uc}_n,\mathcal{Z}_n) = -\overset{n}{\underset{i=1}{\sum}}\ln c_{\theta(\ZZ^\top_i\beta)} \big(\widehat{U}_{i,1},\widehat{U}_{i,2} \big),
\end{equation*}
where $\ln c_{\theta(\ZZ^\top_i\beta)}(.)$ is the log-density of the Clayton/Gumbel copula with parameter $\theta(\ZZ^\top_i\beta)$. As for the penalty function, we consider the SCAD, MCP and LASSO penalty functions to estimate $\Ac$ and $\beta$. Moreover, we choose $a_{\text{scad}} \in \{3.7,10,20,40,70\}$ and $b_{\text{mcp}}\in\{3.5,10,20,40,70\}$. To assess the finite sample performance of the penalization methods and as in Section~\ref{gaussian_empiric}, we report in Table~\ref{support_conditional_Gumbel_Clayton_copula} the percentage of zero coefficients correctly estimated ($\text{C1}$), the percentage of non-zero coefficients correctly identified ($\text{C2}$) and the mean squared error (MSE), averaged over the two hundred batches. For both models and low/large sample sizes, our results emphasize the poor performances of the LASSO penalization in terms of support recovery (correct identification of the zero coefficients). As for non-convex penalization, \textcolor{black}{the trade-off between $\text{C1}$ and $\text{C2}$ is more indicative than in the application to the Gaussian copula}: small $a_{\text{scad}}, b_{\text{mcp}}$ provide better $\text{C1}$ to the detriment of $\text{C2}$, which results in larger MSE since $\text{C2}$ worsens in that case. The MSE results are significantly improved for $n$ large. Mid-range $a_{\text{scad}}, b_{\text{mcp}}$ values provide an optimal trade-off in terms of the combined $\text{C1}$, $\text{C2}$ and MSE metrics.}

\begin{table}[h!]\centering\caption{Model selection and accuracy based on 200 replications. $\text{C1},\text{C2}$ are expressed in percentage, and larger numbers are better; for each MSE metric, smaller numbers are better.} \label{support_conditional_Gumbel_Clayton_copula}
{\color{black}\scalebox{0.7}{\begin{tabular}{c | c | c | c | c}\hline \hline
 & \, Gumbel $n=500$ \, & Gumbel $n=1000$ & \, Clayton $n=500$ \, & Clayton $n=1000$ \\
Penalty & $\text{C1}$ \qquad \;$\text{C2}$\; \qquad MSE & $\text{C1}$ \qquad \;$\text{C2}$\; \qquad MSE & $\text{C1}$ \qquad \;$\text{C2}$\; \qquad MSE & $\text{C1}$ \qquad \;$\text{C2}$\; \qquad MSE\\
\hline
& & & & \\          
SCAD, $a_{\text{scad}}=3.7$ & $85.41 \;-\; 64.17 \;-\; 0.5612$ & $90.11 \;-\;  75.33  \;-\;  0.2964$ & $82.57 \;-\;  59.83 \;-\;   0.6255$ & $94.20 \;-\;  75.67 \;-\;   0.1856$ \\
& & & & \\
$a_{\text{scad}}=10$ & $84.46 \;-\;  63.83  \;-\;  0.5860$ & $87.02 \;-\;  74.00   \;-\; 0.2607$ & $83.72 \;-\;  58.83  \;-\;  0.6021$ & $90.33 \;-\;  77.17  \;-\;  0.1800$ \\
& & & & \\
$a_{\text{scad}}=20$ & $86.07 \;-\;  63.33 \;-\;   0.5711$ & $87.93 \;-\;  78.50  \;-\;  0.2464$ & $82.94 \;-\;  61.17  \;-\;  0.6037$ & $90.87 \;-\;  80.83  \;-\;  0.1584$ \\
& & & & \\
$a_{\text{scad}}=40$ & $83.46 \;-\;  69.33  \;-\;  0.5246$ & $83.70 \;-\;  83.67   \;-\; 0.2280$ & $80.70  \;-\; 65.33  \;-\;  0.5602$ & $87.57  \;-\; 85.33  \;-\;  0.1456$ \\
& & & & \\
$a_{\text{scad}}=70$ & $82.11 \;-\;  72.50  \;-\;  0.4834$ & $79.50  \;-\; 86.17   \;-\; 0.2297$ & $78.72 \;-\;  67.67 \;-\;   0.5322$ & $84.59  \;-\; 86.17  \;-\;  0.1462$ \\ \hline
& & & & \\               
MCP, $b_{\text{mcp}}=3.5$ & $88.04  \;-\; 58.50  \;-\;  0.6430$ & $94.98 \;-\;  69.00  \;-\;  0.3396$ & $86.41  \;-\; 53.33  \;-\;  0.7295$ & $95.78  \;-\; 71.67   \;-\; 0.2044$ \\
& & & & \\ 
$b_{\text{mcp}}=10$ & $83.22 \;-\;  62.00  \;-\;  0.5809$ & $88.98  \;-\; 71.50 \;-\;   0.2563$ & $84.06 \;-\;  58.00 \;-\;   0.5846$ & $92.26 \;-\;  72.17  \;-\;  0.1653$ \\
& & & & \\
$b_{\text{mcp}}=20$ & $84.17 \;-\;  63.83  \;-\;  0.5711$ & $89.74 \;-\;  77.00   \;-\; 0.2386$ & $81.89 \;-\;  61.17  \;-\;  0.5927$ & $92.35 \;-\;  79.67   \;-\; 0.1442$ \\
& & & & \\
$b_{\text{mcp}}=40$ & $83.43 \;-\;  69.33  \;-\;  0.5234$ & $84.65 \;-\;  82.50  \;-\;  0.2211$ & $79.89 \;-\;  66.17  \;-\;  0.5513$ & $90.11 \;-\;  84.67 \;-\;   0.1385$ \\
& & & & \\
$b_{\text{mcp}}=70$ & $81.33 \;-\;  72.67 \;-\;   0.4844$ & $81.06  \;-\; 87.17  \;-\;  0.2249$ & $77.48 \;-\;  68.50   \;-\; 0.5312$ & $85.43  \;-\; 85.33  \;-\;  0.1476$ \\ \hline
& & & & \\
LASSO & $76.87 \;-\;  79.00  \;-\;  0.4092$ & $72.52 \;-\;  88.33  \;-\;  0.2245$ & $76.00 \;-\;  74.33  \;-\;  0.4384$ & $77.65 \;-\;  89.00  \;-\;  0.1398$ \\
\hline
\hline\end{tabular}}}
\end{table}

\section{Conclusion}

We studied the asymptotic properties of sparse M-estimator based on pseudo-observations, where we treat the marginal distributions entering the loss function as unknown, which is a common situation in copula inference. Our framework includes, among others, semi-parametric copula models and the CML inference method. We assume sparsity among the coefficients of the true copula parameter and apply a penalty function to recover the sparse underlying support. Our method is based on penalized M-estimation and accommodates data-dependent penalties, such as the LASSO, SCAD and MCP. We establish the consistency of the sparse M-estimator together with the oracle property for the SCAD and MCP cases for both fixed and diverging dimensions of the vector of parameters.
Because of the presence of non-parametric estimators of marginal cdfs' and potentially unbounded loss functions, it is difficult to exhibit simple regularity conditions and to derive the oracle property. 
This would make the large sample analysis intricate when $p$ and $d$ simultaneously diverge. We shall leave it as a future research direction. Among potential applications of our methodology, the (brute force) estimation of vine models (\cite{czado2019analyzing}) under sparsity seems to be particularly relevant. Nonetheless, checking our regularity assumptions for such highly nonlinear models would surely be challenging. 
In addition, it would be interesting to prove similar theoretical results in the case of conditional copulas for which their conditional margins would depend on covariates.

\noindent\textbf{Acknowledgements}

\mds

J.D. Fermanian was supported by the labex Ecodec (reference project ANR-11-LABEX-0047) and B. Poignard by the Japanese Society for the Promotion of Science (Grant 22K13377).

\newpage

\appendix

\section{Multivariate rank statistics and empirical copula processes indexed by functions}\label{technicalities}


In this section, we prove some theoretical results about the asymptotic behavior of multivariate rank statistics
$$ \int f d\widehat C_n = \frac{1}{n} \sum_{i=1}^n f\big(  F_{n,1}(X_{i,1}),\ldots,F_{n,d}(X_{i,d}) \big) ,$$
for a class of maps $f:(0,1)^d\rightarrow \Rb$ that will be of ``locally'' bounded variation and sufficiently regular.
Such maps will be allowed to diverge when some of their arguments tend to zero or one, i.e. when their arguments are close to the boundaries of $(0,1)^d$.
We will prove the asymptotic normality of $\int f d\widehat C_n$, extending Theorem 3.3 of~\cite{berghaus2017weak} to any dimension $d\geq 2$.
Moreover, we will state the weak convergence $\sqrt{n}\int f d(\widehat C_n-C)$ seen as an empirical process indexed by $f\in \Fc$, for a convenient family of maps $\Fc$.

\mds

To be specific, consider a family of measurable maps $\Fc=\{f : (0,1)^d \rightarrow \Rb\}$.
As in~\cite{berghaus2017weak} and for any $\omega\geq 0$, define the weight function
$$g_{\omega,d} (\uu):= \min \{ \min_{k=1,\ldots,d} u_k, 1-\min_{j\neq 1} u_j,\ldots,1-\min_{j\neq d} u_j \}^\omega,\; \uu \in [0,1]^d.$$
When $u_1\leq u_2 \leq \cdots \leq u_d$, check that $g_{\omega,2} (\uu)=\min(u_1,1-u_2)$. Moreover, if $d=2$, then
$g_{\omega,2} (u_1,u_2)=\min (u_1,u_2,1-u_1,1-u_2)$.
To lighten notations and when there will be no ambiguity, the map $g_{\omega,d}$ will simply denoted as $g_{\omega}$ hereafter.
For technical reasons, we will need the map $\tilde g_\omega (\uu):=g_\omega (\uu) + \1(g_\omega (\uu)=0)$ for every $\uu\in [0,1]^d$.

\mds 
Recall the process $\widehat\Cb_n :=\sqrt{n}(\widehat C_n - C)$ and $\widehat\Cb_n(f)=\int f\, d\widehat \Cb_n$ for any $f \in \Fc$.
Therefore, $\widehat\Cb_n$ may be considered as a process defined on $\Fc$.
The maps $f\in \Fc$ may potentially be unbounded, particularly when their arguments tend to the boundaries of the hypercube $[0,1]^d$.
This is a common situation when $f$ is chosen as the log-density of many copula families.
Moreover, we will need to apply an integration by parts trick that has proved its usefulness in several copula-related papers, particularly~\cite{radulovic2017weak} and~\cite{berghaus2017weak}.
To this end, we introduce the following class of maps.
\begin{definition}
A map $f$ is of locally bounded Hardy Krause variation, 
a property denoted by $f\in BHKV_{loc}\big((0,1)^d\big)$, if, for any sequence $(a_n)$ and $(b_n)$, $0<a_n<b_n<1$, $a_n\rightarrow 0$, $b_n\rightarrow 1$, the restriction of $f$ to $[a_n,b_n]^d$ is of bounded Hardy-Krause variation.
\end{definition}
The concept of Hardy-Krause variation has become a standard extension of the usual concept of bounded variation for multivariate maps: see the Supplementary Material in~\cite{berghaus2017weak} or Section 2 and Appendix A in~\cite{radulovic2017weak}, and the references therein.

\mds

Denote the box $\BB_{n,m}:=(1/2n;1-1/2n]^m$ and $\BB^c_{n,m}$ its complementary in $[0,1]^m$, $1<m\leq d$.
Moreover, any sub-vector whose components are all equal to $1/2n$ (resp. $1-1/2n$) will be denoted as $\cc_{n}$ (resp. $\dd_{n}$).
For any $f\in BHKV_{loc}\big((0,1)^d\big)$ and a measurable map $g:(0,1)^d\rightarrow \Rb$, the integral $\int_{(0,1)^d} g \,df$ can be conveniently defined: see~\cite{berghaus2017weak}, Section 3.1 and its Supplementary Material.

\mds 

In terms of notations, we use the same rules as~\cite{radulovic2017weak}, Section 1.1, to manage sub-vectors and the way of concatenating them.
More precisely, for $J \subset \{1,\ldots,d\}$, $|J|$ denotes the cardinality of $J$, and the unary minus refers to the complement with respect to $\{1,\ldots,d\}$ so that $-J=\{1,\ldots,d\}\setminus J $.
For $J \subset \{1,\ldots,d\}$, $\uu_{J}$ denotes a $|J|$-tuple of real numbers whose elements are $u_j, j \in J$; the vector $\uu_J$ typically belongs to $[0,1]^{|J|}$. Now let $J_1,J_2 \subset \{1,\ldots,d\}$, $J_1 \cap J_2 = \emptyset$ and $\uu,\mathbf{v}$ two vectors in $[0,1]^d$. The concatenation symbol ``$:$'' is defined as follows: the vector $\uu_{J_1}:\mathbf{v}_{J_2}$ denotes the point $\xx \in [0,1]^{|J_1 \cup J_2|}$ such that $x_j = u_j$ for $j \in J_1$ and $x_j = v_j$ for $j \in J_2$. The vector $\uu_{J_1}:\mathbf{v}_{J_2}$ is well defined for $\uu_{J_1} \in [0,1]^{|J_1|}$ and $\mathbf{v}_{J_2}\in [0,1]^{|J_2|}$ when $J_1 \cap J_2=\emptyset$ even if $\uu_{-J_1}$ and $\mathbf{v}_{-J_2}$ remains unspecified. We use this concatenation symbol to glue together more than two sets of components: let $\uu_{J_1} \in [0,1]^{|J_1|},\mathbf{v}_{J_2}\in [0,1]^{|J_2|},\ww_{J_3} \in [0,1]^{|J_3|}$ with $J_1,J_2,J_3$ mutually disjoint sets such that $J_1 \cup J_2 \cup J_3=\{1,\ldots,d\}$. Then $\uu_{J_1}:\mathbf{v}_{J_2}:\ww_{J_3}$ is a well defined vector in $[0,1]^d$.
Finally, for a function $f:[0,1]^d \rightarrow \Rb$ and a constant vector $\mathbf{c}_{J} \in [0,1]^{|J|}$, the function $\xx_\mapsto f(\xx_{J}:\mathbf{c}_{-J})$ denotes a lower-dimensional projection of $f$ onto $[0,1]^{|J|}$. The integral of a function $g: [0,1]^{|J|} \mapsto \Rb$ w.r.t. the measure induced by the latter map will be denoted as $\int g(\xx_J)\,f(d\xx_{J}:\mathbf{c}_{-J})$.

\begin{definition}
\label{def_gomega_reg}
A family of maps $\Fc$ is said to be regular with respect to the weight function $g_{\omega,d}$ for some $\omega\in (0,1/2)$ (or $g_\omega$-regular, to be short) if
\begin{itemize}
\item[(i)]
every $f\in \Fc$ is $BVHK_{loc}\big((0,1)^d\big)$ and right-continuous;
\item[(ii)]
the map $\uu\mapsto \sup_{f\in\Fc}\min_k \min(u_k,1-u_k)^\omega |f(\uu)|$ is bounded on $(0,1)^d$,
\begin{equation}
\sup_{f\in \Fc} \int_{(0,1)^d} g_{\omega,d}(\uu)\,  |f(d\uu) | <\infty,
\label{bound_fg}
\end{equation}
and, for any partition $(J_1,J_2,J_3)$ of the set of indices $\{1,\ldots,d\}$ with $J_1\neq \emptyset$,
\begin{equation}
 \sup_{f\in \Fc}\int_{\BB_{n,|J_1|}} g_{\omega,d}(\uu_{J_1}:\cc_{n,J_2}:\dd_{n,J_3}\big) \big|f\big(d\uu_{J_1}:\cc_{n,J_2}:\dd_{n,J_3}\big)\big| =O(1).
\label{bound_Cn_1}
\end{equation}
Moreover, the latter sequence tends to zero when $J_2\neq \emptyset$.
\end{itemize}
\end{definition}

When $\Fc=\{f_0\}$ is a singleton, one simply says that the map $f_0$ is $g_\omega$-regular. Note that, if $\uu_{J_1}\in \BB_{n,|J_1|} $, then $g_{\omega,d}(\uu_{J_1}:\cc_{n,J_2}:\dd_{n,J_3}\big)=(2n)^{-\omega}$ except when $J_2=\emptyset$ and $|J_1|\geq 2$ simultaneously. 
In the latter case, $g_{\omega,d}(\uu_{J_1}:\dd_{n,-J_1}\big)=g_{\omega,|J_1|}(\uu_{J_1})$.
\begin{remark}
\label{gomega_regularity_decomposition}
Consider a family $\Fc$ of maps from $(0,1)^d$ to $\Rb$.
Assume there exist $m$ subsets $I_k\subset \{1,\ldots,d\}$, $k\in \{1,\ldots,m\}$ s.t. every member $f \in\Fc$ can be written
$$ f(\uu)= f_{1,I_1}(\uu_{I_1}) + \ldots + f_{m,I_m}(\uu_{I_m}), \;\; \uu\in (0,1)^d,  $$
for some maps $f_{k,I_k}:(0,1)^{|I_k|}\rightarrow \Rb$, $k\in \{1,\ldots,m\}$.
Define $\Fc_k=\{ f_{k,I_k}; f\in \Fc\}$ for every $k$. 
If every $\Fc_k$, $k\in \{1,\ldots,m\},$ is regular w.r.t. the weight function $g_{\omega,|I_k|}$, then 
it is easy to see that $\Fc$ is regular w.r.t. the weight function $g_{\omega,d}$.
This property may be invoked to prove the $g_\omega$ regularity of the Gaussian copula family, for instance (see Section \ref{verif_regul_Gaussian_cop} in the Appendix).
\end{remark}
\begin{remark}
\label{additiv_gomega_prop}
Any family $\Fc$ of maps defined on $(0,1)^d$ may formally 
be seen as a family $\tilde \Fc$ of maps defined on a larger dimension, say $(0,1)^{d+p}$, $p>0$: every $f\in \Fc$ defines a map $\tilde f$ on $(0,1)^{d+p}$ by setting
$\tilde f(\uu,\mathbf{v})=f(\uu)$, $\uu\in (0,1)^d$, $\mathbf{v}\in (0,1)^p$. 
It can be easily checked that, if $\Fc$ is $g_\omega$ regular then this is still the case for $\tilde \Fc$. 
\end{remark}
Beside \textcolor{black}{the} regularity conditions on the family of maps $\Fc$, we will need that the (standard) empirical process $\alpha_n$ is well-behaved.
To this aim, we recall the so-called conditions 4.1, 4.2 and 4.3 in~\cite{berghaus2017weak}.
\begin{assumption}
\label{oscillation_modulus_assump}
There exists $\textcolor{black}{\kappa}_1\in (0,1/2]$ such that, for all $\mu\in (0,\textcolor{black}{\kappa}_1)$ and all sequences $\delta_n\rightarrow 0$, we have
$$  \sup_{|\uu-\mathbf{v}|<\delta_n} \frac{|\alpha_n(\uu) - \alpha_n(\mathbf{v})|}{|\uu - \mathbf{v}|^\mu \vee n^{-\mu}}= o_P(1).    $$
\end{assumption}
\begin{assumption}
\label{tail_empir_processes}
There exists $\textcolor{black}{\kappa}_2 \in (0,1/2 ]$ and $\textcolor{black}{\kappa}_3\in (1/2,1]$ such that, for any $\textcolor{black}{\nu} \in (0,\textcolor{black}{\kappa}_2)$, any $\lambda \in (0,\textcolor{black}{\kappa}_3)$ and any $j\in \{1,\ldots,d\}$, we have
$$ \sup_{u\in (0,1)}\Big| \frac{\sqrt{n} \big\{ G_{nj}(u)-u \big\}}{u^{\textcolor{black}{\nu}} (1-u)^{\textcolor{black}{\nu}}} \Big| +
\sup_{u\in (1/n^\lambda,1-1/n^\lambda)}\Big| \frac{\sqrt{n} \big\{ G^-_{nj}(u)-u \big\}}{u^{\textcolor{black}{\nu}} (1-u)^{\textcolor{black}{\nu}}} \Big| =O_P(1).$$
\end{assumption}

\begin{assumption}
\label{cond_wc_empir_process}
The empirical process $(\alpha_n)$ converges weakly in $\ell^\infty([0,1]^d)$ to some limit process $\alpha_C$ which has continuous sample paths, almost surely.
\end{assumption}

As pointed out in~\cite{berghaus2017weak}, such conditions~\ref{oscillation_modulus_assump}-\ref{cond_wc_empir_process} are satisfied for i.i.d. data with
$\textcolor{black}{\kappa}_1=1/2$, $\textcolor{black}{\kappa}_2=1/2$ and $\textcolor{black}{\kappa}_3=1$. In the latter case, the limiting process $\alpha_C$ is a $C$-Brownian bridge, such that $\text{cov}\big\{\alpha_C(\uu),\alpha_C(\mathbf{v})\big\}=C(\uu \wedge \mathbf{v}) - C(\uu)C(\mathbf{v})$ for any $\uu$ and $\mathbf{v}$ in $[0,1]^d$, with the usual notation
$\uu \wedge \mathbf{v}=\big(\min(u_1,v_1),\ldots,\min(u_d,v_d) \big)$.
More generally, if the process $(\XX_i)_{i\in \Nb}$ is \textcolor{black}{strongly} stationary and geometrically $\alpha$-mixing, then
the assumptions~\ref{oscillation_modulus_assump}-\ref{cond_wc_empir_process} are still satisfied, with the same choice $\textcolor{black}{\kappa}_1=1/2$, $\textcolor{black}{\kappa}_2=1/2$ and $\textcolor{black}{\kappa}_3=1$
(Proposition 4.4 in~\cite{berghaus2017weak}). In the latter case, the covariance of the limiting process is more complex:
$\text{cov}\big\{\alpha_C(\uu),\alpha_C(\mathbf{v})\big\}=\sum_{j\in \Zb} \text{cov}\big\{ \1( \UU_0 \leq \uu), \1( \UU_j \leq \uu) \big\}$.

\begin{assumption}
\label{regularity_cond_wc}
For any $I\subset \{1,\ldots,d\}$, $I\neq \emptyset$, any $f$ that belongs to a regular family $\Fc$ and for any continuous map $h:[0,1]^{|I|}\rightarrow \Rb$, the sequence
$\int_{\BB_{n,|I|}} h(\uu_I) g_\omega\big(\uu_{I}:\dd_{n,-I}\big) \, f\big(d\uu_{I}:\dd_{n,-I}\big)$ is convergent when $n\rightarrow \infty$.
Its limit is denoted as
$\int_{(0,1)^{|I|}} h(\uu_I) g_\omega\big(\uu_{I}:\1_{-I}\big) \, f\big(d\uu_{I}:\1_{-I}\big)$, i.e. it is
given by an integral w.r.t. a borelian measure on $(0,1)^{|I|}$ denoted as $f\big(\cdot:\1_{-I}\big)$.
\end{assumption}
The latter regularity condition is required to get the weak convergence of our main statistic in Theorem~\ref{Th_fondam_multiv_rank_stat}.
Note that the map $f(\uu)$ is not defined when one component of
$\uu$ is one. Therefore, the way we write the limits in Assumption~\ref{regularity_cond_wc} is a slight abuse of notation.
Typically, $f(\uu_I:\1_{-I})$ will be defined as the limit $f(\uu_I:\mathbf{v}_{-I})$ when $\mathbf{v}_{-I}$ tends to $\1_{-I}$ when such a limit exists.
In other standard situations, there exists a measurable map $h_f$ such that
$f\big(d\uu_I:\uu_{-I}\big)=h_f(\uu)\,d\uu_I$. If it is possible to extend by continuity the map $\uu \mapsto g_\omega(\uu)h_f(\uu)$ when $\uu_{-I}$
tends to $\1_{-I}$, simply set $g_\omega\big(\uu_{I}:\1_{-I}\big)  f\big(d\uu_I:\1_{-I}\big)= g_\omega\big(\uu_{I}:\1_{-I}\big)  h_f(\uu_I:\1_{-I})\,d\uu_I$.
But other more complex situations.

\mds

To get the weak convergence of the process $\widehat \Cb_n$ indexed by the maps in $\Fc$, we will need to strengthen the latter assumption~\ref{regularity_cond_wc}, so that
it becomes true uniformly over $\Fc$.
\begin{assumption}
\label{regularity_cond_wc_stronger}
For any $I\subset \{1,\ldots,d\}$ and any continuous map $h:[0,1]^{|I|}\rightarrow \Rb$,
$$ \sup_{f\in \Fc} \big|
\int_{\BB_{n,|I|}} h(\uu_I) g_\omega\big(\uu_{I}:\dd_{n,-I}\big) \, f\big(d\uu_{I}:\dd_{n,-I}\big)
-  \int_{(0,1)^{|I|}} h(\uu_I) g_\omega\big(\uu_{I}:\1_{-I}\big) \, f\big(d\uu_{I}:\1_{-I}\big) \big| \longrightarrow 0,$$
when $n\rightarrow \infty$.
\end{assumption}

\begin{theorem}\label{Th_fondam_multiv_rank_stat}
(i) Assume the assumptions~\ref{cond_reg_copula},\ref{oscillation_modulus_assump} and~\ref{tail_empir_processes} are satisfied and consider a family $\Fc$ of maps that
is $g_\omega$-regular, for some $\omega \in \big(0,\min(\frac{\textcolor{black}{\kappa}_1}{2(1-\textcolor{black}{\kappa}_1)},\frac{\textcolor{black}{\kappa}_2}{2(1-\textcolor{black}{\kappa}_2)},\textcolor{black}{\kappa}_3 - 1/2)\big)$. Then, for any $f\in \Fc$, we have
\begin{eqnarray}
\lefteqn{ \int f \,d\widehat\Cb_n =
(-1)^{d}\int_{\BB_{n,d}} \bar\Cb_n(\uu) \, f\big(d\uu\big) \nonumber }\\
  & \hspace*{-0.5cm}+&\hspace*{-0.5cm} \sum_{ \substack{I \subset \{1,\ldots,d\} \\ I\neq \emptyset, I\neq \{1,\ldots,d\} } } (-1)^{|I|}
  \int_{\BB_{n,|I|}} \frac{ \bar\Cb_n(\uu_{I}:\1_{-I})}{\tilde g_\omega(\uu_{I}:\1_{-I})}
   g_\omega\big(\uu_{I}:\dd_{n,-I}\big) \, f\big(d\uu_{I}:\dd_{n,-I}\big) + r_n(f),
\label{key_prop_1}
\end{eqnarray}
where $\bar \Cb_n(\uu):= \alpha_n(\uu) - \sum_{k=1}^d \dot C_k (\uu)\alpha_n(\1_{-k}:u_k)$, $\uu\in [0,1]^d$ and $\sup_{f\in \Fc}|r_n(f)|=o_P(1)$.

\mds

(ii) In addition, assume the conditions~\ref{cond_wc_empir_process} and~\ref{regularity_cond_wc} apply.
Then, for any function $f\in \Fc$, the \textcolor{black}{sequence of random variables} 
$ \sqrt{n}\int f \,d(\widehat C_n-C)$ tends in law to the centered Gaussian r.v.
\begin{equation}
 (-1)^d\int_{(0,1)^d} \Cb(\uu) \, f(d\uu)+
\sum_{ \substack{I \subset \{1,\ldots,d\} \\ I\neq \emptyset, I\neq \{1,\ldots,d\} } } (-1)^{|I|}
  \int_{(0,1)^{|I|}} \Cb(\uu_{I}:\1_{-I}) \, f\big(d\uu_{I}:\1_{-I}\big),
  \label{limiting_law_ref}
  \end{equation}
where $\Cb(\uu):= \alpha_C(\uu) - \sum_{k=1}^d \dot C_k (\uu)\alpha_C(\1_{-k}:u_k)$ for any $\uu\in [0,1]^d$.

(iii)
Under the latter assumptions of (i) and (ii), in addition to Assumption~\ref{regularity_cond_wc_stronger},
$\widehat\Cb_n $ weakly tends in $\ell^{\infty}(\Fc)$ to a Gaussian process.
\end{theorem}
Points (i) and (ii) of the latter theorem yield a generalization of Theorem 3.3 in~\cite{berghaus2017weak} for any arbitrarily chosen dimension $d\geq 2$,
and uniformly over a class of functions. Note that it is always possible to set $\Fc=\{f_0\}$ and we have proved the weak convergence of a single 
multivariate rank statistic.
The proof is based on the integration by part formula in~\cite{radulovic2017weak}.
Note that, in dimension $d=2$, $\Cb(u_1,1)=\Cb(1,u_2)=0$ for any $u_1,u_2\in (0,1)$. Thus, in the bivariate case, the limiting law of $\sqrt{n}\int f \,d(\widehat C_n-C)$ is simply the law of
$\int_{(0,1)^2} \Cb \, df$, as stated in Theorem 3.3 in~\cite{berghaus2017weak}.
Nonetheless, this is no longer true in dimension $d>2$, explaining the more complex limiting laws in our Theorem~\ref{Th_fondam_multiv_rank_stat}.
Finally, the sum in~(\ref{limiting_law_ref}) can be restricted to the subsets $I$ such that $|I|\geq 2$. Indeed, when $I$ is a singleton, then
$\Cb(\uu_{I}:\1_{-I})$ is zero a.s.

\mds

The point (iii) of Theorem~\ref{Th_fondam_multiv_rank_stat} extends Theorem 5 in~\cite{radulovic2017weak}. The latter one was restricted to 
right-continuous maps $f$ of bounded Hardy-Krause variation and defined on the whole hypercube $[0,1]^d$. For any of these maps $f$, there exists 
a finite signed Borel measure bounded $\nu_f$ on $[0,1]^d$ such that $f(\uu)=\nu_f\big([\0,\uu]\big)$ for every $\uu\in [0,1]^d$ (Theorem 3 in~\cite{aistleitner2014functions}). In particular, they are bounded on $[0,1]^d$, an excessively demanding assumption in many cases. 
Indeed, for the inference of copulas, many families of relevant maps $\Fc$ contain elements that are not of bounded variation or cannot be defined on $[0,1]^d$ as a whole, as pointed out by several authors, following~\cite{segers2012asymptotics}; see Section 3.4 in~\cite{radulovic2017weak} too. This is in particular the case with the Canonical Maximum Likelihood method and Gaussian copulas. In such a case, $f(\uu)=-\ln c_{\Sigma}(\uu)$ and $c_\Sigma$ is the density of a Gaussian copula with correlation parameter $\Sigma$.
Therefore, we have preferred the less restrictive approach of~\cite{berghaus2017weak}, that can tackle unbounded maps $f$ (in particular copula log-densities), through the concept of locally bounded Hardy Krause variation.

We deduce from Theorem~\ref{Th_fondam_multiv_rank_stat} an uniform law of large numbers too.
\begin{cor}
Assume the assumptions~\ref{cond_reg_copula},\ref{oscillation_modulus_assump} and~\ref{tail_empir_processes} are satisfied and consider a family $\Fc$ of maps that
is $g_\omega$-regular (for some $\omega$ in the same range as in Theorem~\ref{Th_fondam_multiv_rank_stat}). If $\sup_{\uu\in (0,1)^d} |\alpha_n(\uu)|=O_P(1)$ then, for any positive sequence $(\mu_n)$ of real numbers s.t. $\mu_n\rightarrow +\infty$ \textcolor{black}{when $n\rightarrow \infty$}, we have
\begin{equation*}
 \sup_{f\in \Fc} \big| \int f \,d(\widehat C_n-C) \big| = O_P(\mu_n/\sqrt{n}).
\label{GC_copulas}
\end{equation*}
\label{cor_GC}
\end{cor}
\begin{remark}
\label{rem_ULLN_bis}
In the literature, some ULLN for copula models have already been applied, but without specifying the corresponding rates of convergence to zero.
In semi-parametric models, some authors invoked some properties of bracketing numbers (Lemma 1 in~\cite{chen2005pseudo}; Th. 17 in the working paper version of~\cite{fermanian2002weak}): if, for every $\delta>0$, the $L^1(C)$ bracketing number of $\Fc$ (denoted $N_{[\cdot]}\big(\delta, \Fc,L^1(C)\big)$ in the literature) is finite, then 
$ \sup_{f\in \Fc} \big| \int f \,d(\widehat C_n-C) \big|$ tends to zero a.s.
\end{remark}

\section{Asymptotic variance of $\WW$}\label{asymptoticvarianceZ}

Here, we provide a plug-in estimator of the variance-covariance matrix of the vector $\WW$ that appeared in Theorem~\ref{oracle_property}.
The latter vector is centered Gaussian and, for every $(i,k)\in \Ac^2$, 
{\footnotesize{\begin{eqnarray}
\lefteqn{ \Eb[W_j W_k]=
\int_{(0,1)^{2d}} \Eb\big[\Cb(\uu)\Cb(\uu')\big]  \, \partial_{\theta_j}\ell(\theta_0;d\uu)\partial_{\theta_k}\ell(\theta_0;d\uu') \label{formula_cov_Z} }\\
&+& \sum_{ \substack{I \subset \{1,\ldots,d\} \\ I\neq \emptyset, I\neq \{1,\ldots,d\} } } (-1)^{d+|I|}
    \int_{(0,1)^{d+|I|}} \Eb\big[ \Cb(\uu) \Cb(\uu'_{I}:\1_{-I}) \big]\, \partial_{\theta_k}\ell(\theta_0;d\uu'_{I};\1_{-I})\, \partial_{\theta_j}\ell(\theta_0;d\uu)   \nonumber  \\
&+&  \sum_{ \substack{I \subset \{1,\ldots,d\} \\ I\neq \emptyset, I\neq \{1,\ldots,d\} } } (-1)^{d+|I|}
  \int_{(0,1)^{|I|}} \Eb\big[ \Cb(\uu') \Cb(\uu_{I}:\1_{-I}) \big]\, \partial_{\theta_j}\ell(\theta_0;d\uu_{I};\1_{-I})\, \partial_{\theta_k}\ell(\theta_0;d\uu')   \nonumber \\
&+&
\sum_{ \substack{I,I' \subset \{1,\ldots,d\} \\ I,I'\neq \emptyset; I,I'\neq \{1,\ldots,d\} } } 
(-1)^{|I|+ |I'|}
   \int_{(0,1)^{|I|+|I'|}}  
  \Eb\big[ \Cb(\uu_{I}:\1_{-I})  \Cb(\uu'_{I'}:\1_{-I'}) \big]
  \partial_{\theta_j}\ell(\theta_0;d\uu_{I};\1_{-I})
 \, \partial_{\theta_k}\ell(\theta_0;d\uu'_{I'};\1_{-I'}).
\nonumber
\end{eqnarray}}}
In the latter formula, we will replace $\theta_0$ with $\widehat\theta_n$. 
Denote the covariance function of the process $\alpha_C$ as 
$v_{\alpha}$, i.e. $v_\alpha(\uu,\mathbf{v}):=\Eb\big[\alpha_C(\uu)\alpha_C(\mathbf{v}) \big]$, for every $\uu$ and $\mathbf{v}$ in $[0,1]^d$.
Then, the covariance function of the process $\Cb$ is 
\begin{eqnarray*}
\lefteqn{
\Eb\big[\Cb(\uu)\Cb(\mathbf{v})\big] =v_\alpha(\uu,\mathbf{v}) 
- \sum_{k=1}^d \dot C_k (\mathbf{v}) v_\alpha\big(\uu,(\1_{-k}:v_k)\big)    }\\
&-& \sum_{k=1}^d \dot C_k (\uu) v_\alpha\big(\mathbf{v},(\1_{-k}:u_k)\big)
+  \sum_{k,k'=1}^d \dot C_k (\uu) \dot C_k (\mathbf{v}) v_\alpha\big((\1_{-k}:u_k),(\1_{-k}:v_{k'})\big).
\end{eqnarray*}
In the latter formula, every partial derivative of the copula $C$ could be empirically approximated, as in~\citep{remillard2009testing} for instance. 
Moreover, assume we have found an estimator of the map $(\uu,\mathbf{v})\mapsto v_\alpha(\uu,\mathbf{v})$, denoted $\widehat v_\alpha$. With i.i.d data, 
$v_\alpha(\uu,\mathbf{v})=C(\uu\wedge \mathbf{v}) - C(\uu)C(\vv)$ is obviously approximated by 
$\widehat v_\alpha(\uu,\mathbf{v}):=\widehat C_n(\uu\wedge \mathbf{v}) - \widehat C_n(\uu)\widehat C_n(\mathbf{v})$.
This would yield and estimator 
of $\Eb\big[\Cb(\uu)\Cb(\uu')\big]$, for every $(\uu,\mathbf{v})\in (0,1)^d$, 
that can be plugged in~(\ref{formula_cov_Z}). Taking all pieces together yields an estimator of $\Eb[W_j W_k]$.

\section{Proofs of Theorem~\ref{Th_fondam_multiv_rank_stat} and Corollary~\ref{cor_GC}}\label{proof_appendix_multi}

\subsection{Proof of Theorem~\ref{Th_fondam_multiv_rank_stat}}

To state (i), we follow the same paths as in the proof of Theorem 3.3 in~\citep{berghaus2017weak}.
For any $0<a<b<1/2$, define $N(a,b):=\{\uu\in [0,1]^d: a< g_{1,d}(\uu)\leq b\}$.
Note that, when $d=2$, $N(a,1/2)=(a,1-a)^2$, but this property does not extend to larger dimensions.
Any remainder term that tends to zero in probability uniformly w.r.t. $f\in \Fc$ will be denoted as $o_{P,u}(1)$.
 Note that, for any $f\in \Fc$,
 $$  \sqrt{n}\Big\{ \int_{\BB_{n,d}} f \, d\widehat C_n - \Eb\big[ f(\UU) \big] \Big\}= \int_{\BB_{n,d}} f\, d\widehat \Cb_n - \sqrt{n}\int_{\BB_{n,d}^c} f\, dC =: A_n - r_{n1}.$$

\mds
Let us prove that $\sup_{f\in \Fc}|r_{n1}|=o(1)$. Indeed, a vector $\uu$ belongs to $\BB_{n,d}^c$ iff one of its components is smaller than $1/2n$ or is strictly larger than $1-1/2n$. Thus, let us decompose $\BB_{n,d}^c$ as the disjoint union of ``boxes'' on $[0,1]^d$ such as
$$ \BB_{n}^{J_1,J_2,J_3}:=\big\{ \uu \:| \: \uu_{J_1}\in [0,1/2n]^{|J_1|},\uu_{J_2}\in (1/2n,1-1/2n]^{|J_2|},\uu_{J_3}\in (1-1/2n,1]^{|J_3|} \big\},$$
where $J_1\cup J_3 \neq \emptyset$ and $(J_1,J_2,J_3)$ is a partition of $\{1,\ldots,d\}$.
Note that, for any $\uu \in \BB_{n}^{J_1,J_2,J_3}$, we have
$$\{\min_k \min(u_k,1-u_k)\}^{-\omega}\leq \sum_{k\in I_1\cup I_3} \{u_k^{-\omega}+(1-u_k)^{-\omega} \}.$$
Since there exists a constant $C_\Fc$ such that $\sup_{\uu\in [0,1]^d} \{\min_k \min(u_k,1-u_k)\}^{\omega}  |f(\uu)| \leq C_\Fc$ for every $f\in \Fc$ by $g_{\omega}$-regularity,
we have for any $f\in \Fc$
  \begin{eqnarray*}
  \lefteqn{0\leq \sqrt{n}\int_{\BB_{n}^{J_1,J_2,J_3}} |f|\, dC  \leq \sqrt{n} C_\Fc \int_{\BB_{n}^{J_1,J_2,J_3}} \{\min_k \min(u_k,1-u_k)\}^{-\omega} \,C(d\uu)      }\\
&\leq & \sqrt{n}C_\Fc \sum_{k\in J_1\cup J_3} \int_{\BB_{n}^{J_1,J_2,J_3}} \big\{ u_k^{-\omega}+ (1-u_k)^{-\omega} \big\}\,C(d\uu)   \\
&\leq & \sqrt{n} C_\Fc\sum_{k\in J_1} \int_{\{u_k\in (0,1/2n],\uu_{-k}\in (0,1]^{d-1}\}} \big\{ \frac{C(d\uu)}{u_k^{\omega}}
+  \frac{C(d\uu)}{(1-u_k)^{\omega}} \big\} \\
&+ & \sqrt{n} C_\Fc\sum_{k\in J_3} \int_{\{u_k\in (1-1/2n,1],\uu_{-k}\in (0,1]^{d-1}\}} \big\{ \frac{C(d\uu)}{u_k^{\omega}}
+  \frac{C(d\uu)}{(1-u_k)^{\omega}} \big\} \\
&\leq & \sqrt{n} C_\Fc\sum_{k\in J_1\cup J_3} \big\{ \int_{\{u_k\in (0,1/2n]\}} u_k^{-\omega}\, du_k+ \int_{\{u_k\in (1-1/2n,1]\}} (1-u_k)^{-\omega}\, du_k \big\}\\
& \leq & 2C_\Fc|J_1\cup J_3 | (2n)^{\omega-1/2}/(1-\omega),
\end{eqnarray*}
that tends to zero with $n$ uniformly wrt $f\in \Fc$.
Therefore, we have proven that $\sup_{f\in \Fc}|r_{n1}|=o(1)$.

\mds

Moreover, invoking the integration by parts formula (40) in~\citep{radulovic2017weak}, we get
  \begin{eqnarray*}
  \lefteqn{ A_n= \int_{\BB_{n,d}} f\, d\widehat \Cb_n = (-1)^{d}\int_{\BB_{n,d}} \widehat\Cb_n(\uu-) \, f\big(d\uu\big) }\\
  &+& \sum_{ \substack{I_1+I_2+I_3= \{1,\ldots,d\} \\ I_1\neq \emptyset, I_1\neq \{1,\ldots,d\} } } (-1)^{|I_1|+|I_2|}
  \int_{\BB_{n,|I_1|}} \widehat\Cb_n(\uu_{I_1}-:\cc_{n,I_2}:\dd_{n,I_3}) \, f\big(d\uu_{I_1}:\cc_{n,I_2}:\dd_{n,I_3}\big) \\
&+& \Delta\big(\widehat\Cb_n f \big)\big( \BB_{n,d}\big)=: A_{n,1} +  A_{n,2} + r_{n2}.
  \end{eqnarray*}
In $A_{n,2}$, the `+' symbol within $I_1+I_2+I_3$ denotes the disjoint union. In other words, the summation is taken over all partitions of 
$\{1,\ldots,d\}$ into three disjoint subsets.
Moreover, we have used the usual notation $\Delta (f)((\uu,\mathbf{v}])$ that
is the sum of component-wise differentials of $f$ over all the vertices of the hypercube $(\uu,\mathbf{v}]$. For instance, in dimension two,
$$ \Delta (f)((\uu,\mathbf{v}]) = f(u_2,v_2)-f(u_1,v_2)-f(u_2,v_1)+f(u_1,v_1).$$
By assumptions~\ref{cond_reg_copula},~\ref{oscillation_modulus_assump} and~\ref{tail_empir_processes}, Theorem 4.5 in~\citep{berghaus2017weak} holds. Then, the term $A_{n,1}$ can be rewritten as
\begin{equation}
A_{n,1}= (-1)^{d}\int_{\BB_{n,d}} \frac{\widehat\Cb_n(\uu-)}{g_\omega(\uu)} \, g_\omega(\uu) f\big(d\uu\big)
 = (-1)^{d}\int_{\BB_{n,d}} \frac{\bar\Cb_n(\uu-)}{g_\omega(\uu-)} \, g_\omega(\uu) f\big(d\uu\big) + o_{P,u}(1).
 \label{Bn}
 \end{equation}
Moreover, due to Lemma 4.10 in~\citep{berghaus2017weak} and their theorem 4.5 again, this yields
$$ A_{n,1} =
(-1)^d\int_{\BB_{n,d}} \frac{\bar\Cb_n(\uu)}{g_\omega(\uu)} \, g_\omega(\uu) f\big(d\uu\big) +o_{P,u}(1)=
(-1)^d\int_{\BB_{n,d}} \bar\Cb_n\,  df +o_{P,u}(1).$$

\mds

The term $A_{n,2}$ is a finite sum of integrals as
$$  \Ic_{n,I_1,I_2,I_3}:= \int_{\BB_{n,|I_1|}} \widehat\Cb_n(\uu_{I_1}-:\cc_{n,I_2}:\dd_{n,I_3}) \, f\big(d\uu_{I_1}:\cc_{n,I_2}:\dd_{n,I_3}\big) ,$$
where $I_1$ is not empty and is not equal to the whole set $\{1,\ldots,d\}$.
By the first part of Theorem 4.5 and Lemma 4.10 in~\citep{berghaus2017weak}, we obtain
$$  \Ic_{n,I_1,I_2,I_3}:= \int_{\BB_{n,|I_1|}} \frac{\bar\Cb_n}{g_\omega}(\uu_{I_1}:\cc_{n,I_2}:\dd_{n,I_3}) \, g_\omega(\uu_{I_1}:\cc_{n,I_2}:\dd_{n,I_3}) f\big(d\uu_{I_1}:\cc_{n,I_2}:\dd_{n,I_3}\big)+o_{P,u}(1).$$
If $I_2\neq \emptyset$, any argument $(\uu_{I_1}:\cc_{n,I_2}:\dd_{n,I_3})$ of $\bar \Cb_n/g_\omega$ belongs to the subset $N(0,1/n)$. In such a case, for any $\epsilon>0$, we have
\begin{eqnarray*}
\lefteqn{ \Pb\Big(  \sup_{f\in \Fc}\big|\Ic_{n,I_1,I_2,I_3}\big| > \epsilon \Big) \leq
\Pb\Big( \sup_{\uu\in N(0,1/n)} \frac{|\bar \Cb_n(\uu)|}{g_\omega(\uu)} > \epsilon^2 \Big) }\\
  &+& \Pb\Big( \sup_{f\in \Fc}\int_{\BB_{n,|I_1|}}   g_\omega\big(\uu_{I_1}:\cc_{n,I_2}:\dd_{n,I_3}\big) \big| f\big(d\uu_{I_1}:\cc_{n,I_2}:\dd_{n,I_3}\big)\big| > 1/\epsilon \Big).
\end{eqnarray*}
The two latter terms tend to zero with $n$, for any sufficiently small $\epsilon$. 
Indeed, the first probability tends to zero with $n$ by Lemma 4.9 in~\citep{berghaus2017weak}, and the second one may be arbitrarily small by $g_{\omega}$-regularity, choosing a sufficiently small $\epsilon$.
Therefore, all the terms of $A_{n,2}$ for which $I_2\neq \emptyset$ are negligible.
Moreover, if $I_2=\emptyset$, then $I_3\neq \emptyset$. By the stochastic equicontinuity of the process $\bar\Cb_n/\tilde g_\omega$ (Lemma 4.10 in~\citep{berghaus2017weak}), we have
\begin{eqnarray}
\lefteqn{  \Ic_{n,I_1,\emptyset,I_3}=
\int_{\BB_{n,|I_1|}} \frac{ \bar\Cb_n(\uu_{I_1}:\dd_{n,I_3})}{ \tilde g_\omega(\uu_{I_1}:\dd_{n,I_3})}
g_\omega (\uu_{I_1}:\dd_{n,I_3}) \, f\big(d\uu_{I_1}:\dd_{n,I_3}\big) }\label{InI1I3} \\
&=&
\int_{\BB_{n,|I_1|}} \frac{ \bar\Cb_n(\uu_{I_1}:\1_{I_3})}{ \tilde g_\omega(\uu_{I_1}:\1_{I_3})}
g_\omega (\uu_{I_1}:\dd_{n,I_3}) \, f\big(d\uu_{I_1}:\dd_{n,I_3}\big) + o_{P,u}(1), \nonumber
\end{eqnarray}
invoking again~(\ref{bound_Cn_1}), when $I_2=\emptyset$.
Re-indexing the subsets $I_j$, check that $A_{n,2}$ yields the sum in~(\ref{key_prop_1}) plus a negligible term.

\mds

The remaining term $r_{n2}=\Delta\big(\widehat\Cb_n f \big)\big( \BB_{n,d}\big)$ is a sum of $2^d$ terms. By $g_\omega$-regularity, all these terms are smaller than a constant times $ |\widehat \Cb_n|/g_\omega$, evaluated at a $d$-vector whose components are $1/2n$ or $1-1/2n$. 
This implies these terms are equal to $ |\bar \Cb_n|/g_\omega$ with the same arguments (Th. 4.5 in~\citep{berghaus2017weak}), plus a negligible term.
Due to Lemma 4.9 in~\citep{berghaus2017weak}, all of the latter terms tend to zero in probability,  
and then $r_{n2}=o_{P,u}(1)$.
Therefore, we have proven~(\ref{key_prop_1}) and point (i).

\mds
(ii) If, in addition, the process $(\alpha_n)$ is weakly convergent, then $(\bar\Cb_n/\tilde g_\omega)$ is weakly convergent to $(\Cb/\tilde g_\omega)$ in $\big(\ell^\infty([0,1]^d),\|\cdot\|_\infty\big)$, 
by Theorem 2.2 in~\citep{berghaus2017weak}.
For a given $f\in \Fc$, define the sequence of maps $g_n:\ell^{\infty}([0,1]^d) \rightarrow \Rb$ as
  \begin{eqnarray}
  \lefteqn{ g_n(h):= (-1)^{d}\int_{\BB_{n,d}} h(\uu) g_\omega(\uu) \, f\big(d\uu\big)
    \label{decomp_gnhn} }\\
&+&
\sum_{ \substack{I \subset \{1,\ldots,d\} \\ I\neq \emptyset, I\neq \{1,\ldots,d\} } } (-1)^{|I|}
  \int_{\BB_{n,|I|}} h(\uu_I:\1_{-I}) g_\omega(\uu_{I}:\dd_{-I}) \, f\big(d\uu_{I}:\dd_{-I}\big). \nonumber
  \end{eqnarray}
If a sequence of maps $(h_n)$ tends to $h_\infty$ in $\ell^{\infty}([0,1]^d)$ and $h_\infty$ is continuous on $[0,1]^d$, let us prove that $g_n(h_n)\rightarrow g_\infty(h_\infty)$, where
\begin{eqnarray*}
\lefteqn{ g_\infty(h):= (-1)^d\int_{(0,1)^d} h(\uu) g_\omega(\uu) \, f(d\uu)    }\\
&+&
\sum_{ \substack{I \subset \{1,\ldots,d\} \\ I\neq \emptyset, I\neq \{1,\ldots,d\} } } (-1)^{|I|}
  \int_{(0,1)^{|I|}} h(\uu_I:\1_{-I}) g_\omega(\uu_{I}:\1_{-I}) \, f\big(d\uu_{I}:\1_{-I}\big).
  \end{eqnarray*}
The difference $g_n(h_n)-g_\infty(h_\infty)$ is a sum of $2^d-1$ differences between integrals that come from~(\ref{decomp_gnhn}).
The first one is managed as
\begin{eqnarray*}
\lefteqn{  \big|\int_{\BB_{n,d}} h_n(\uu) g_\omega(\uu) \, f\big(d\uu\big) - \int_{(0,1)^d} h_\infty(\uu) g_\omega(\uu) \, f(d\uu) \big|
\leq \|h_n - h_\infty \|_\infty \int_{\BB_{n,d}}  g_\omega(\uu)\, \left| f(d\uu)\right|    }\\
&+& \big| \int_{\BB_{n,d}} h_\infty(\uu) g_\omega(\uu)\, f(d\uu)- \int_{(0,1)^{d}} h_\infty(\uu) g_\omega(\uu)\, f(d\uu) \big|,\hspace{5cm}
\end{eqnarray*}
that tends to zero by~(\ref{bound_fg}) and Assumption~\ref{regularity_cond_wc}.
The other terms of $g_n(h_n)-g_\infty(h_\infty)$ are indexed by a subset $I$, and can be bounded similarly:
{\small{\begin{eqnarray*}
\lefteqn{  \big|\int_{\BB_{n,|I|}} h_n(\uu_I:\1_{-I}) g_\omega(\uu_I:\dd_{n,-I}) \, f\big(d\uu_{I}:\dd_{n,-I}\big) -
\int_{(0,1)^{|I|}} h_\infty(\uu_I:\1_{-I}) g_\omega(\uu_I:\1_{-I}) \, f(d\uu_I:\1_{-I})) \big|  }\\
&\leq & \|h_n - h_\infty \|_\infty \int_{\BB_{n,I}}  g_\omega(\uu_I:\dd_{n,-I})\, \left| f(d\uu_I:\dd_{n,-I})\right|    \\
&+& \big| \int_{\BB_{n,I}} h_\infty(\uu_I:\1_{-I}) g_\omega(\uu_{-I}:\dd_{n,-I})\, f(d\uu_I:\dd_{n,-I})
- \int_{(0,1)^{|I|}} h_\infty(\uu_I:\1_{-I}) g_\omega(\uu_I:\1_{-I})\,
f(d\uu_I:\1_{-I}) \big|,
\end{eqnarray*}}}
that tends to zero by Equation~(\ref{bound_Cn_1}) and Assumption~\ref{regularity_cond_wc}.

\mds

Therefore, apply the extended continuous mapping (Theorem 1.11.1 in~\citep{van1996weak}) to obtain the weak convergence of
$ g_n(\bar\Cb_n/\tilde g_\omega)$ (that is equal to $g_n(\bar\Cb_n/g_\omega)$ in our case) towards $g_\infty( \Cb/\tilde g_\omega)$ in $\ell^{\infty}([0,1]^d)$.
Note that almost every trajectory of $\Cb/\tilde g_\omega$ on $[0,1]^d$ is continuous.
Since $\int f\, d\widehat \Cb_n=g_n(\bar\Cb_n/\tilde g_\omega) + o_{P,u}(1)$, this proves the announced weak convergence result (ii).

\mds

(iii) Our arguments are close to those invoked to prove Theorem 1 in~\citep{radulovic2017weak}.
Our point (ii) above yields the finite-dimensional convergence of $\widehat\Cb_n$ in $\ell^{\infty}(\Fc)$.
For any (possibly random) map $X:[0,1]^d\rightarrow \Rb$ and any $f\in \Fc$, set
\begin{eqnarray*}
\lefteqn{\Gamma_\infty (X,f)}\\
& := & (-1)^d\int_{(0,1)^d} (X g_\omega)(\uu) \, f(d\uu)+
\sum_{ \substack{I \subset \{1,\ldots,d\} \\ I\neq \emptyset, I\neq \{1,\ldots,d\} } } (-1)^{|I|}
  \int_{(0,1)^{|I|}} (X g_\omega)(\uu_{I}:\1_{-I}) \, f\big(d\uu_{I}:\1_{-I}\big).
\end{eqnarray*}
Moreover, define
\begin{eqnarray*}
\lefteqn{ \Gamma_n(X,f):=
(-1)^{d}\int_{\BB_{n,d}} X(\uu) g_\omega(\uu)\,  f\big(d\uu\big) \nonumber }\\
  &+& \sum_{ \substack{I \subset \{1,\ldots,d\} \\ I\neq \emptyset, I\neq \{1,\ldots,d\} } } (-1)^{|I|}
  \int_{\BB_{n,|I|}} X(\uu_{I}:\dd_{n,-I}) g_\omega(\uu_{I}:\dd_{n,-I})\, f\big(d\uu_{I}:\dd_{n,-I}\big).
\label{def_Gamma_n}
\end{eqnarray*}
We have proved above (recall Equation~(\ref{key_prop_1}) and~(\ref{InI1I3})) that, for any $f\in \Fc$,
$$ \widehat \Cb_n(f):=\int f\, d\widehat \Cb_n = \Gamma_n\Big( \frac{\bar \Cb_n}{\tilde g_\omega},f \Big)+ o_{P,u}(1).$$
Therefore, we expect that the weak limit of $\widehat \Cb_n$ in $\ell^\infty (\Fc)$ will be
$ \Gamma_\infty\big( \Cb/\tilde g_\omega,. \big)$.
It is sufficient to prove that $\Gamma_n\big( \bar \Cb_n/\tilde g_\omega,\cdot\big)$ weakly tends to the latter process.

\mds

To this end, we slightly adapt our notations to deal with functionals defined on $\Fc$.
The weak limit of $\widehat \Cb_n/\tilde g_\omega$ on $\ell^\infty([0,1]^d)$ is the Gaussian process $\Cb / \tilde g_\omega$, that is tight (Ex. 1.5.10 in~\citep{van1996weak}).
Moreover, define the map $\widetilde\Gamma_\infty: C_0([0,1]^d,\|\cdot \|_\infty)\rightarrow \ell^\infty(\Fc)$ as
$ \widetilde \Gamma_\infty (X)(f)=\Gamma_\infty(X,f)$, where $C_0([0,1]^d,\|\cdot \|_\infty)$ denotes the set on continuous maps on $[0,1]^d$, endowed with the sup-norm.
Similarly, define $\widetilde \Gamma_n:
\ell^\infty([0,1]^d,\|\cdot \|_\infty)\rightarrow \ell^\infty(\Fc)$ as
$ \widetilde \Gamma_n (X)(f)=\Gamma_n(X,f)$. We now have to prove that
$\widetilde \Gamma_n\big( \bar \Cb_n/ \tilde g_\omega \big)$ weakly tends to $\widetilde \Gamma_\infty \big( \Cb/ \tilde g_\omega \big)$
on $\ell^\infty(\Fc)$.

\mds

\textcolor{black}{First, we prove} that $\widetilde\Gamma_\infty$ is continuous. Let $(X_n)$ be a sequence of maps in $C_0([0,1]^d,\|\cdot \|_\infty)$ that tends to $X$ is the latter space.
We want to prove that $\widetilde\Gamma_\infty (X_n)$ tends to $\widetilde\Gamma_\infty(X)$ in $\ell^\infty(\Fc)$.
The first term of $\widetilde\Gamma_\infty (X_n)- \widetilde\Gamma_\infty (X) $ that comes from the definition of $\Gamma_\infty$ is easily managed:
$$ \sup_{f\in \Fc}\big|\int_{(0,1)^d}  (X_n g_\omega)(\uu)\, f(d\uu) - \int_{(0,1)^d}  (X g_\omega)(\uu)\, f(d\uu) \big|
\leq \| X_n - X \|_\infty \sup_{f\in \Fc} \int_{(0,1)^d}  g_\omega(\uu)\, \left|f(d\uu)\right|,$$
that tends to zero because of~(\ref{bound_fg}).
The other terms are tackled similarly:
\begin{eqnarray*}
\lefteqn{
\sup_{f\in \Fc}
\big|  \int_{(0,1)^{|I|}} (X_n g_\omega)(\uu_{I}:\1_{-I}) \, f\big(d\uu_{I}:\1_{-I}\big)-
\int_{(0,1)^{|I|}} (X g_\omega)(\uu_{I}:\1_{-I}) \, f\big(d\uu_{I}:\1_{-I}\big)\big|   }\\
&\leq &\| X_n- X \|_\infty \sup_{f\in \Fc}
  \int_{(0,1)^{|I|}}  g_\omega(\uu_{I}:\1_{-I}) \, | f\big(d\uu_{I}:\1_{-I}\big) | , \hspace{5cm}
\end{eqnarray*}
that tends to zero.
We have used the fact that, due to~(\ref{bound_Cn_1}) and Assumption~\ref{regularity_cond_wc_stronger}, we have 
$$  \sup_{f\in \Fc} \int_{(0,1)^{|I|}}  g_\omega(\uu_{I}:\1_{-I}) \, | f\big(d\uu_{I}:\1_{-I}\big) | <\infty. $$
As a consequence, $\widetilde \Gamma_\infty(X_n)$ tends to $\widetilde \Gamma_\infty(X)$ in $\ell^\infty(\Fc)$.
Therefore, by continuity, the expected weak limit $\widetilde\Gamma_{\infty}( \Cb/\tilde g_\omega)$ of $\widetilde\Gamma_{\infty}( \bar\Cb_n/\tilde g_\omega)$ is tight on $\ell^\infty(\Fc)$.

\mds

Then, the weak convergence of $\Gb_n:=\widetilde \Gamma_n\big( \bar \Cb_n/\tilde g_\omega \big)$ towards $\Gb_\infty:=\widetilde\Gamma_\infty( \Cb/\tilde g_\omega)$ in $\ell^{\infty}(\Fc)$ is obtained if we prove that the bounded Lipschitz distance between the two processes tends to zero with $n$ (Th. 1.12.4 in~\citep{van1996weak}), i.e. if
\begin{equation*}
d_{BL}\big( \Gb_n,\Gb_\infty    \big)=\sup_h \big| \Eb\big[ h(\Gb_n) \big]-  \Eb\big[ h(\Gb_\infty) \big] \big| \underset{n \rightarrow \infty}{\longrightarrow} 0,
\end{equation*}
with the supremum taken over all the uniformly bounded and Lipschitz maps $h:\ell^\infty(\Fc)\rightarrow \Rb$,
$\sup_{x\in\ell^\infty(\Fc)} |h(x)|\leq 1 $ and $ |h(x)-h(y)|\leq \| x-y\|_\infty$ for all $x,y\in \ell^\infty(\Fc)$.
By the triangle inequality, we have
\begin{eqnarray*}
\lefteqn{
d_{BL}\big( \Gb_n,\Gb_\infty    \big)=
d_{BL}\bigg( \widetilde \Gamma_n\Big( \frac{\bar \Cb_n}{\tilde g_\omega} \Big),\widetilde \Gamma_\infty\Big( \frac{\Cb}{\tilde g_\omega} \Big)    \bigg)  }\\
&\leq &
d_{BL}\bigg( \widetilde \Gamma_n\Big( \frac{\bar \Cb_n}{\tilde g_\omega} \Big),\widetilde \Gamma_n\Big( \frac{ \Cb}{\tilde g_\omega} \Big)    \bigg)
+ d_{BL}\bigg( \widetilde \Gamma_n\Big( \frac{\Cb}{\tilde g_\omega} \Big),\widetilde \Gamma_\infty\Big( \frac{ \Cb}{\tilde g_\omega} \Big)    \bigg)=: d_{1,n} + d_{2,n}.
\end{eqnarray*}
To deal with $d_{1,n}\rightarrow 0$, note that
\begin{eqnarray*}
\lefteqn{
\|\widetilde \Gamma_n\big( \bar \Cb_n/\tilde g_\omega \big) - \widetilde \Gamma_n\big(  \Cb/\tilde g_\omega \big)\|_\infty =
\sup_{f\in \Fc} \big|\Gamma_n\big( \bar \Cb_n/\tilde g_\omega,f \big) - \Gamma_n\big(  \Cb/\tilde g_\omega, f \big)\big| }\\
 &\leq &
\| \frac{\bar \Cb_n}{\tilde g_\omega} - \frac{\Cb}{\tilde g_\omega} \|_\infty
\sum_{I\neq \emptyset } \sup_{f\in \Fc}  \int_{\BB_{n,|I|} } g_\omega(\uu_{I}:\dd_{n,-I}) \,\big|  f\big(d\uu_{I}:\dd_{n,-I}\big) \big| \\
 &\leq &
M \| \frac{\bar \Cb_n}{\tilde g_\omega} - \frac{\Cb}{\tilde g_\omega}\|_\infty,
\end{eqnarray*}
for some positive constant $M$, due to~(\ref{bound_Cn_1}).
This proves that $\widetilde \Gamma_n$ is Lipschitz, with a Lipschitz constant that does not depend on $n$ nor $f\in \Fc$.
Therefore, we get
$$ d_{1,n}= \sup_h \big| \Eb\big[ h\circ \widetilde \Gamma_n \big( \frac{\bar\Cb_n}{\tilde g_\omega}\big)  \big]
- \Eb\big[ h\circ \widetilde \Gamma_n \big( \frac{\Cb}{\tilde g_\omega}\big)  \big] \big|
\leq  M \, d_{BL}\Big( \frac{\bar\Cb_n}{\tilde g_\omega},\frac{\Cb}{\tilde g_\omega}\Big),$$
that tends to zero because of the weak convergence of $(\bar \Cb_n/\tilde g_\omega)$
 to $(\Cb/\tilde g_\omega)$ in $\ell^\infty(|0,1]^d)$ (Th. 4.5 in~\citep{berghaus2017weak}).

 \mds

To show that $(d_{2,n})$ tends to zero, note that, for every $I\neq \emptyset$, we have
\begin{eqnarray*}
\lefteqn{
 \big|  \int_{\BB_{n,|I|} } \Cb(\uu_{I}:\dd_{n,-I}) \, f\big(d\uu_{I}:\dd_{n,-I}\big) - \int_{ (0,1)^{|I|} } \Cb(\uu_{I}:\1_{-I}) \, f\big(d\uu_{I}:\1_{-I}\big) \big| }\\
&\leq &
 \int_{ \BB_{n,|I|}} \big|  \frac{\Cb}{g_\omega}(\uu_{I}:\dd_{n,-I}) - \frac{\Cb}{\tilde g_\omega}(\uu_{I}:\1_{-I})\big|
 \, g_\omega\big(\uu_{I}:\dd_{n,-I}\big) \, \big|f\big(d\uu_{I}:\dd_{n,-I}\big) \big|   \\
 &+&
\big| \int_{ \BB_{n,|I|} }  \frac{\Cb}{\tilde g_\omega}(\uu_{I}:\1_{-I})
 \, g_\omega\big(\uu_{I}:\dd_{n,-I}\big) f\big(d\uu_{I}:\dd_{n,-I}\big) \\
 &-& \int_{ (0,1)^d}   \frac{\Cb}{\tilde g_\omega}(\uu_{I}:\1_{-I})
 \, \tilde g_\omega\big(\uu_{I}:\1_{-I}\big) f\big(d\uu_{I}:\1_{-I}\big)  \big| =: e_{1,n}(f)+e_{2,n}(f).
\end{eqnarray*}
Clearly, $\sup_{f\in \Fc} e_{1,n}$ tends to zero, invoking~(\ref{bound_Cn_1})
and the continuity of $\Cb/\tilde g_\omega$ on $(0,1]^d$.
By Assumption~\ref{regularity_cond_wc_stronger}, $\sup_{f\in \Fc} e_{2,n}$ tends to zero a.s.
Thus, we have proved that
\begin{eqnarray*}
\lefteqn{ 
\| \widetilde \Gamma_n\Big( \frac{\Cb}{\tilde g_\omega} \Big)  - \widetilde \Gamma_{\textcolor{black}{\infty}} \Big( \frac{\Cb}{\tilde g_\omega} \Big) \|_\infty 
=\sup_{f\in \Fc} \big| \Gamma_n\Big( \frac{\Cb}{\tilde g_\omega},f \Big)  - \Gamma \Big( \frac{\Cb}{\tilde g_\omega},f \Big) \big|    }\\
& \leq &
\sum_{I\neq \emptyset} \sup_{f\in \Fc} \big|
 \int_{\BB_{n,|I|} } \Cb(\uu_{I}:\dd_{n,-I}) \, f\big(d\uu_{I}:\dd_{n,-I}\big) - \int_{ (0,1)^{|I|} } \Cb(\uu_{I}:\1_{-I}) \, f\big(d\uu_{I}:\1_{-I}\big) \big|  
\end{eqnarray*}
tends to zero for almost every trajectory and when $n\rightarrow \infty$,
Considering the bounded Lipschitz maps $h$ as in the definition of $d_{BL}$, deduce
\begin{eqnarray*}
\lefteqn{ 
 d_{2,n}=\sup_h \big| \Eb\Big[ h \circ \widetilde \Gamma_n\Big( \frac{\Cb}{\tilde g_\omega} \Big)  - h \circ \widetilde \Gamma_{\textcolor{black}{\infty}} \Big( \frac{\Cb}{\tilde g_\omega} \Big)     \Big] \big|  \leq \Eb\Big[ \sup_h \big|  h \circ \widetilde \Gamma_n\Big( \frac{\Cb}{\tilde g_\omega} \Big)  - h \circ \widetilde \Gamma_{\textcolor{black}{\infty}} \Big( \frac{\Cb}{\tilde g_\omega} \Big)  \big|   \Big]   }\\
&\leq & 
\Eb\Big[ 
\| \widetilde \Gamma_n\Big( \frac{\Cb}{\tilde g_\omega} \Big)  - \widetilde \Gamma_{\textcolor{black}{\infty}} \Big( \frac{\Cb}{\tilde g_\omega} \Big) \|_\infty\textcolor{black}{\wedge 2}    \Big]=: \Eb[V_n].\hspace{6cm}
\end{eqnarray*}
Since $\sup_{x\in\ell^\infty(\Fc)} |h(x)|\leq 1 $ for every $h$, the sequence $V_n$ is bounded by two. 
And we have proved above that $V_n$ tends to zero a.s. 
Thus, the dominated convergence theorem implies that $\Eb[V_n]\rightarrow 0$ when $n\rightarrow \infty$, i.e. $d_{2,n} \rightarrow 0$ when $n\rightarrow \infty$.

\mds

To conclude, we have proved that $d_{BL}\big( \Gb_n,\Gb_\infty    \big) \rightarrow 0$ and $n\rightarrow \infty$. Since the limit $\Gb_\infty$ is tight, we get the weak convergence of
$\Gb_n$ to $\Gb_\infty$ in $\ell^\infty(\Fc)$, i.e. the weak convergence of $\widehat \Cb_n$ indexed by $f\in \Fc$, i.e. in $\ell^\infty(\Fc)$, as announced.

\subsection{Proof of Corollary~\ref{cor_GC}}

By inspecting the proof of Theorem~\ref{Th_fondam_multiv_rank_stat} (i), it appears that it is sufficient to prove
$$ \sup_{f\in\Fc} \Big| \int_{\BB_{n,|I|}} \bar\Cb_n(\uu_{I}:\dd_{n,-I}) \, f\big(d\uu_{I}:\dd_{n,-I}\big)\Big|=O_P(\mu_n),$$
for any $I\subset \{1,\ldots,d\}$, $I\neq \emptyset$. For any constant $A>0$, we have
  \begin{eqnarray}
  \lefteqn{\Pb\Big(\sup_{f\in \Fc}\big| \int_{\BB_{n,|I|}} \bar\Cb_n(\uu_{I}:\dd_{n,-I}) \, f\big(d\uu_{I}:\dd_{n,-I}\big) \big| > A\mu_n \Big) \leq
   \Pb\Big( \sup_{\uu_I \in \BB_{n,|I|}}\frac{ |\bar\Cb_n(\uu_{I}:\dd_{n,-I}) |  }{g_\omega(\uu_I:\dd_{n,-I})} > \sqrt{A}\mu_n \Big)  } \nonumber\\
&+& \Pb\Big( \sup_{f\in \Fc} \int_{\BB_{n,|I|}} g_\omega(\uu_{I}:\dd_{n,-I}) \, \big| f\big(d\uu_{I}:\dd_{n,-I}\big) \big| > \sqrt{A} \Big).
\hspace{5cm}
\label{GC_remainder_term}
\end{eqnarray}
Check that
$\{ (\uu_I:\dd_{n,-I})\, | \,  \uu_I\in \BB_{n,|I|} \} \subset N(\delta_n,1/2)$,
for any sequence of positive numbers $(\delta_n)$, $\sup_n\delta_n < 1/2$ and $\delta_n\rightarrow 0$ with $n$. This yields
  \begin{equation*}
   \Pb\Big( \sup_{\uu_I \in \BB_{n,|I|}}\frac{ |\bar \Cb_n(\uu_{I}:\dd_{n,-I}) |}{g_\omega(\uu_{I}:\dd_{n,-I})} > \sqrt{A}\mu_n \Big)
   \leq \Pb\Big( \sup_{(\uu_I:\dd_{n,-I}) \in N(\delta_n,1/2)}\frac{ |\bar \Cb_n(\uu_{I}:\dd_{n,-I})|}{g_\omega(\uu_{I}:\dd_{n,-I})} > \sqrt{A}\mu_n \Big).
\end{equation*}
Note that $g_\omega(\uu)\geq \delta_n^\omega$ when $\uu\in N(\delta_n,1/2)$,
and set $\delta_n:=\mu_n^{-1/\omega}$.
Thus, for $n$ sufficiently large, $\delta_n <1/2$ and we get
{\small{\begin{equation*}
\Pb\Big( \sup_{(\uu_I:\dd_{n,-I}) \in N(\delta_n,1/2)}\frac{ |\bar \Cb_n(\uu_{I}:\dd_{n,-I})|}{g_\omega(\uu_{I}:\dd_{n,-I})} > \sqrt{A}\mu_n \Big)
\leq \Pb\Big( \sup_{(\uu_I:\dd_{n,-I}) \in N(\delta_n,1/2)} |\bar \Cb_n(\uu_{I}:\dd_{n,-I})| > \sqrt{A} \Big). 
\end{equation*}}}
Since $\sup_{\uu\in [0,1]^d}|\alpha_n(\uu)|=O_P(1)$, then $\sup_{\uu\in [0,1]^d}|\bar\Cb_n(\uu)|=O_P(1)$ too because all partial derivatives $\dot C_j(\uu)$, $j\in \{1,\ldots,d\}$, belong to $[0,1]$ (Th. 2.2.7 in~\citep{nelsen2006introduction}).
Therefore, the first term on the r.h.s. of~(\ref{GC_remainder_term}) tends to zero. Finally, the second term on the r.h.s. of~(\ref{GC_remainder_term}) may be arbitrarily
small with a large $A$, due to~(\ref{bound_Cn_1}), proving the result.

\section{Additional proofs}\label{proof_results_sparsity}

\subsection{Proof of Theorem \ref{bound_proba}}

We denote $\nu_n = \ln(\ln n) n^{-1/2}+a_{n}$ and we would like to prove that, for any $\eps>0$, there exists $L_{\eps} > 0$ such that, for any $n$, we have
\begin{equation} \label{objective_bound_proba}
\Pb\Big(\|\widehat{\theta}-\theta_0\|_2/\nu_n \geq L_{\eps}\Big)  < \eps.
\end{equation}
Now, following the reasoning of Fan and Li (2001), Theorem 1, and denoting the penalized loss $\Lb^{\text{pen}}_n(\theta;\widehat{\Uc}_n) = \Lb_n(\theta;\widehat{\Uc}_n) + n\sum^p_{k=1}\pp(\lambda_n,|\theta_{k}|)$, we have
\begin{equation}
\Pb\Big(\|\widehat{\theta}-\theta_0\|_2/\nu_n \geq L_{\eps}\Big) \leq \Pb\Big(\exists \mathbf{v} \in \Rb^{p},\|\mathbf{v}\|_2=L'_{\eps}\geq L_\eps: \Lb^{\text{pen}}_n(\theta_0+\nu_n \mathbf{v};\widehat{\Uc}_n) \leq \Lb^{\text{pen}}_n(\theta_0;\widehat{\Uc}_n)\Big),
\label{ineg1_FanLi}
\end{equation}
and we can always impose $L'_\eps=L_\eps$, our choice hereafter.
If the r.h.s. of~(\ref{ineg1_FanLi}) is smaller than $\eps$,
there is a local minimum in the ball $\big\{\theta_0+\nu_n \mathbf{v}, \|\mathbf{v}\|_2 \leq L_{\eps}\big\}$ with a probability larger than $1-\eps$.
In other words,~(\ref{objective_bound_proba}) is satisfied and $\|\widehat{\theta}-\theta_0\|_2 = O_p(\nu_n)$. Now, by a Taylor expansion of the penalized loss function around the true parameter, we get
\begin{eqnarray*}
\lefteqn{\Lb^{\text{pen}}_n(\theta_0+\nu_n\mathbf{v};\widehat{\Uc}_n)-\Lb^{\text{pen}}_n(\theta_0;\widehat{\Uc}_n)}\\
& \geq & \nu_n \mathbf{v}^\top \nabla_{\theta} \Lb_n(\theta_0;\widehat{\Uc}_n) + \frac{\nu^2_n}{2} \mathbf{v}^\top \nabla^2_{\theta \theta^\top} \Lb_n(\theta_0;\widehat{\Uc}_n)\mathbf{v}
+  \frac{\nu^3_n}{6} \nabla_\theta \left\{ \mathbf{v}^\top \nabla^2_{\theta \theta^\top} \Lb_n(\overline{\theta};\widehat{\Uc}_n)\mathbf{v} \right\} \mathbf{v}
\\
&+&   n\overset{}{\underset{k \in \Ac}{\sum}}\left\{\pp(\lambda_n,|\theta_{0,k}+\nu_n v_k|)-\pp(\lambda_n,|\theta_{0,k}|)\right\}=:\sum_{j=1}^4 T_j,
\end{eqnarray*}
for some parameter $\bar\theta$ such that $\|\overline{\theta}-\theta_0\|_2 \leq L_{\eps}\nu_n$.
Note that we have used $\pp(\lambda_n,0)=0$ and the positiveness of the penalty.
Thus, it is sufficient to prove there exists $L_{\eps}$ such that
\begin{equation}
\Pb\left(\exists \mathbf{v} \in \Rb^{p}, \|\mathbf{v}\|_2=L_{\eps}:
T_1+\ldots+T_4 \leq 0\right) < \eps. \label{bound_obj_double}
\end{equation}
Let us deal with the non-penalized quantities. First, for any $k\in \{1,\ldots,p\}$, we have
\begin{equation*}
\partial_{\theta_k}\Lb_n(\theta_0;\widehat{\Uc}_n) =
n\int_{(0,1]^d}\partial_{\theta_k}\ell(\theta_0;\uu) \text{d}\big(C_n(\uu)-C(\uu)\big),
\end{equation*}
due to the first-order conditions.
For any $\theta_0 \in \Theta$,
we have assumed that the family of maps
$\Fc_0$ is $g_\omega$-regular. Moreover, Assumption~\ref{oscillation_modulus_assump} and the compacity of $[0,1]^d$ implies
$\|\alpha_n\|_\infty=O_P(1)$.
Then, we can apply Corollary~\ref{cor_GC}, that yields
$$ \sup_{k=1,\ldots,p}\big| \partial_{\theta_k} \Lb_n(\theta_0;\widehat{\Uc}_n) \big|=
n\sup_{k=1,\ldots,p}\big| \int \partial_{\theta_k} \ell(\theta_0;\uu)\, (C_n-C)(d\uu) \big| = O_P(\ln(\ln n)\sqrt{n}).$$
By Cauchy-Schwarz, we deduce
\begin{equation*}
|T_1|\textcolor{black}{=} \nu_n\big|\mathbf{v}^\top \nabla_{\theta} \Lb_n(\theta_0;\widehat{\Uc}_n)\big| \leq \nu_n \|\mathbf{v}\|_2 \|\nabla_{\theta} \Lb_n(\theta_0;\widehat{\Uc}_n)\|_2 =
 O_p\big(\nu_n\sqrt{n p} \ln(\ln n) \big)\|\mathbf{v}\|_2.
\end{equation*}

The empirical Hessian matrix can be expanded as
{\small{\begin{equation*}
n^{-1}\mathbf{v}^\top \nabla^2_{\theta \theta^\top} \Lb_n(\theta_0;\widehat{\Uc}_n)\mathbf{v} =
\int_{(0,1]^d} \mathbf{v}^\top\nabla_{\theta\theta^\top} \ell(\theta_0;\uu)\mathbf{v}\,C_n(d\uu)
=  \sum_{k,l=1}^p v_k v_l
\int_{(0,1]^d} \partial^2_{\theta_k\theta_l} \ell(\theta_0;\uu)\,C_n(d\uu).
\end{equation*}}}

We have assumed that the maps $ \uu \mapsto \partial^2_{\theta_k \theta_l} \ell(\theta_0;\uu) $, $(k,l)\in \{1,\ldots,p\}^2$,
belong the a $g_\omega$-regular family. Therefore, applying Corollary~\ref{cor_GC}, we get
{\small{\begin{eqnarray*}
\lefteqn{\int_{(0,1]^d} \partial^2_{\theta_k\theta_l} \ell(\theta_0;\uu)\,C_n(d\uu) = \int_{(0,1]^d} \partial^2_{\theta_k\theta_l} \ell(\theta_0;\uu)\,C(d\uu)
+ \int_{(0,1]^d} \partial^2_{\theta_k\theta_l} \ell(\theta_0;\uu)\,(C_n-C)(d\uu) }\\
&=& \int_{(0,1]^d} \partial^2_{\theta_k\theta_l} \ell(\theta_0;\uu)\,C(d\uu) + O_P(\ln(\ln n)/\sqrt{n}).\hspace{5cm}
\end{eqnarray*}}}
As a consequence, since $\| \mathbf{v} \|_1^2\leq p \| \mathbf{v} \|_2^2$, this yields
\begin{equation*}
T_2=\frac{n\nu^2_n}{2} \mathbf{v}^\top\Eb\big[ \nabla_{\theta \theta^\top}^2 \ell(\theta_0;\UU) \big]\mathbf{v} + O_P\big(p\|\mathbf{v}\|_2^2 \nu_n^2\ln(\ln n)\sqrt{n}\big),
\end{equation*}
that is positive by assumption, when $n$ is sufficiently large and for a probability arbitrarily close to one.
By similar arguments with the family of maps $\Fc_3$, we get
\begin{equation*}
T_3=  n\frac{\nu^3_n}{6} \nabla_\theta \big\{ \mathbf{v}^\top \nabla^2_{\theta \theta^\top} \Eb\big[ \ell(\overline{\theta};\UU)\big]\mathbf{v} \big\} \mathbf{v}
+ O_P\big(p^{3/2}\|\mathbf{v}\|_2^3\nu_n^3\ln(\ln n)\sqrt{n}\big)=O_P(n p^{3/2}\|\mathbf{v}\|_2^3\nu_n^3).
\end{equation*}
Let us now treat the penalty part as in~\citep{fan2004nonconcave} (proof of Theorem 1, equations (5.5) and (5.6)).
By using exactly the same method, we obtain
\begin{equation*}
 |T_4|\leq \sqrt{|\Ac|} n\nu_n a_n \|\mathbf{v}\|_2 + 2 n b_n \nu_n^2 \|\mathbf{v}\|_2^2,
\end{equation*}
and the latter term is dominated by $T_2$, allowing $\| \mathbf{v}\|$ to be large enough.
Thus, for such $\mathbf{v}$, we have
$$ \sum_{j=1}^4 T_j = \frac{n\nu^2_n}{2} \mathbf{v}^\top\Eb\big[ \nabla_{\theta \theta^\top}^2 \ell(\theta_0;\UU) \big]\mathbf{v} \big(1+o_P(1)\big).$$
Since the latter dominant term is larger than
$n L^2_{\eps}\lambda_{\min}(\Hb)\nu^2_n/2>0$ for $n$ large enough, where $\lambda_{\min}(\Hb)$ denotes the smallest eigenvalue of $\Hb$, we deduce~(\ref{bound_obj_double}) and finally $\|\widehat{\theta}-\theta_0\|_2=O_p(\nu_n)$.

\subsection{Proof of Theorem \ref{oracle_property}}

\emph{Point (i):} The proof is performed in the same spirit as in Fan and Li (2001).
Consider an estimator $\widehat\theta := (\widehat{\theta}^\top_{\Ac},\widehat{\theta}^\top_{\Ac^c})^\top$ of $\theta_0$ such that $\|\widehat\theta - \theta_0\|_2=O_P(\nu_n)$, as in Theorem \ref{bound_proba}, with
$\nu_n:=n^{-1/2}\ln(\ln n)+a_n$. Using our notations for vector concatenation, as detailed in Appendix~\ref{technicalities}, the support recovery property holds asymptotically if
\begin{equation}
\Lb^{\text{pen}}_n(\widehat{\theta}_{\Ac}:\mathbf{0}_{\Ac^c};\widehat{\Uc}_n) = \underset{\|\theta_{\Ac^c}\|_2\leq C \nu_n}{\min} \Lb^{\text{pen}}_n(\widehat{\theta}_{\Ac}:\theta_{\Ac^c};\widehat{\Uc}_n),
\label{CS_support_recovery}
\end{equation}
for any constant $C>0$ with a probability that tends to one with $n$. Set $\eps_n := C \nu_n$. To prove~(\ref{CS_support_recovery}), it is sufficient to show that, for any $\theta\in \Theta$ such that $\|\theta-\theta_0\|\leq \eps_n$, we have
with a probability that tends to one
\begin{equation}
\label{sign_consistency}
\partial_{\theta_j} \Lb^{\text{pen}}_n(\theta;\widehat{\Uc}_n)>0 \;\; \text{when} \;\; 0 < \theta_j < \eps_n; \;
\partial_{\theta_j} \Lb^{\text{pen}}_n(\theta;\widehat{\Uc}_n)<0 \;\; \text{when} \;\; -\eps_n < \theta_j < 0,
\end{equation}
for any $j \in \Ac^c$. By a Taylor expansion of the partial derivative of the penalized loss around $\theta_0$, we obtain
\begin{eqnarray*}
\lefteqn{\partial_{\theta_j} \Lb^{\text{pen}}_n(\theta;\widehat{\Uc}_n) = \partial_{\theta_j} \Lb_n(\theta;\widehat{\Uc}_n) + n \partial_{2}\pp(\lambda_n,|\theta_j|)\text{sgn}(\theta_j)}\\
& = & \partial_{\theta_j} \Lb_n(\theta_0;\widehat{\Uc}_n) + \overset{p}{\underset{l=1}{\sum}} \partial^2_{\theta_j\theta_l} \Lb_n(\theta_0;\widehat{\Uc}_n) (\theta_l-\theta_{0,l})+  
\frac{1}{2}\overset{p}{\underset{l,m=1}{\sum}} \partial^3_{\theta_j\theta_l\theta_m} \Lb_n(\overline{\theta};\widehat{\Uc}_n) (\theta_l-\theta_{0,l})(\theta_m-\theta_{0,m}) \\
& +&  n \partial_{2}\pp(\lambda_n,|\theta_j|)\text{sgn}(\theta_j),
\end{eqnarray*}
for some $\overline{\theta}$ that satisfies $\|\overline{\theta}-\theta_0\|_2 \leq \eps_n$.
The family of maps $\Fc$ is $g_{\omega}$-regular and $\|\alpha_n\|_{\infty}=O_p(1)$. As a consequence, by Corollary~\ref{cor_GC},
\begin{equation*}
\big|\partial_{\theta_j} \Lb_n(\theta_0;\widehat{\Uc}_n)\big| = O_p\big(\ln(\ln n) \sqrt{n}\big) .
\end{equation*}
As for the second order term, the maps $\uu \mapsto \partial^2_{\theta_i \theta_l} \ell(\theta_0;\uu)$ 
are $g_{\omega}$-regular by assumption, for any $(i,l) \in \Ac^c \times \{1,\ldots,p\}$. Then, by Corollary~\ref{cor_GC}, we deduce
\begin{eqnarray*}
\frac{1}{n}\partial^2_{\theta_j\theta_l} \Lb_n(\theta_0;\widehat{\Uc}_n) = \int_{(0,1]^d} \partial^2_{\theta_j\theta_l}\ell(\theta_0;\uu)C_n(d\uu) =  \Eb[\partial^2_{\theta_j\theta_l}\ell(\theta_0;\UU)] + O_p(\ln(\ln n)/\sqrt{n}).
\end{eqnarray*}
Finally, for the remaining third order term, since the family of maps $\uu \mapsto \partial^3_{\theta_j\theta_l \theta_m}\ell(\overline{\theta};\uu)$ is $g_{\omega}$-regular by assumption, Corollary~\ref{cor_GC} yields
\begin{equation*}
\frac{1}{n}\partial^3_{\theta_j\theta_l\theta_m} \Lb_n(\overline{\theta};\widehat{\Uc}_n) = \Eb[\partial^3_{\theta_j\theta_l\theta_m} \ell(\overline{\theta};\UU)] + O_p\big(\ln(\ln n)/\sqrt{n}\big),
\end{equation*}
that is bounded in probability by Assumption~\ref{regularity_assumption}.
Hence putting the pieces together and using the Cauchy-Schwarz inequality, we get
\begin{eqnarray}
\lefteqn{\partial_{\theta_j} \Lb^{\text{pen}}_n(\theta;\widehat{\Uc}_n)}\nonumber\\
&=&  O_p\big(\ln(\ln n)\sqrt{n}+ n\|\theta - \theta_0\|_1 + n\|\theta - \theta_0\|^2_1  \big)
+ n \partial_{2}\pp(\lambda_n,|\theta_j|)\text{sgn}(\theta_j)    \nonumber \\
&=&  O_p\big(\ln(\ln n)\sqrt{n}  + n\sqrt{p}\nu_n + np\nu_n^2    \big)
+ n \partial_{2}\pp(\lambda_n,|\theta_j|)\text{sgn}(\theta_j) \label{bound_oracle} \\
&=& n \lambda_n\Big\{O_p\big(\ln(\ln n)/(\sqrt{n}\lambda_n) + a_n/\lambda_n\big)  + \lambda^{-1}_n \partial_{2}\pp(\lambda_n,|\theta_j|)\text{sgn}(\theta_j)\Big\}.
\nonumber
\end{eqnarray}
Under the assumptions $\underset{n \rightarrow \infty}{\lim \, \inf} \; \underset{\theta \rightarrow 0^+}{\lim \, \inf} \, \lambda^{-1}_n \partial_2\pp(\lambda_n,\theta) > 0$,
$a_n=o(\lambda_n)$ and $\sqrt{n}\lambda_n/\ln(\ln n)$ tends to the infinity, the sign of the derivative is determined by the sign of $\theta_j$. As a consequence, (\ref{sign_consistency}) is satisfied, implying our assertion (i). 
\textcolor{black}{Indeed, all zero components of $\theta_0$ will be estimated as zero with a high probability. And its non-zero components will be consistently estimated. Then, the probability that all the latter estimates will not be zero tends to one, when $n\rightarrow \infty$.} 

\noindent\emph{Point (ii):} We have proved that $\underset{n \rightarrow \infty}{\lim}\; \Pb(\widehat{\Ac}=\Ac)=1$. Therefore, for any $\epsilon>0$,
the event $\widehat{\theta}_{\Ac^c}=\mathbf{0}$ in $\Rb^{|\Ac^c|}$ occurs with a probability larger than $1-\epsilon$ for $n$ large enough.
Since we want to state a convergence in law result, we can consider that the latter event is always satisfied.
By a Taylor expansion around the true parameter, the orthogonality conditions yield
\begin{eqnarray*}
\lefteqn{0=\nabla_{\theta_{\Ac}} \Lb^{\text{pen}}_n(\widehat\theta_{\Ac}:\mathbf{0}_{\Ac^c};\widehat{\Uc}_n) }\\
& = &  \nabla_{\theta_{\Ac}} \Lb_n(\theta_0;\widehat{\Uc}_n) +  \nabla^2_{\theta_{\Ac}\theta^\top_{\Ac}} \Lb_n(\theta_0;\widehat{\Uc}_n) (\widehat\theta-\theta_{0})_{\Ac}+   \frac{1}{2}\nabla_{\theta_{\Ac}}\Big\{(\widehat\theta-\theta_{0})^\top_{\Ac} \nabla^2_{\theta_{\Ac}\theta^\top_{\Ac}} \Lb_n(\overline{\theta};\widehat{\Uc}_n) \Big\}(\widehat\theta-\theta_{0})_{\Ac} \\
& + &  n \mathbf{A}_n(\theta_0) + n\Big\{ \mathbf{B}_n(\theta_0) + o_p(1) \Big\}(\widehat\theta-\theta_{0})_{\Ac},
\end{eqnarray*}
where $\mathbf{A}_n(\theta) = \big[\partial_{2}\pp(\lambda_n,|\theta_k|)\text{sgn}(\theta_k)\big]_{k \in \Ac}$ and
$\mathbf{B}_n(\theta) = \big[\partial^2_{\theta_k\theta_l}\pp(\lambda_n,|\theta_k|)\big]_{(k,l) \in \Ac^2}$, which simplifies into a diagonal matrix since the penalty function is coordinate-separable. 
Obviously, $\overline{\theta}$ is a random parameter such that $\|\overline{\theta}_{\Ac} - \theta_{0,\Ac} \|_2 < \|\widehat{\theta}_{\Ac} - \theta_{0,\Ac} \|_2$.  
Rearranging the terms and multiplying by $\sqrt{n}$, we deduce
\begin{eqnarray*}
\lefteqn{\sqrt{n}\Kb_n(\theta_0)\Big\{\big(\widehat{\theta}-\theta_0\big)_{\Ac} + \mathbf{A}_n(\theta_0)\Big\}}\\
& = & - \frac{1}{\sqrt{n}}\nabla_{\theta_{\Ac}} \Lb_n(\theta_0;\widehat{\Uc}_n) -  \frac{1}{2\sqrt{n}}\nabla_{\theta_{\Ac}}\Big\{(\widehat\theta-\theta_{0})^\top_{\Ac} \nabla^2_{\theta_{\Ac}\theta^\top_{\Ac}} \Lb_n(\overline{\theta};\widehat{\Uc}_n) \Big\}(\widehat\theta-\theta_{0})_{\Ac}+o_p(1) \\
& := & T_1+T_2+o_p(1),
\end{eqnarray*}
where $\Kb_n(\theta_0) = n^{-1}\nabla^2_{\theta_{\Ac}\theta^\top_{\Ac}} \Lb_n(\theta_0;\widehat{\Uc}_n)+\mathbf{B}_n(\theta_0)$.
First, under the $g_{\omega}$-regularity conditions and by Corollary~\ref{cor_GC}, we have $n^{-1}\nabla^2_{\theta_{\Ac}\theta^\top_{\Ac}} \Lb_n(\theta_0;\widehat{\Uc}_n) = \Hb_{\Ac\Ac^\top} + o_p(1)$. Second, the third order term $T_2$ can be managed as follows:
\begin{equation*}
T_2 = -  \frac{1}{2\sqrt{n}}\nabla_{\theta_{\Ac}}\Big\{(\widehat\theta-\theta_{0})^\top_{\Ac} \nabla^2_{\theta_{\Ac}\theta^\top_{\Ac}} \Lb_n(\overline{\theta};\widehat{\Uc}_n) \Big\}(\widehat\theta-\theta_{0})_{\Ac},
\end{equation*}
is a vector as size $|\Ac|$ whose $j$-th component is
$$ T_{2,j}:=n^{-1/2}\underset{l,m \in \Ac}{\sum}\partial^3_{\theta_j \theta_l\theta_m} \Lb_n(\overline{\theta};\widehat{\Uc}_n) (\widehat\theta_l-\theta_{0,l})(\widehat\theta_m-\theta_{0,m}), $$
for any $j \in \Ac$.
Invoking Corollary~\ref{cor_GC} and Assumption~\ref{regularity_assumption}, $T_{2,j}=O_P\big( \sqrt{n} \nu_n^2\big)$ for any $j$.
Then, since $a_n=o(\lambda_n)$, we obtain
$$ T_2 = O_P\big( \ln(\ln n)^2 n^{-1/2} + \sqrt{n} \lambda_n^2   \big)=o(1).$$
Regarding the gradient in $T_1$, since $\partial_{\theta_j}\ell(\theta_0;\uu)$ belongs to our $g_{\omega}$-regular family $\Fc$ for any $j \in \Ac$, apply Theorem~\ref{Th_fondam_multiv_rank_stat} (ii): 
for any $j \in \Ac$, we have
\begin{eqnarray}
\lefteqn{n^{-1/2}\partial_{\theta_j} \Lb_n(\theta_0;\widehat{\Uc}_n) = \sqrt{n}\int_{(0,1]^d} \partial_{\theta_j}\ell(\theta_0;\uu)
(C_n-C)(d\uu)} \label{tight_limit}\\
& \overset{d}{\underset{n \rightarrow \infty}{\longrightarrow}} & (-1)^d\int_{(0,1]^d} \Cb(\uu) \, \partial_{\theta_j}\ell(\theta_0;d\uu)+
\sum_{ \substack{I \subset \{1,\ldots,d\} \\ I\neq \emptyset, I\neq \{1,\ldots,d\} } } (-1)^{|I|}
  \int_{(0,1]^{|I|}} \Cb(\uu_{I}:\1_{-I}) \, \partial_{\theta_j}\ell(\theta_0,d\uu_{I}:\1_{-I}).\nonumber
  \end{eqnarray}
We then conclude by Slutsky's Theorem to deduce the asymptotic distribution
\begin{eqnarray*}
\sqrt{n}\Big[\Hb_{\Ac\Ac}+\mathbf{B}_n(\theta_0)\Big]\Big\{\big(\widehat{\theta}-\theta_0\big)_{\Ac} + \Big[\Hb_{\Ac\Ac}+\mathbf{B}_n(\theta_0)\Big]^{-1}\mathbf{A}_n(\theta_0)\Big\}
& \overset{d}{\underset{n \rightarrow \infty}{\longrightarrow}} & \mathbf{W},
\end{eqnarray*}
where $\mathbf{W}$ the $|\mathcal{A}|$-dimensional random vector defined in (\ref{tight_limit}).

\begin{remark}\label{alternative_proof_AN}
\textcolor{black}{It would be possible to state Theorem~\ref{oracle_property}-(ii) under other sets of regularity conditions. Indeed, the latter result mainly requires a CLT (given by our Theorem A.1 (ii)) and a ULLN (given by our Corollary A.2). 
The latter one may be obtained with a condition on the bracketing numbers associated with the family of maps 
$\Fc_\delta:=\{ \uu\in [0,1]^d\mapsto \partial_{\theta_k}\ell(\theta;\uu),\; \|\theta -\theta_0\|<\delta, k=1,\ldots,p\},$
for some (small) $\delta>0$: see Remark~\ref{rem_ULLN_bis} at the end of the main \textcolor{black}{text}. This would provide an alternative way of managing the term $T_2$.
To deal with $T_1$, a CLT for 
$\sqrt{n}\int_{(0,1]^d} \nabla_{\theta}\ell(\theta_0;\uu)d\big\{C_n(\uu)-C(\uu)\big\}$ 
can be obtained under some regularity conditions on $\uu\mapsto \nabla_{\theta}\ell(\theta_0;\uu)$ and its derivative w.r.t. $\uu$, as introduced in~\citep{ruymgaart1974} and~\citep{ruymgaart1972}. 
See~\citep{tsukahara2005}, Prop. 3 and Assumption (A.1), to be more specific. At the opposite, the proof of Theorem~\ref{oracle_property}-(i) (support recovery) requires Corollary~\ref{cor_GC} and not only a usual ULLN. Indeed, an upper bound for the rate of convergence to zero of $\|\widehat\theta_n-\theta_0\|$ is here required to manage the penalty functions and sparsity.}
\end{remark}

\section{Regularity conditions for Gaussian copulas}\label{verif_regul_Gaussian_cop}

Let us verify that the Gaussian copula family fulfills all regularity conditions that are required to apply Theorems \ref{bound_proba} and~\ref{oracle_property}.
Here, the loss function is $\ell(\theta;\UU)=\ln |\Sigma| + \ZZ^\top \Sigma^{-1} \ZZ,$ where 
$ \ZZ:=\big( \Phi^{-1}(U_1),\ldots,\Phi^{-1}(U_d) \big)^\top $.
Since the true underlying copula is Gaussian, the random vector 
$\ZZ$ in $\Rb^d$ is Gaussian $\Nc(0,\Sigma)$. 
The vector of parameters is $\theta=vech(\Sigma)$, whose ``true value'' will be $\theta_0=vech(\Sigma_0)$. 
Note that $\ell(\theta;\UU)$, as a function of $\UU$, is a quadratic form w.r.t. $\ZZ$, and that 
\begin{equation}
\sup_{j,k\in \{1,\ldots,d\}} \Eb\Big[ \big| \Phi^{-1}(U_j)  \Phi^{-1}(U_k)\big| \Big] <\infty.
\label{Gaussian_trick}
\end{equation}
Assumption~\ref{regularity_assumption}:
when $\Sigma_0$ is invertible, $\Hb$ and $\Mb$ are positive definite. This is exactly the same situation as the 
Hessian matrix associated with the (usual) MLE for the centered Gaussian random vector $\ZZ$.
When $\theta$ belongs to a small neighborhood of $\theta_0$, the associated correlation $\Sigma$ is still invertible by continuity. 
Then, the third order partial derivatives of the loss are uniformly bounded in expectation in such a neighborhood, due to~(\ref{Gaussian_trick}).

\mds 

The first part of \textcolor{black}{Assumption~\ref{cond_reg_copula}} is satisfied for Gaussian copulas, as noticed in~\citep{segers2012asymptotics}, Example 5.1. Checking~(\ref{cond_2nd_der_cop}) is more complex. This is the topic of Lemma~\ref{cond_derivative_gaussian_cop} below.
    
\mds 

\textcolor{black}{Let us verify Assumption~\ref{assumption_regularity} for the Gaussian loss defined as $\ell(\theta;\uu) = \ln |\Sigma| + \text{tr}(\zz\zz^\top\Sigma^{-1})$, with $\zz := (\Phi^{-1}(u_1),\ldots,\Phi^{-1}(u_d))^\top$}: every member $f$ of the family $\Fc$ has
continuous and integrable partial derivatives on $(0,1)^d$ and then is $BHKV_{loc}\big((0,1)^d \big)$.
Moreover, any $f\in \Fc$
can be written as a quadratic form w.r.t. $\zz$:
$$ f(\uu)= \sum_{k,l=1}^d  \nu_{k,l} z_k z_l = \sum_{k,l=1}^d  \nu_{k,l} \Phi^{-1}(u_k) \Phi^{-1}(u_l),\;\; \uu\in (0,1)^d. $$
Note that, for every $u\in (0,1)$, $|\Phi^{-1}(u)| \leq (u(1-u))^{-\eps}$ for any $\eps>0$ and $ \min (u,1-u)\leq u(1-u)/2$.
Thus, for every $\omega>0$, we clearly have 
$$ \sup_{\uu\in (0,1)^d} \min_k \min (u_k, 1-u_k)^\omega  |f(\uu)| <\infty,$$
for every $f\in \Fc$, proving the first condition for the $g_\omega$ regularity of $\Fc$. 
To check~(\ref{bound_fg}), note that $f(d\uu)$ is zero when $d>2$. Otherwise, when $d=2$,  
$f(d\uu) \propto \, du_1 \, du_2/ \big\{ \phi(z_1) \phi(z_2)\big\}$. 
Thus, to apply our theoretical results, it is sufficient to check Assumption \ref{assumption_regularity} by replacing $g_\omega(\uu)$ by 
$g_\omega(u_1,u_2)$, as if the loss and its derivatives were some functions of $(u_1,u_2)$ only, instead of $\uu$.
See the remark after the definition of the $g_\omega$ regularity too.
Since $1/\phi\circ \Phi^{-1}(u) = O\big( 1/u(1-u) \big)$ (c.f.~(\ref{ineg_utile}) 
in the proof of Lemma~\ref{cond_derivative_gaussian_cop}), we obtain
\begin{eqnarray*}
\lefteqn{
 \int_{(0,1)^2} \frac{ g_\omega(u_1,u_2)}{|\phi(z_1) \phi(z_2) |} \, du_1 \, du_2 \leq 
 M\int_{(0,1)^2} 
\frac{ g_\omega(u_1,u_2)}{u_1(1-u_1) u_2 (1-u_2)} \, du_1 \, du_2   }\\
&=& M \Big\{ \int_{(0,1)} 
\frac{ \min(u,1-u)^{\omega/2}}{u(1-u) } \, du \Big\}^2 <\infty,\hspace{7cm}
\end{eqnarray*}
for some constant $M$, yielding~(\ref{bound_fg}).
Note that we have used the inequality 
$ \min (u_1,u_2)^\omega\leq u_1^{\omega/2}u_2^{\omega/2}$.

\mds 

Now consider~(\ref{bound_Cn_1}). 
We restrict ourselves to the case $|J_1|=1$, because we are interested in the situation for which $J_2\cup J_3 \neq \emptyset$. 
Again, when the cardinality of $J_1$ is larger than two, the latter condition is satisfied 
because $f(d\uu_{J_1}: \cc_{n,J_2}:\dd_{n,J_3})=0$. 
When $J_1$ is a singleton, say $J_1=\{1\}$, the absolute value of 
$$ \Jc_{J_2,J_3}:=\int_{\BB_{n,|J_1|}} g_\omega(\uu_{J_1}:\cc_{n,J_2}:\dd_{n,J_3}\big) f\big(d\uu_{J_1}:\cc_{n,J_2}:\dd_{n,J_3}\big) $$ 
is smaller than a constant times 
$$ 
\int_{(1/2n,1-1/2n]} g_\omega(u_1:\cc_{n,J_2}:\dd_{n,J_3}\big) 
\frac{|\Phi^{-1}(u_1)| +  |\Phi^{-1}(1/2n)|}{\phi \circ \Phi^{-1}(u_1)} \,  du_1.$$
We have used the identity $\Phi^{-1}(1-1/2n) = -\Phi^{-1}(1/2n)$.
Note that $g_\omega(u_1:\cc_{n,J_2}:\dd_{n,J_3}\big) = (1/2n)^\omega$ when $u_1\in (1/2n,1-1/2n]$.
Using $\Phi^{-1}(u)/ \phi \circ \Phi^{-1}(u)=O\big(1/ u(1-u)\big)$, this yields 
\begin{eqnarray*}
\lefteqn{
\big| \Jc_{J_2,J_3} \big| \leq 
\frac{K_0}{n^\omega} \int_{1/2n}^{1-1/2n} \frac{du}{u(1-u)} + 
\frac{K_0 |\Phi^{-1}(1/2n)|}{n^\omega}\int_{1/2n}^{1-1/2n} \frac{du}{\phi\circ \Phi^{-1}(u)}   }\\
&\leq & \frac{2 K_0}{n^\omega} \ln (2n) + \frac{K_0 |\Phi^{-1}(1/2n)|^2}{n^\omega}, \hspace{6cm}
\end{eqnarray*} 
for some constant $K_0$. The latter r.h.s. tends to zero, because $\Phi^{-1}(1/2n) \sim - \sqrt{2\ln (2n)}$ (\citep{abramowitz1972handbook}, 26.2.23).
This reasoning can be led for every map $f\in  \Fc$. 
This proves~(\ref{bound_Cn_1}). Importantly, we have proved that all integrals as $\Jc_{J_2,J_3}$ tend to zero with $n$. As a consequence, the limiting law in Theorem~\ref{oracle_property} is simply $ \WW := (-1)^d\int_{(0,1]^d} \Cb(\uu) \, \nabla_{\theta}\ell(\theta_0;d\uu)$.

\mds

Assumption~\ref{regularity_cond_wc} is a direct consequence of the dominated convergence theorem and the upper bounds that have 
been exhibited just before.
\begin{remark}
Alternatively to the Gaussian loss function, consider the least squares loss 
$$\normalfont\ell_{\text{LS}}(\theta;\uu) := \|\zz\zz^\top-\Sigma\|^2_F=\text{tr}\big( (\zz\zz^\top-\Sigma)^2 \big).$$
Then, it is easy to check our regularity assumptions as above.
In particular, every member $f$ of $\Fc$ can be written as a quadratic form w.r.t. $\zz$, as for the previous log-likelihood loss. 
Then the same techniques apply. 
\end{remark}

\begin{remark}
For completeness, let us provide the gradient of the \textcolor{black}{Gaussian and least squares} losses w.r.t. the parameter $\theta$. 
By the identification of the gradient following \citep{abadir2005}, the derivatives \textcolor{black}{of the Gaussian and least squares functions are, respectively, given by} 
\begin{equation*}
\normalfont\nabla_{\text{vec}(\Sigma)} \ell(\theta;\uu) = \text{vec}(\Sigma^{-1} - \Sigma^{-1} \zz \zz^\top \Sigma^{-1}), \;\text{and}\; 
\textcolor{black}{\nabla_{\text{vec}(\Sigma)} \ell_{\text{LS}}(\theta;\uu) = -2 \text{vec}(\zz \zz^\top -\Sigma).}
\end{equation*}
\textcolor{black}{The Gaussian and least squares based Hessian matrices respectively follow as}
\begin{equation*}
\normalfont\nabla^2_{\text{vec}(\Sigma) \text{vec}(\Sigma)^\top} \ell(\theta;\uu) = -\big(\Sigma^{-1} \otimes \Sigma^{-1}\big) + \big(\Sigma^{-1}\zz \zz^\top\Sigma^{-1} \otimes \Sigma^{-1}\big) + \big( \Sigma^{-1} \otimes \Sigma^{-1}\zz \zz^\top\Sigma^{-1} \big),
\end{equation*}
\textcolor{black}{and $\normalfont\nabla^2_{\text{vec}(\Sigma) \text{vec}(\Sigma)^\top} \ell_{\text{LS}}(\theta;\uu) =2\big(I_p \otimes I_p\big)$.}
\end{remark}

\begin{lemma}\label{cond_derivative_gaussian_cop}
Let $C_\Sigma$ be a $d$-dimensional Gaussian copula. Then, there exists a constant $M_{d,\Sigma}$ s.t., for every $j_1,j_2$ in $\{1,\ldots, d\}$ and every $\uu\in V_{j_1}\cap V_{j_2}$, 
$$ \big| \frac{\partial^2 C_{\Sigma}}{\partial u_{j_1} \partial u_{j_2}} (\uu)\big|  \leq 
M_{d,\Sigma}\min \Big\{  \frac{1}{u_{j_1}(1-u_{j_1})},\frac{1}{u_{j_2}(1-u_{j_2})}     \Big\}.$$
\end{lemma}
Obviously, the constant $M_{d,\Sigma}$ depends on the dimension $d$ and on the correlation matrix $\Sigma$.
We prove the latter property below. In dimension two, it had been announced in~\citep{segers2012asymptotics} and~\citep{berghaus2017weak}, but without providing the corresponding (non trivial) technical details. 

\begin{proof}
First assume that $d=2$ and $(j_1,j_2)=(1,2)$. 
Note that the random vector $(X_1,X_2):=\big( \Phi^{-1}(U_1),\Phi^{-1}(U_2) \big)$ is Gaussian $\Nc(\0,\Sigma)$.
The extra diagonal coefficient of $\Sigma$ is $\theta$.
Moreover, $C_\Sigma(u,v)=\Phi_{\Sigma}(x,y)$, where $\Phi_\Sigma$ is the bivariate cdf of $(X_1,X_2)$, $x:=\Phi^{-1}(u)$ and 
$y:=\Phi^{-1}(v)$. With obvious notations, simple calculations provide
$$ \partial^2_{1,2} C_\Sigma(u,v)= \frac{ \partial^2_{1,2} \Phi_\Sigma(x,y) }{\phi(x)\phi(y)}   
= \frac{1}{2\pi \sqrt{ 1-\theta^2}} \exp\Big(-\frac{1}{2(1-\theta^2)} \big\{ \theta^2 x^2 + \theta^2 y^2 - 2\theta xy  \big\}\Big).$$
Since $\theta^2 x^2 + \theta^2 y^2 - 2\theta xy  = (\theta x-y)^2 + (\theta^2 -1)y^2$, we deduce
$$ \big|    \partial^2_{1,2} C_\Sigma(u,v)  \big| \leq  \frac{1}{2\pi \sqrt{1-\theta^2}} \exp(y^2/2).$$
But there exists a constant $M$ such that 
\begin{equation}
 \frac{1}{\phi(y)} \leq \frac{M}{v(1-v)}, \; v\in (0,1)\cdot
\label{ineg_utile}
\end{equation}
Indeed, when $v$ tends to zero, $\Phi^{-1}(1-v)= -\Phi^{-1}(v) \sim 
\sqrt{\ln (1/v^2)}$ (\citep{abramowitz1972handbook}, 26.2.23). Then, the map 
$v\mapsto v(1-v) /\phi\circ \Phi^{-1}(v)$ is bounded. 
As a consequence, $  \partial^2_{1,2} C_\Sigma(u,v)  =O\big( 1/v(1-v) \big)$. 
Similarly, by symmetry, $  \partial^2_{1,2} C_\Sigma(u,v)  =O\big( 1/u(1-u) \big)$, proving the announced inequality for crossed partial derivatives in the bivariate case.

\mds 

Second, assume that $d=2$ and $j_1=j_2=1$. 
We the same notations as above, simple calculations provide
$$ \partial^2_{1,1} C_\Sigma(u,v)= \frac{ x\partial_{1} \Phi_\Sigma(x,y) }{\phi(x)^2} 
+ \frac{ \partial^2_{1,1} \Phi_\Sigma(x,y) }{\phi(x)^2} =: T_1 + T_2.$$
Note that
\begin{eqnarray*}
\lefteqn{ \partial_{1} \Phi_\Sigma(x,y)= 
\int_{-\infty}^y
\frac{1}{2\pi \sqrt{1-\theta^2}} \exp\Big(-\frac{1}{2(1-\theta^2)} \big\{ x^2 + t^2 - 2\theta x t  \big\}\Big) \, dt   }\\
&=&
\int_{-\infty}^y
\frac{1}{2\pi \sqrt{1-\theta^2}} \exp\Big(-\frac{1}{2(1-\theta^2)} \big\{ (t-\theta x)^2 + (1-\theta^2)x^2 \big\}\Big) \, dt
\leq \frac{\phi(x)}{\sqrt{2\pi }},
\end{eqnarray*}
implying 
\begin{equation}
 |T_1| \leq \frac{|x|}{\phi(x)\sqrt{2\pi } }  = O\Big(\frac{1}{u(1-u)}\Big).
 \label{order_T1}
\end{equation}
Indeed, it is easy to check that the map $x\mapsto x\Phi(x) \big(1-\Phi(x) \big)/\phi(x)$ is bounded, because 
$\Phi(x)\sim \phi(x)/|x|$ (resp. $1-\Phi(x)\sim \phi(x)/|x|$) when $x \rightarrow -\infty$ (resp. $x \rightarrow +\infty$).
Moreover, 
\begin{eqnarray*}
\lefteqn{ 
\partial^2_{1,1} \Phi_\Sigma(x,y)= 
\int_{-\infty}^y
\frac{(-1)}{2\pi \sqrt{1-\theta^2}} \exp\Big(-\frac{1}{2(1-\theta^2)} \big\{ x^2 + t^2 - 2\theta x t  \big\}\Big) 
\Big\{ \frac{x-t\theta }{1-\theta^2} \Big\}\, dt   }\\
&=&
\int_{-\infty}^y
\frac{1}{2\pi \sqrt{1-\theta^2}} \exp\Big(-\frac{1}{2(1-\theta^2)} \big\{ (t-\theta x)^2 + (1-\theta^2)x^2 \big\}\Big) 
\Big\{ \frac{x-t\theta }{1-\theta^2} \Big\} \, dt.
\end{eqnarray*}
Then, after an integration w.r.t. $t$, we get 
$$ \big| \partial^2_{1,1} \Phi_\Sigma(x,y) \big| \leq M_1 \big(  |x| + 1\big) \phi(x), $$
for some constant $M_1$. This yields 
\begin{equation}
 T_2=O\Big(  \frac{|x| + 1}{\phi(x)} \Big) =  O\Big(\frac{1}{u(1-u)}\Big).
\label{order_T2}
\end{equation}
Globally,~(\ref{order_T1}) and~(\ref{order_T2}) imply the announced result in this case. 

\mds

Third, assume $d>2$ and choose the indices $(j_1,j_2)=(1,2)$, w.l.o.g.
By the definition of partial derivatives, we have
\begin{eqnarray*}
\lefteqn{ \partial^2_{1,2} C_\Sigma(\uu)= 
\lim_{|\Delta u_1|+|\Delta u_2|\rightarrow 0} \frac{1}{\Delta u_1 \Delta u_2} 
\Pb\Big( U_1\in [u_1,u_1+\Delta u_1], U_2\in [u_2,u_2+\Delta u_2], \UU_{3:d} \leq \uu_{3:d}    \Big)
 }\\
&\leq & 
\lim_{|\Delta u_1|+|\Delta u_2|\rightarrow 0} \frac{1}{\Delta u_1 \Delta u_2} 
\Pb\Big( U_1\in [u_1,u_1+\Delta u_1], U_2\in [u_2,u_2+\Delta u_2] \Big)   \\
&=& \partial^2_{1,2} C_{\Sigma_{1,2}}(u_1,u_2)\leq M_{2,\Sigma_{1,2}}
\min \Big\{  \frac{1}{u_{1}(1-u_{1})},\frac{1}{u_{2}(1-u_{2})}     \Big\},\hspace{3cm}
\end{eqnarray*}
applying the previously proved result in dimension $2$. Thus, the latter inequality is extended for any dimension $d$, at least concerning 
crossed partial derivatives.

\mds 

Concerning the second-order derivative of $C_\Sigma$ w.r.t. $u_1$ (to fix the ideas) and when $d>2$, we can mimic the bivariate case. 
For any $\uu\in (0,1)^d$, set $\xx=(x_1,\ldots,x_d)$ with $x_j:=\Phi^{-1}(u_j)$, $j\in \{1,\ldots,d\}$. We obtain
$$ \partial^2_{1,1} C_\Sigma(\uu)= \frac{ x_1\partial_{1} \Phi_\Sigma(\xx) }{\phi(x_1)^2} 
+ \frac{ \partial^2_{1,1} \Phi_\Sigma(\xx) }{\phi(x_1)^2} =: V_1 + V_2.$$
Note that, with $X_j:=\Phi^{-1}(U_j)$ for every $j$, we have
\begin{eqnarray}
\lefteqn{ 
\partial_{1} \Phi_\Sigma(\xx) =  \lim_{\Delta x_1\rightarrow 0}\frac{1}{\Delta x_1} 
\Pb\Big(X_1\in [x_1,x_1+\Delta x_1], \XX_{2:d}\leq \xx_{2:d}\Big) }\nonumber\\
&\leq & 
\lim_{\Delta x_1\rightarrow 0}\frac{1}{\Delta x_1} 
\Pb\Big(X_1\in [x_1,x_1+\Delta x_1]\Big)=\phi(x_1).
\label{upper_bound_partial1}
\end{eqnarray}
The latter upper bound does not depend on $\Sigma$.
Therefore, we get
\begin{equation}
    |V_1|\leq \frac{|x_1|}{\phi(x_1)} \leq \frac{M_3}{u_1(1-u_1)},
\label{order_V1}
\end{equation}
for some constant $M_3$.
Moreover, for some constants $a\in \Rb$ and $\bbb \in \Rb^{d-1}$ that depend on $\Sigma$ only, we have
\begin{equation}
\partial^2_{1,1} \Phi_\Sigma(\xx)= 
\frac{1}{(2\pi)^{d/2} \sqrt{|\Sigma|}} \int_{\Rb^{d-1}} \1( \ttt\leq  \xx_{2:d}) 
\exp\Big(-\frac{1}{2}[x_1,\ttt]^\top\Sigma^{-1} [x_1,\ttt] \Big) (a x_1+\ttt^\top \bbb) \, d\ttt.
\label{partial_der_11_dimd}
\end{equation}
 It can be proved that there exist some constants $M_4$ and $M_5$ s.t.     
\begin{equation}
\big| \partial^2_{1,1} \Phi_\Sigma(\xx) \big| \leq  
\frac{M_4|x_1|+M_5}{(2\pi)^{d/2} \sqrt{|\Sigma|}} \int_{\Rb^{d-1}} \1( \ttt\leq  \xx_{2:d}) 
\exp\Big(-\frac{1}{2}[x_1,\ttt]^\top\Sigma^{-1} [x_1,\ttt] \Big)  \, d\ttt,
\label{ineg_partial22C}
\end{equation} 
where $\partial_{1} \Phi_{\Sigma}(\xx)$ appears 
on the r.h.s. of~(\ref{ineg_partial22C}). 
Indeed, the multiplicative factors $t_j$, $j\in \{2,\ldots,d\}$, inside the integral sum in~(\ref{partial_der_11_dimd})
do not prevent the use of the same change of variable trick that had been used 
in the bivariate case for the treatment of $T_1$ above. Therefore, after $d-1$ integrations, 
the $x_1$-function that would remain is the same as for $\partial_{1} \Phi_{\Sigma}(\xx)$, apart from a multiplicative factor.
Since 
the latter quantity is $O\big(\phi(x_1)\big)$ due to Equation~(\ref{upper_bound_partial1}), this yields 
\begin{equation}
    |V_2| = O\big( \frac{|x_1|+1}{\phi(x_1)} \Big)=O\Big( \frac{1}{u_1(1-u_1)}\Big).
\label{order_V2}
\end{equation}
Thus,~(\ref{order_V1}) and~(\ref{order_V2}) provide the result when $d>2$, for the second-order derivatives of $C_\Sigma$ w.r.t. $u_1$. 
\end{proof}

Now, consider the case of a mixture of Gaussian copulas, i.e. the true underlying copula density is 
$$ c_\theta(\uu)= \sum_{k=1}^q \pi_k c_{k}(\uu),\;\; \uu\in (0,1)^d,$$
where $\sum_{k=1}^q \pi_k=1$, $\pi_k\in [0,1]$ and $c_k$ is a Gaussian copula density with a correlation matrix $\Sigma_k$, $k\in \{1,\ldots,d\}$.
Here, $\theta$ is the concatenation of the $q-1$ first weights and the unknown parameters of every Gaussian copula. 
The latter ones are given by the lower triangular parts of the correlations matrices $\Sigma_k$. 
Assume the true weights are strictly positive.  
The $m$-order partial derivatives of the loss function $\ell(\uu):=-\ln c_\theta(\uu)$ w.r.t. $\theta$ and/or w.r.t. its arguments $u_1,\ldots,u_d$ are linear combinations of maps as 
\begin{equation}
 \prod_{j=1}^r \partial^{\nu_j} c_{i_j}(\uu) / c^r_\theta(\uu) ,   
 \label{key_Gaussian_mixture}
\end{equation}
where $i_j\in \{1,\ldots,d\}$ for every $j$, $\sum_{j=1}^r \nu_j=m$ and $r\leq m$. Here, the derivatives of order $\nu_j$ have to be understood w.r.t. the corresponding parameters and/or the arguments of the copula $c_{i_j}$, possibly multiple times. 
The latter derivatives may be written as 
$$ \partial^{\nu_j} c_{i_j}(\uu) =  c_{i_j}(\uu) Q_{j}(z_1,\ldots,z_d),$$
for some polynomials $Q_{j}$ of the variables $z_k:=\Phi^{-1}(u_k)$.
When all the weights are positive, every map given by~(\ref{key_Gaussian_mixture}) with be (in absolute value) smaller than 
$\prod_{j=1}^r Q_{j}(\zz)$ times a constant, that is itself a polynomial in terms of $\zz$'s components. Checking our Assumption 3, the single problematic one, is then reduced to checking this assumption for polynomials of $z_1,\ldots,z_d$. This task is easy and details are left to the reader. 
To conclude, the penalized CML method can be used with mixtures of Gaussian copulas, at least when the underlying true weights are strictly positive. In practice, penalization should then be restricted to the copulas parameters.

\section{Regularity conditions for Gumbel copulas}
\label{verif_regul_Gumbel_cop}

We now verify that the Gumbel copula family fulfills all regularity conditions that are required to apply Theorems \ref{bound_proba} and~\ref{oracle_property} when the loss is chosen as the opposite of the log-copula density.
A $d$-dimensional Gumbel copula is defined by $C_\theta(\uu):=\psi_\theta\big( \sum_{j=1}^d \psi_\theta^{-1}(u_j)\big)$ where $\psi_\theta(t)=\exp(-t^{1/\theta})$, $t\in \Rb^+$, for some parameter $\theta>1$.
Note that $\psi_\theta^{-1}(u)=(-\ln u)^\theta$, $u\in (0,1]$.
The associated density is
$$ c_\theta(\uu):=\frac{\psi_\theta^{(d)}\big(\sum_{j=1}^d \psi_\theta^{-1}(u_j)   \big)}{\prod_{j=1}^d \psi_\theta'\circ \psi_\theta^{-1}(u_j)},$$
and the considered loss will be $\ell(\theta;\uu)=-\ln c_\theta(\uu)$.
Simple calculations show that $\psi^{(d)}(t)=(-1)^d\psi_\theta(t)Q_{\theta}(t)$ for every $t$, where
$$Q_\theta(t):=\sum_{k=1}^d a_{k,\theta} t^{k/\theta - d},$$
for some coefficients $a_{k,\theta}$ that depend on $\theta$.
Since $\psi_\theta'\circ \psi^{-1}_\theta (u)=-u (-\ln u)^{1-\theta}/\theta $, deduce
\begin{eqnarray*}
\lefteqn{\ell(\theta;\uu)+d\ln \theta=\Big\{ \sum_{j=1}^d (-\ln u_j)^\theta \Big\}^{1/\theta} - \ln Q_\theta\Big( \big(-\ln C_\theta(\uu) \big)^\theta \Big)      }\\
&+& \sum_{j=1}^d \Big\{ \ln u_j + (1-\theta)\ln \big( \ln(1/u_j) \big)   \Big\}=:\ell_{1}(\theta;\uu)-\ell_{2}(\theta;\uu)+\ell_{3}(\theta;\uu).
\end{eqnarray*}

\mds

Assumption \ref{regularity_assumption} is satisfied because all the derivatives of $\theta\mapsto \ell(\theta;\UU)$ are nonzero and given by some polynomial maps of $\theta$, of the quantities $\ln(U_j)$ and $\ln (-\ln U_j)$, $j\in \{1,\ldots,d\}$, or by the logarithm of such maps. The latter maps are clearly integrable, even uniformly w.r.t. $\theta$ in a small neighborhood of $\theta_0$. This may be seen by doing the $d$ changes of variables $-\ln u_j=:x_j$, $j\in \{1,\ldots,d\}$ in the corresponding integrals.

\mds

To state Assumption \ref{cond_reg_copula} and w.l.o.g., let us focus on the cross-derivatives of the true copula w.r.t. its first two components. Note that
$$ \partial^2_{1,2} C_\theta(\uu)=\frac{\psi^{''}\big(\sum_{j=1}^d \psi_\theta^{-1}(u_j)  \big)}{\psi'\circ \psi(u_1)\psi'\circ \psi(u_2)}
=\frac{\psi(t)\big\{ t^{2/\theta-2}- (1-\theta)t^{1/\theta-2}\big\}}{u_1u_2(-\ln u_1)^{1-\theta}(-\ln u_2)^{1-\theta}},$$
setting $t_\theta(\uu):=\sum_{j=1}^d (-\ln u_j)^{\theta} \in \Rb^+$, denoted also by $t$ when there is no ambiguity.
Since $(-\ln u_k)^\theta \leq t_\theta(\uu)$ for every $k$ and $\theta > 1$, we have $\big(-\ln \min (u_1,u_2)\big)^\theta\leq t_\theta(\uu).$
Then, we deduce
{\small{\begin{equation*}
0\leq
\frac{\big\{t^{2/\theta-2}+ t^{1/\theta-2}\big\}}{(-\ln u_1)^{1-\theta}(-\ln u_2)^{1-\theta}}
\leq \Big\{ \big(-\ln \min(u_1,u_2)\big)^{2- 2\theta}+ \big(-\ln \min(u_1,u_2)\big)^{1-2\theta} \Big\} (-\ln \min(u_1,u_2))^{2\theta-2}.
\end{equation*}}}
Since $\psi(t)=C_\theta(\uu)\leq \min_j u_j$,
this easily yields
$$ |\partial_{1,2} C_\theta(\uu)| \leq
\frac{(\theta+1) (\min_j u_j)}{u_1 u_2} \Big\{  1+ \big(-\ln \min(u_1,u_2)\big)^{-1} \Big\} = O\Big( \min \big\{ \frac{1}{u_1(1-u_1)};\frac{1}{u_2(1-u_2)}   \big\} \Big).$$

To check Assumption \ref{assumption_regularity} (the $g_\omega$ regularity of the partial derivatives of the loss function), it is sufficient to verify this assumption for every term $\ell_{k}(\theta;\uu)$, $k\in \{1,2,3\}$.

\mds

{\bf Study of $\ell_{1}(\theta;\UU)$ and Assumption \ref{assumption_regularity}:}
by simple calculations, we get
$$ \partial_\theta \ell_{1}(\theta;\uu) = \frac{(-1)}{\theta^2} \Big\{ \sum_{j=1}^d (-\ln u_j)^\theta \Big\}^{1/\theta}
\ln\Big( \sum_{j=1}^d (-\ln u_j)^\theta \Big)
\Big\{ \sum_{j=1}^d (-\ln u_j)^\theta \ln(-\ln u_j) \Big\}.$$
Let us prove that $\Fc_1=\big\{\partial_\theta \ell_{1}(\theta_0;\cdot) \big\}$ is $g_\omega$-regular for any $\omega >0$. For convenience, $\theta_0$ will simply be denoted as $\theta$ hereafter.

\mds

Since $u \mapsto (\ln u)^\alpha \ln(-\ln u)^\beta u^\gamma$ is bounded on $(0,1)$, for any triplet $(\alpha,\beta,\gamma)$ of positive numbers,
the map $\uu\mapsto \min_k \min(u_k,1-u_k)^\omega |\partial_\theta \ell_{1}(\theta;\uu)|$ is bounded on $(0,1)$ for any positive $\omega$.

\mds

To verify (\ref{bound_fg}), it is necessary to calculate
$$\partial_\theta \ell_{1}(\theta;d\uu) = \partial_\theta \partial^d_{u_1,\ldots,u_d} \ell_{1}(\theta;\uu) d\uu .$$
By tedious calculations, it can be shown that
$$\partial_\theta \partial^d_{u_1,\ldots,u_d} \ell_{1}(\theta;\uu) =
A_d(\theta;\uu)\Big\{ \sum_{j=1}^d (-\ln u_j)^\theta \Big\}^{1/\theta-d}\prod_{j=1}^d \frac{(-\ln u_j)^{\theta-1}}{u_j} ,$$
for some map $A_d(\theta;\cdot)$ that is a power function of the quantities $(-\ln u_j)^\theta$, $\ln(-\ln u_j)$, $j\in \{1,\ldots,d\}$, and $\ln \big( \sum_{j=1}^d (-\ln u_j)^\theta \big)$.
Therefore, by symmetry, we have
\begin{eqnarray*}
\lefteqn{ \int_{(0,1]^d} g_\omega(\uu) |\ell_1(\theta;d\uu)| =
 d!\int_{0< u_1\leq u_2\leq \cdots \leq u_d\leq 1} g_\omega(\uu) |\ell_1(\theta;d\uu)| }\\
&= &  \int_{0< u_1\leq u_2\leq \cdots \leq u_d\leq 1} \min(u_1,1-u_2)^\omega |A_d(\theta;\uu)|\Big\{ \sum_{j=1}^d (-\ln u_j)^\theta \Big\}^{1/\theta-d}\prod_{j=1}^d \frac{(-\ln u_j)^{\theta-1}}{u_j}\,d\uu.
\end{eqnarray*}
For any $\uu\in (0,1]^d$ s.t. $u_1\leq u_2\leq \cdots \leq u_d$,
$ |A_d(\theta;\uu)| \leq \widetilde A_1(\theta;u_1),$
for some map $\widetilde A_1(\theta;u_1)$ that is a power function of the quantities $(-\ln u_1)^\theta$ and $\ln(-\ln u_1)$.
Therefore, after $d-2$ integrations w.r.t. $u_d,\ldots,u_3$ successively, we get
\begin{eqnarray*}
\lefteqn{ \int_{(0,1]^d} g_\omega(\uu) |\ell_1(\theta;d\uu)| \leq
K_\theta
\int_{0< u_1\leq u_2 \leq 1} \min(u_1,1-u_2)^\omega \widetilde A_1(\theta;u_1)  }\\
&\times   & \Big\{  (-\ln u_1)^\theta + (-\ln u_2)^\theta \Big\}^{1/\theta-2}\prod_{j=1}^2 \frac{(-\ln u_j)^{\theta-1}}{u_j}\,du_1\, du_2,
\end{eqnarray*}
for some constant $K_\theta$.
Note that $\min(u_1,1-u_2)^\omega \leq u_1^\omega$.
After another integration w.r.t. $u_2$, there exists another constant $K'_\theta$ s.t.
$$ \int_{(0,1]^d} g_\omega(\uu) |\ell_1(\theta;d\uu)| \leq K'_\theta
\int_{0< u_1 \leq 1} u_1^\omega \widetilde A_1(\theta;u_1)  (-\ln u_1)^{1-\theta} \frac{(-\ln u_1)^{\theta-1}}{u_1}\,du_1 <+\infty .$$
This means (\ref{bound_fg}) is satisfied for $\partial_\theta\ell_1(\theta_0;\cdot)$.

\mds

The same technique can be applied to check (\ref{bound_Cn_1}), assuming $J_2\cup J_3\neq \emptyset$.
Set $m_k:=\text{Card}(J_k)$, $k\in \{1,2,3\}$.
For every $\uu \in \BB_{n,|J_1|}$, $J_1\neq \emptyset$, we have
$$ g_\omega(\uu_{J_1}:\cc_{n,J_2}:\dd_{n,J_3}\big) = 1/(2n)^{\omega},$$
in every case, except when $J_2=\emptyset$ and $m_1\geq 2$ simultaneously.
W.l.o.g., let us assume that the components indexed by $J_1$ are the first ones, i.e. are $u_1,\ldots,u_{m_1}$.
Note that
\begin{eqnarray}
\lefteqn{
 \int_{\BB_{n,|J_1|}} g_\omega(\uu_{J_1}:\cc_{n,J_2}:\dd_{n,J_3}\big) \Big|\partial_\theta \ell_1\big(\theta;d\uu_{J_1}:\cc_{n,J_2}:\dd_{n,J_3}\big)  \Big|}\nonumber \\
 &=&  \int_{\BB_{n,|J_1|}} g_\omega(\uu_{J_1}:\cc_{n,J_2}:\dd_{n,J_3}\big)
|\bar A_d(\theta;\uu_{J_1};n)|\prod_{j=1}^{m_1}\frac{(-\ln u_j)^{\theta-1}}{u_j}
 \nonumber \\
&\times & \Big\{ \sum_{j=1}^{m_1} (-\ln u_j)^\theta
+ m_2 \ln (2n)^\theta + m_3 \ln \big(\frac{2n}{2n-1}\big)^\theta \Big\}^{1/\theta-m_1}\, d\uu,
\label{A2_ell1_decomp}
\end{eqnarray}
for some map $\bar A_d(\theta;\uu_{J_1};n)$ defined on $(0,1]^{m_1}$ that is a power function of the quantities $\ln n$, $(-\ln u_k)^\theta$, $\ln(-\ln u_k)$, $k\in J_1$, and $\ln \big( \sum_{j\in J_1} (-\ln u_j)^\theta \big)$.

\mds

Let us first deal with the case $J_2=\emptyset$ and $m_1\geq 2$. Here,
$$ g_\omega(\uu_{J_1}:\cc_{n,J_2}:\dd_{n,J_3}\big) =\min(u_1,1-u_2)^\omega,$$
when $u_1\leq u_2\leq \ldots \leq u_{m_1}$.
The reasoning is then exactly the same as for checking (\ref{bound_fg}), starting from~(\ref{A2_ell1_decomp}), noting that $m_2=0$ and by doing $m_1$ integrations on
$(1/(2n);1-1/(2n)]$ instead of $(0,1]$. In the calculation of the latter integrals, the single difference w.r.t. (\ref{bound_fg}) comes from the terms $\sum_{j=1}^k (-\ln u_j)^\theta+ \ln(2n/(2n-1))^\theta$ instead of $\sum_{j=1}^k (-\ln u_j)^\theta$. 
This will not change the conclusion 
and we have stated (\ref{bound_Cn_1}).

\mds 

Incidentally, Assumption~\ref{regularity_cond_wc} is easily checked by applying the dominated convergence theorem.
This statement is general and will not be repeated hereafter for the other terms.

\mds

When $J_2\neq \emptyset$ or $m_1=1$, note that
\begin{eqnarray*}
\lefteqn{
 \int_{\BB_{n,|J_1|}} g_\omega(\uu_{J_1}:\cc_{n,J_2}:\dd_{n,J_3}\big) \Big|\partial_\theta \ell_1\big(d\uu_{J_1}:\cc_{n,J_2}:\dd_{n,J_3}\big) \Big| }\\
 &=&  \frac{1}{(2n)^{\omega}}\int_{\BB_{n,|J_1|}}
|\bar A_d(\theta;\uu_{J_1};n)|\Big\{ \sum_{j=1}^{m_1} (-\ln u_j)^\theta
+ m_2 \ln (2n)^\theta + m_3 \ln \big(\frac{2n}{2n-1}\big)^\theta \Big\}^{1/\theta-m_1} \\
&\times & \prod_{j=1}^{m_1}\frac{(-\ln u_j)^{\theta-1}}{u_j}\, d\uu_1.
\end{eqnarray*}
Thus, after $m_1-1$ integration on $(1/(2n);1-1/(2n)]$, we obtain
\begin{eqnarray*}
\lefteqn{ \int_{\BB_{n,|J_1|}} g_\omega(\uu_{J_1}:\cc_{n,J_2}:\dd_{n,J_3}\big) \big|\partial_\theta \ell_1\big(d\uu_{J_1}:\cc_{n,J_2}:\dd_{n,J_3}\big)\big| }\\
& \leq & \frac{\bar K_\theta}{(2n)^{\omega}} \int_{1/(2n)}^{1-1/(2n)} \big\{ (-\ln u_1)^\theta + (\ln n)^\theta \big\}^{1/\theta - 1} \frac{(-\ln u_1)^{\theta-1}}{u_1} \, du_1  \\
&\leq &\frac{\bar K'_\theta}{n^{\omega}}  (\ln n)^{1-\theta} \big[ (-\ln u_1)^\theta \big]_{1/(2n)}^{1-1/(2n)},
\end{eqnarray*}
for some constants $\bar K_\theta$ and $\bar K'_\theta$.
The latter term on the r.h.s. is $O(n^{-\omega} \ln n)=o(1)$, proving (\ref{bound_Cn_1}) when $J_2\neq \emptyset$ or when $m_1=1$.

\mds

With exactly the same techniques, it can be proved that $\Fc_2$ and $\Fc_3$ (the family given by the second and third order derivatives of $\theta \mapsto \ell_{1}(\theta;\cdot) $) are $g_\omega$-regular for any $\omega >0$, even when the parameter $\theta$ is free to belong to a neighborhood of $\theta_0$.

\mds

{\bf Study of $\ell_{2}(\theta;\UU)$ and Assumption \ref{assumption_regularity}:}
Since $\psi$ is the Laplace transform of a stable distribution, it is completely monotone and $(-1)^d \psi^{(d)}(t) >0$ for every $t\in \Rb^+$.
Thus, $ Q_\theta(t) >0$, $t\in \Rb_+^*$.
By definition, $Q_\theta(t):=\sum_{k=1}^d a_{k,\theta} t^{k/\theta - d}.$
By recursion, it can be proved that $a_{1,\theta}=-\prod_{k=0}^{d-1} (k-1/\theta )$, that is not zero because $\theta>1$.
Moreover, $a_{d,\theta}=1/\theta^d$.
Therefore, there exists a positive constant $\lambda_\theta$ s.t.
\begin{equation}
t^d \big| Q_\theta(t)  \big|= t^d Q_\theta(t) =   \sum_{k=1}^d a_{k,\theta} t^{k/\theta} \geq \lambda_\theta t^{1/\theta},
 \label{lower_bound_Qtheta}
 \end{equation}
for any $t>0$.
Note that we can write $\partial_\theta Q_\theta(t)= \sum_{k=1}^d \big\{\beta_{k,\theta}+\gamma_{k,\theta}\ln t\big\} t^{k/\theta-d}$, for some constants $\beta_{k,\theta}$ and 
$\gamma_{k,\theta}$.
As above and to lighten notations, denote $ t_\theta(\uu):=\sum_{j=1}^{d} (-\ln u_j)^\theta$, also denoted $t$ simply.
Simple calculations yield
$$ \partial_\theta \ell_{2}(\theta;\uu)=\frac{\partial_\theta Q_\theta}{Q_\theta}\big( t_\theta(\uu)\big) \partial_\theta t_\theta(\uu),\;\;
\partial_\theta t_\theta(\uu)=
 \sum_{j=1}^{d} (-\ln u_j)^\theta  \ln (-\ln u_j) .$$
Moreover, successive derivatives yield
$$ \partial^2_{\theta,u_1} \ell_{2}(\theta;\uu)=
(\partial_\theta \ln Q_\theta)'(t) \partial_\theta t_\theta(\uu) \partial_{u_1} t_\theta(\uu)
+ (\partial_\theta \ln Q_\theta)(t) \partial^2_{\theta,u_1} t_\theta(\uu),$$
\begin{eqnarray*}
\lefteqn{ \partial^3_{\theta,u_1,u_2} \ell_{2}(\theta;\uu)=
(\partial_\theta \ln Q_\theta)''(t) \partial_\theta t_\theta(\uu) \partial_{u_1} t_\theta(\uu) \partial_{u_2} t_\theta(\uu) }\\
&+& (\partial_\theta \ln Q_\theta)'(t) \Big\{\partial^2_{\theta,u_1} t_\theta(\uu)\partial_{u_2} t_\theta(\uu)
 +  \partial^2_{\theta,u_2} t_\theta(\uu)\partial_{u_1} t_\theta(\uu)\Big\},
\end{eqnarray*}
and, by iteration, we get
\begin{eqnarray*}
\lefteqn{ \partial^{d+1}_{\theta,u_1,\cdots,u_d} \ell_{2}(\theta;\uu)=
(\partial_\theta \ln Q_\theta)^{(d)}(t) \partial_\theta t_\theta(\uu) \prod_{j=1}^d \partial_{u_j} t_\theta(\uu) }\\
&+& (\partial_\theta \ln Q_\theta)^{(d-1)}(t) \sum_{k=1}^d \partial^2_{\theta,u_k} t_\theta(\uu) \prod_{j=1,j\neq k}^d \partial_{u_j} t_\theta(\uu) \\
&=&
(\partial_\theta \ln Q_\theta)^{(d)}(t) \big\{ \sum_{j=1}^{d} (-\ln u_j)^\theta \ln(-\ln u_j) \big\} \prod_{j=1}^d \frac{\theta(-\ln u_j)^{\theta-1}}{u_j} \\
&-& (\partial_\theta \ln Q_\theta)^{(d-1)}(t)
\sum_{k=1}^d \big\{ 1+\theta \ln (-\ln u_k)\big\}\prod_{j=1}^d \frac{(-\ln u_j)^{\theta-1}}{u_j}=:W_1(\uu)- W_2(\uu).
\end{eqnarray*}
Note that, for every positive integer $p$, there exists some constants $b_{p,k}$ and $c_{p,k}$ s.t.
\begin{equation}
 (\partial_\theta \ln Q_\theta )^{(p)}(t)  = \frac{\sum_{k=p+1}^{d(p+1)} (b_{k,p}+ c_{k,p}\ln t) t^{k/\theta - p}}{\big( \sum_{k=1}^d  a_{k,\theta} t^{k/\theta}\big)^p },
 \;\; t>0.
\label{Hp(t)}
\end{equation}
As before, we have by symmetry
\begin{eqnarray*}
\lefteqn{ \int_{(0,1]^d} g_\omega(\uu) |\ell_2(\theta;d\uu)| =
 d!\int_{0< u_1\leq u_2\leq \cdots \leq u_d\leq 1} g_\omega(\uu) |\ell_2(\theta;d\uu)| }\\
&= &  \int_{0< u_1\leq u_2\leq \cdots \leq u_d\leq 1} \min(u_1,1-u_2)^\omega
|\partial^{d+1}_{\theta,u_1,\ldots,u_d} \ell_{2}(\theta;\uu)|\,d\uu \\
&\leq  &  \int_{0< u_1\leq u_2\leq \cdots \leq u_d\leq 1} u_1^\omega
|W_1(\uu)|\,d\uu + \int_{0< u_1\leq u_2\leq \cdots \leq u_d\leq 1} u_1^\omega
|W_2(\uu)|\,d\uu \\
&=:& \Ic_1+\Ic_2.
\end{eqnarray*}
Deduce from~(\ref{lower_bound_Qtheta}) and~(\ref{Hp(t)}) that the first latter term $\Ic_1$ is smaller than a linear combination of integrals as
\begin{equation}
\int_{0< u_1\leq u_2\leq \cdots \leq u_d\leq 1} u_1^\omega
|\ln t|^\tau t^{k/\theta - d} t^{-d/\theta} (-\ln u_i)^\theta |\ln(-\ln u_i)| \prod_{j=1}^d \frac{(-\ln u_j)^{\theta-1}}{u_j} \,d\uu,
\label{ex_I1}
\end{equation}
for some $i\in \{1,\ldots,d\}$, $\tau\in \{0,1\}$, and $k\in \{d+1,\ldots, d+d^2 \}.$
But, for any $\epsilon>0$,
there exist some positive constants $\alpha$ and $\beta$ s.t.
\begin{equation}
 |\ln t|\leq \alpha t^\epsilon + \beta t^{-\epsilon},\;\; t>0.
\label{ineg_ln_eps}
\end{equation}
Moreover, there exist some positive constants $\alpha'$ and $\beta'$ s.t.
 $$ (-\ln u)^\theta |\ln(-\ln u)| \leq \alpha' + \beta' (-\ln u)^{\theta'},\;\; u\in (0,1),$$
 for any $\theta'>\theta$.
In the ``worst case'', the latter integral~(\ref{ex_I1}) is smaller than a constant times
\begin{eqnarray}
\lefteqn{ \int_{0< u_1\leq u_2\leq \cdots \leq u_d\leq 1} u_1^\omega
t^{k/\theta - d-d/\theta-\epsilon} \big\{ (-\ln u_1)^{\theta+1}+1\big\} \prod_{j=1}^d \frac{(-\ln u_j)^{\theta-1}}{u_j} \,d\uu } \nonumber \\
&\propto &
\int   u_1^\omega
t_\theta(u_1,1,\ldots,1)^{(k-d)/\theta-\epsilon-1} \big\{ (\ln \frac{1}{u_1})^{\theta+1} + 1 \big\} \frac{(-\ln u_1)^{\theta-1}}{u_1} \,du_1,
\label{upper_bound_I1}
\end{eqnarray}
after $d-1$ integrations w.r.t. $u_d,u_{d-1},\ldots,u_2$ successively.
The r.h.s. of~(\ref{upper_bound_I1}) is smaller than
$$
\int   u_1^{\omega-1}
(-\ln u_1)^{\theta(k/\theta-d/\theta-\epsilon-1) + \theta-1} \big\{ (-\ln u_1)^{\theta+1} + 1 \big\} \,du_1.
$$
By choosing $\epsilon = 1/(2\theta)$, the latter integral is finite because
$$ \theta(k/\theta-d/\theta-\epsilon-1) + \theta-1 = k-d-1 -\theta \epsilon \geq -\theta \epsilon =(-1/2).$$
This proves that $\Ic_1$ is finite.

\mds

Similarly, $\Ic_2$ is smaller than a linear combination of integrals as
\begin{equation}
\int_{0< u_1\leq u_2\leq \cdots \leq u_d\leq 1} u_1^\omega
|\ln t|^\tau t^{k/\theta - d+1} t^{-(d-1)/\theta} |\ln(-\ln u_i)|^{\bar\tau} \prod_{j=1}^d \frac{(-\ln u_j)^{\theta-1}}{u_j} \,d\uu,
\label{ex_I2}
\end{equation}
for some $i\in \{1,\ldots,d\}$, $(\tau,\bar\tau)\in \{0,1\}^2$, and $k\in \{d,\ldots, d(d-1) \}.$
Reminding~(\ref{ineg_ln_eps}), note that
$$ |\ln(-\ln u)| \leq \alpha (-\ln u)^{-\epsilon}  + \beta (-\ln u)^{\epsilon},\;\; u\in (0,1).$$
Therefore, the ``worst situation'' to manage $\Ic_2$ is to evaluate
\begin{equation}
 \int_{0< u_1\leq u_2\leq \cdots \leq u_d\leq 1} u_1^\omega
t^{k/\theta - (d-1)-(d-1)/\theta-\epsilon} (-\ln u_i)^{-\epsilon} \prod_{j=1}^d \frac{(-\ln u_j)^{\theta-1}}{u_j} \,d\uu .
\label{Ic2_type_integral}
\end{equation}
 When $i=1$, integrate the latter integral w.r.t. $u_d,u_{d-1},\ldots,u_2$ successively, and we obtain a scalar times the integral
$$ \int u_1^\omega
t_\theta(u_1,1,\ldots,1)^{(k-d+1)/\theta-\epsilon} \frac{(-\ln u_1)^{\theta-1-\epsilon}}{u_1} \,du_1 
= \int u_1^{\omega-1} (-\ln u_1)^{\theta \{(k-d+1)/\theta -\epsilon \}+\theta-1-\epsilon} \,du_1,$$
that is finite because
$$ \theta\big\{(k-d+1)/\theta -\epsilon\big\}-\epsilon+\theta-1 \geq \theta-\epsilon (\theta+1)-1>0,$$
for some sufficiently small constant $\epsilon$.

\mds

When $i>1$, first integrate~(\ref{Ic2_type_integral}) w.r.t. $u_i$, but on $(u_1,1]$ instead of $(u_{i-1},1]$.
This will yield an upper bound of the $\Ic_2$-type term~(\ref{Ic2_type_integral}).
In such a case, note that
$$ t_\theta(\uu_{-i},1) \leq  t_\theta(\uu) \leq  2 t_\theta(\uu_{-i},1),\;\; \uu\in (0,1]^d,$$
with obvious notations.
Thus, the term $t_\theta(\uu)$ in~(\ref{ex_I2}) can be replaced by $t_\theta(\uu_{-i},1)$ that does not depend on $u_i$.
The new integral w.r.t. $u_i$ is then
$$ \int_{u_1}^1 \frac{(-\ln u_i)^{\theta-1-\epsilon}}{u_i}du_i = \frac{(-\ln u_1)^{\theta-\epsilon}}{\theta - \epsilon},$$
and we will choose $\epsilon < \theta$.
To bound~(\ref{Ic2_type_integral}), we are restricted to the evaluation of
$$ \int_{u_1\leq \cdots \leq u_{i-1}\leq u_{i+1}\leq \cdots \leq u_d} u_1^\omega
t_\theta(\uu_{-i},1)^{(k-d+1)/\theta - (d-1)-\epsilon} (-\ln u_1)^{\theta-\epsilon} \prod_{j=1,j\neq i}^d \frac{(-\ln u_j)^{\theta-1}}{u_j} \,d\uu_{-i} .$$
Now, integrate w.r.t. $u_d,u_{d-1},\ldots,u_{i+1},u_{i-1},\ldots, u_2$ successively. We obtain a scalar times the integral
\begin{eqnarray*}
\lefteqn{
\int u_1^{\omega}
t_\theta(u_1,1,\ldots,1)^{(k-d+1)/\theta-\epsilon-1} \frac{(-\ln u_1)^{2\theta-\epsilon -1}}{u_1} \,du_1 }\\
&=& \int_0^1 u_1^{\omega-1}
(-\ln u_1)^{k-d + \theta -\theta\epsilon-\epsilon} \,du_1,
\end{eqnarray*}
that is finite for any $\omega>0$, because
$$ k-d +\theta -\theta\epsilon-\epsilon \geq \theta-(\theta+1)\epsilon  >0,$$
for some sufficiently small constant $\epsilon$.
This means (\ref{bound_fg}) is satisfied for $\partial_\theta\ell_2(\theta_0;\cdot)$.

\mds

The same technique can be applied to check (\ref{bound_Cn_1})
for $\partial_\theta\ell_2(\theta_0;\cdot)$,
mimicking the reasonings we led for $\partial_\theta\ell_1(\theta_0;\cdot)$.

\mds

With exactly the same techniques, it can be proved that $\Fc_2$ and $\Fc_3$ (here defined through $\ell_2$)
are $g_\omega$-regular for any $\omega >0$. Actually, this is still the case for any higher-order $\theta$-derivatives of the loss function. Indeed, the effect of such derivatives will be to add some multiplicative factors $\ln(-\ln u_j)$, $j\in \{1,\ldots,d\}$, and such factors do not play any role to check $g_\omega$-regularity.

\mds

Since Assumption 3 is trivially satisfied with $\ell_{3}(\theta;\uu)$, we have proven the validity of this assumption for the Gumbel family.

\begin{remark}
Note that we have proved the regularity assumptions as if the weight function $g_{\omega,d}$ were replaced by $\uu\mapsto \min_{j}u_j^\omega$, implying a stronger requirement.
\end{remark}

\section{Regularity conditions for Clayton copulas}
\label{verif_regul_Clayton_cop}

We now verify that the Clayton copula family fulfills all regularity conditions that are required to apply Theorems \ref{bound_proba} and~\ref{oracle_property} when the loss is chosen as the opposite of the log-copula density.
A $d$-dimensional Clayton copula is defined by $C_\theta(\uu):=\psi_\theta\big( \sum_{j=1}^d \psi_\theta^{-1}(u_j)\big)$ where
$\psi_\theta(t)=(1+t)^{-1/\theta}$, $t\in \Rb^+$, for some parameter $\theta>0$.
Note that $\psi_\theta^{-1}(u)=u^{-\theta} -1$, $u\in (0,1]$ and
$$ C_\theta(\uu)=\Big\{ \sum_{j=1}^d u_j^{-\theta} - d + 1  \Big\}^{-1/\theta},\; \; \uu\in (0,1]^d.$$
The associated density on $(0,1]^d$ is
$$ c_\theta(\uu):=\prod_{k=0}^{d-1}(1+k\theta) \big\{ \sum_{j=1}^d u_j^{-\theta} - d+1 \big\}^{-1/\theta - d} \prod_{j=1}^d u_j^{-\theta - 1},$$
and the considered loss will be the $\ell(\theta;\uu)=-\ln c_\theta(\uu)$.
Note that
$$ \ell(\theta;\uu)=M(\theta) +
\big(\frac{1}{\theta}+d\big) \ln\big(\sum_{j=1}^d u_j^{-\theta}-d+1 \big) +
(1+\theta)\big(\sum_{j=1}^d \ln u_j\big) , $$
where $M(\theta)$ is a positive map of $\theta$ only.

\mds

Assumption \ref{regularity_assumption} is satisfied because $\partial^k_\theta\ell(\theta_0;\UU)$ is nonzero and integrable for any $k\in \{1,2,3\}$, 
even uniformly w.r.t. $\theta$ in a small neighborhood of $\theta_0$. This can be easily seen by noting that 
$$ \frac{ |\sum_{j=1}^d u_j^{-\theta}\ln u_j |}{\sum_{j=1}^d u_j^{-\theta}-d+1}\leq \sum_{j=1}^d |\ln u_j|,$$
for every $\uu\in (0,1]^d$ because $\sum_{j=1}^d u_j^{-\theta}-d+1 \geq u_k^{-\theta}$, $k\in \{1,\ldots,d\}$,
and $\int |\ln u_k| c_\theta(\uu)\, d\uu $ is finite for every $k$.

\mds

To state Assumption \ref{cond_reg_copula} and w.l.o.g., let us focus on the cross-derivative w.r.t. the first two components of the true copula. By simple calculations, we get
$$ \partial_{1,2} C_\theta(\uu)=
\frac{\psi^{''}\big(\sum_{j=1}^d \psi_\theta^{-1}(u_j)  \big)}{\psi'\circ \psi(u_1)\psi'\circ \psi(u_2)}
= (1+\theta) s_\theta(\uu)^{-1/\theta-2} (u_1 u_2)^{-\theta-1},
$$
setting $s_\theta(\uu):=\sum_{j=1}^d u_j^{-\theta}-d+1 \in \Rb^+$.
Since $C_\theta(\uu)\leq \min_j u_j$,
deduce
\begin{eqnarray*}
\lefteqn{ |\partial_{1,2} C_\theta(\uu)| =
(\theta+1) C_\theta (\uu)^{1+2\theta} (u_1 u_2)^{-\theta-1}\leq 
(\theta+1) (\min_j u_j)^{1+2\theta} (u_1 u_2)^{-\theta-1}}\\
&\leq & (\theta+1) \min (u_1,u_2)^{1+2\theta} \min (u_1,u_2)^{-2\theta-2} = 
O\Big( \min \big\{ (u_1(1-u_1))^{-1};(u_2(1-u_2))^{-1}   \big\} \Big).
\end{eqnarray*}

\mds

Let us check Assumption \ref{assumption_regularity}, i.e. the $g_\omega$ regularity of the partial derivatives of the loss function.
We will do the task for $\Fc_1=\{\uu\mapsto \partial_\theta \ell(\theta_0;\uu) \}$ only.
The task for $\Fc_2$ and $\Fc_3$, or even for every higher-order derivatives of the loss w.r.t. $\theta$, is obtained by exactly similar reasonings.
Simple calculations yield
$$\partial_\theta\ell(\theta;\uu)=M'(\theta)
-\frac{1}{\theta^2}\ln s_\theta(\uu)  -
\big(\frac{1}{\theta}+d\big) \frac{\sum_{j=1}^d u_j^{-\theta} \ln u_j }{s_\theta (\uu)}
+ \big(\sum_{j=1}^d \ln u_j\big)   .$$

First, let us show that
the map $\uu\mapsto \min_k \min(u_k,1-u_k)^\omega |\partial_\theta \ell(\theta_0;\uu)|$ is bounded on $(0,1)$ for any positive $\omega$. We will replace $\theta_0$ by $\theta$ hereafter, to weaken our notations.
W.l.o.g., assume that $u_1\leq u_2\leq \cdots \leq u_d$.
Then, $s_\theta(\uu) \in [ u_1^{-\theta}, du_1^{-\theta}-d+1]$.
We deduce
\begin{eqnarray*}
\lefteqn{
\min_k \min(u_k,1-u_k)^\omega |\partial_\theta \ell(\theta;\uu)| \leq
u_1^\omega \Big\{ |M'(\theta)| +
\frac{1}{\theta^2}\ln \big(du_1^{-\theta}-d+1 \big)  }\\
&+&
\big(\frac{1}{\theta}+d\big) \frac{d u_1^{-\theta} |\ln u_1| }{u_1^{-\theta}}
+ d |\ln u_1| \Big\} \\
 &= &
 O\Big( u_1^\omega (1+ |\ln u_1|) +
u_1^\omega\ln \big(du_1^{-\theta}-d+1 \big)   \Big),
\end{eqnarray*}
that is a bounded function of $\uu$.

\mds

Second, by simple calculations, it can be easily seen that $\partial_\theta \ell(\theta;d\uu)=h(\theta ;\uu) \, d\uu$, for some map $\uu\mapsto h(\theta;\uu)$ that
is a linear combination of the maps
$$D_0(\uu):=\prod_{j=1}^d u_j^{-\theta-1} s_\theta(\uu)^{-d},\;\; D_k(\uu):= \prod_{j=1,j\neq k}^d u_j^{-\theta-1} \frac{u_k^{-\theta-1}\ln u_k}{s_\theta(\uu)^{d}},$$
$$\text{and} \;\;D^*_k(\uu):= \prod_{j=1}^d u_j^{-\theta-1} \frac{u_k^{-\theta}\ln u_k}{s_\theta(\uu)^{d+1}},$$
for every $k\in \{1,\ldots,d\}$.

\mds

Let us check (\ref{bound_fg}) for all the latter maps.
Concerning $D_0$, this means we have to show
\begin{equation}
0\leq \int g_\omega(\uu) D_0(\uu)\, d\uu<\infty.
\label{Assump3_Clayton_D0}
\end{equation}
Indeed, by symmetry, we have
$$ \int g_\omega(\uu) D_0(\uu)\, d\uu
= d!\int_{u_1\leq u_2\leq \cdots \leq u_d}  g_\omega(\uu) D_0(\uu)\, d\uu
\leq
d!\int_{u_1\leq u_2\leq \cdots \leq u_d}  u_1^\omega
\frac{\prod_{j=1}^d u_j^{-\theta-1}}{ s_\theta(\uu)^{d}}\, d\uu .$$
By an integration w.r.t. $u_d \in (u_{d-1},1]$, the latter integral is smaller than a constant times
$$ \int_{u_1\leq u_2\leq \cdots \leq u_{d-1}}  u_1^\omega \frac{\prod_{j=1}^{d-1} u_j^{-\theta-1}}{ \big\{ \sum_{j=1}^{d-1}u_j^{-\theta} - d+2 \big\}^{d-1}}\, d\uu .$$
By integrating w.r.t. $u_{d-1},u_{d-3},\ldots,u_2$ successively, we obtain a constant times
$$ \int_0^1 u_1^\omega u_1^{-\theta-1} \frac{du_1}{u_1^{-\theta}-d+d}=\int_0^1 u_1^{\omega - 1}\, du_1 <\infty,$$
proving~(\ref{Assump3_Clayton_D0}).
To manage $\int g_\omega(\uu) D_k(\uu)\, d\uu$ for any $k$,
note that, for any $\epsilon >0$, there exist some positive constants $\alpha$ and $\beta$ s.t.
$$ |\ln t| \leq \alpha t^{-\epsilon}+\beta t^\epsilon,\;\; t>0.$$
Then, apply the same technique as for $D_0$. The upper bound is here reduced to a constant times $\int u_1^{\omega-\epsilon} \, du_1$, that is finite by choosing
$\epsilon <\theta$. The same ideas apply to deal with $D^*_k$:
$u_k^{-\theta}|\ln u_k|/s_\theta(\uu)\leq \gamma u_k^{-\epsilon}$ for some constant $\gamma$, and we recover the $D_k$ case.
We have then proved (\ref{bound_fg}) for $\Fc_1$.

\mds

The same technique can be applied to check (\ref{bound_Cn_1}), assuming $J_2\cup J_3\neq \emptyset$.
Denote $m_k:=\text{Card}(J_k)$, $k\in \{1,2,3\}$.
W.l.o.g., let us assume that the components indexed by $J_1$ are the first ones, i.e. are $u_1,\ldots,u_{m_1}$.
By simple calculations, it can be easily seen that $\partial_\theta \ell(\theta;\uu_{J_1}:\cc_{n,J_2}:\dd_{n,J_3}\big)
= h_{J_2,J_3}(\theta ;\uu_{J_1}) \, d\uu_{J_1}$, for some map $\uu_{J_1}\mapsto h_{J_2,J_3}(\theta;\uu_{J_1})$ whose absolute value
is smaller than a linear combination of the maps
$$\widetilde D_{0,J_1}(\uu_{J_1}):=\frac{1}{\widetilde s_\theta(\uu_{J_1})^{m_1}}\prod_{j=1}^{m_1} u_j^{-\theta-1} ,\;\;
\widetilde D^*_{0,J_1}(\uu_{J_1}):=\frac{(|J_2|n^\theta\ln n+1)}{\widetilde s_\theta(\uu_{J_1})^{m_1+1}}\prod_{j=1}^{m_1} u_j^{-\theta-1} ,$$
$$\widetilde D_{k,J_1}(\uu_{J_1}):= \frac{u_k^{-\theta-1}\ln u_k}{\widetilde s_\theta(\uu_{J_1})^{m_1}}\prod_{j=1,j\neq k}^{m_1} u_j^{-\theta-1} ,
\text{and} \;\widetilde D^*_{k,J_1}(\uu_{J_1}):=  \frac{u_k^{-\theta}\ln u_k}{\widetilde s_\theta(\uu_{J_1})^{m_1+1}}\prod_{j=1}^{m_1} u_j^{-\theta-1},$$
for every $k\in \{1,\ldots,m_1\}$, by setting
$$\widetilde s_\theta(\uu_{J_1}):= \sum_{j=1}^{m_1}u_j^{-\theta} + |J_2| (2n)^\theta + |J_3| (1-1/2n)^{-\theta} - d+1.$$
 We will check Assumption (\ref{bound_Cn_1}) for all latter maps.

\mds

For every $\uu \in \BB_{n,|J_1|}$, $J_1\neq \emptyset$, note that
$$ g_\omega(\uu_{J_1}) := g_\omega(\uu_{J_1}:\cc_{n,J_2}:\dd_{n,J_3}\big) = 1/(2n)^{\omega},$$
when $J_2\neq\emptyset$ or $m_1=1$.
For the moment, let us assume this is the case.

\mds

To deal with $\widetilde D_{0,J_1}$, we have by symmetry
\begin{eqnarray*}
\lefteqn{
 \int_{\BB_{n,|J_1|}}  g_\omega(\uu_{J_1}) \Big|\widetilde D_{0,J_1}\big(\uu_{J_1}\big)  \Big| \, d\uu_{J_1}}\nonumber \\
 &\leq &  m_1!
 \int_{\BB_{n,|J_1|},u_1\leq \cdots \leq u_{m_1}} g_\omega(\uu_{J_1})\Big|\widetilde D_{0,J_1}\big(\uu_{J_1}\big)  \Big| \, d\uu_{J_1} \nonumber \\
 &\leq &  \frac{m_1!}{(2n)^\omega}
 \int_{\BB_{n,|J_1|},u_1\leq \cdots \leq u_{m_1}}  \frac{\prod_{j=1}^{m_1} u_j^{-\theta-1}\,d\uu_{J_1}}{\big(\sum_{j=1}^{m_1} u_j^{-\theta} +
 |J_2| (2n)^{\theta} + |J_3| (1- 1/2n)^{-\theta}- d+1 \big)^{m_1}}  \cdot
\label{A2_D0_decomp}
\end{eqnarray*}
First integrate w.r.t. $u_{m_1}$ between $u_{m_1-1}$ and $1-1/2n$. The absolute value of the latter integral w.r.t. $u_{m_1}$ can be bounded by a constant times
\begin{eqnarray*}
\lefteqn{
  \frac{1}{\big(\sum_{j=1}^{m_1-1} u_j^{-\theta}
 +|J_2| (2n)^{\theta} + (|J_3|+1) (1- 1/2n)^{-\theta}- d+1 \big)^{m_1-1} }   }\\
 &+& \frac{1}{\big(\sum_{j=1}^{m_1-1} u_j^{-\theta} + u_{m_1-1}^{-\theta}
 +|J_2| (2n)^{\theta} + |J_3| (1- 1/2n)^{-\theta} - d+1\big)^{m_1-1} }    \\
 &\leq &
 \frac{2}{\big(\sum_{j=1}^{m_1-1} u_j^{-\theta}
 +|J_2| (2n)^{\theta} + |J_3| (1- 1/2n)^{-\theta} - d + 2\big)^{m_1-1}  }\cdot
\end{eqnarray*}
Then, successively integrate w.r.t. $u_{m_1-1}, u_{m_1-2},\ldots, u_2$ using the same type of upper bounds for every integral.
We finally obtain an $u_1$-integral of order
$$  \frac{1}{n^\omega}\int_{1/2n}^{1-1/2n}  u_1^{-\theta -1} \frac{du_1}{u_1^{-\theta}
+ |J_2| (2n)^{\theta} + |J_3| (1- 1/2n)^{-\theta}  - d + m_1},$$
that is $O(\ln n/n^\omega)$ and then tends to zero with $n$ for any $\omega>0$.
Thus, (A.2) is proven in the case of the integrand $\widetilde D_{0,J_1}$.
The terms $\widetilde D_{k,J_1}$ are managed similarly by noting that $|\ln u_k|\leq \ln (2n) $ when $u_k\in (1/2n;1-1/2n]$.

\mds

The task is similar with $\widetilde D^*_{0,J_1}$: after $m_1-1$ integrations w.r.t. $u_{m_1}$, $u_{m_1-1}$, etc, $u_2$, we get
\begin{eqnarray*}
\lefteqn{
  \frac{(|J_2|n^\theta \ln n +1)}{n^\omega}\int_{1/2n}^{1-1/2n} u_1^{-\theta -1} \frac{du_1}{\big( u_1^{-\theta}
+ |J_2| (2n)^{\theta} + |J_3| (1- 1/2n)^{-\theta}  - d + m_1 \big)^2} }\\
&=&
  O \Big(  \frac{n^{\theta-\omega} \ln n}{
 (|J_2|+1) (2n)^{\theta} + |J_3| (1- 1/2n)^{-\theta}  - d + m_1 }   \Big)=  O \Big(  \frac{n^{\theta-\omega} \ln n}{ n^{\theta}  }   \Big)=o(1),
\end{eqnarray*}
for any $\omega>0$.
The terms $\widetilde D^*_{k,J_1}$ are managed similarly by noting that $|u_k^{-\theta}\ln u_k| = O(n^\theta\ln n) $ when $u_k\in (1/2n;1-1/2n]$.

\mds

Thus, it remains to consider the case $J_2=\emptyset$ and $m_1\geq 2$ to check (A.2).
Apply the same technique as above, invoking $g_\omega(\uu)\leq u_1^\omega$ for every $\uu\in [0,1]^d$.
Moreover, after every integration stage, it is possible to replace $1-1/2n$ by $1$ in the denominators. For instance, in the case of $\widetilde D_{0,J_1}$,
the integration w.r.t. $u_{m_1}$ yields
\begin{eqnarray*}
\lefteqn{
  \frac{1}{\big(\sum_{j=1}^{m_1-1} u_j^{-\theta}
 + (|J_3|+1) (1- 1/2n)^{-\theta}- d+1 \big)^{m_1-1} }   }\\
 &+& \frac{1}{\big(\sum_{j=1}^{m_1-1} u_j^{-\theta} + u_{m_1-1}^{-\theta}
 + |J_3| (1- 1/2n)^{-\theta} - d+1\big)^{m_1-1} }    \\
 &\leq &
 \frac{2}{\big(\sum_{j=1}^{m_1-1} u_j^{-\theta}
 -m_1+ 2\big)^{m_1-1}  },
\end{eqnarray*}
because $|J_3|=d-m_1$ in our case.
After $m_1-1$ integration stages, we obtain
$$  \int_{1/2n}^{1-1/2n}  u_1^{\omega-\theta -1} \frac{du_1}{u_1^{-\theta}},$$
that is finite, as required.
This is still the case for term $\widetilde D_{0,J_1}^*$, obviously, when $|J_2|=0$ because $\widetilde s_\theta(\uu)\geq 1$.
The terms $\widetilde D_{k,J_1}$ are managed similarly because $|\ln u_k|\leq |\ln u_1| $ for every $k$ and
$  \int_{1/2n}^{1-1/2n} |\ln u_1| u_1^{\omega-1} \, du_1 =O(1) .$
The same ideas apply with the terms $\widetilde D^*_{k,J_1}$, because
$\widetilde D^*_{k,J_1}(\uu) \leq \widetilde D_{0,J_1}(\uu)|\ln u_1| .$
To conclude, we have checked (\ref{bound_Cn_1}) in every situation.

\mds

Since Assumption \ref{assumption_regularity} is now satisfied with $\ell(\theta;\uu)$, we have proven the validity of our assumptions in the case of the Clayton family.

\begin{remark}
As for the Gumbel copula family, we have proved the regularity assumptions for the Clayton family as if the weight function were replaced by $\uu\mapsto \min_{j}u_j$, a stronger property.
\end{remark}

\section{Asymptotic properties for parameters of varying dimensions}
\label{asym_prop_pn}

In this section, we deal with the case of copula parameters whose dimensions are functions of the sample size. 
More formally, we consider a sequence of parametric copula models
$\Pc_n:=\{\Pb_{\theta_n}, \,\theta_n \in \Theta_n\}$, for some subsets $\Theta_n\subset \Rb^{p_n}$, $n\geq 1$.
Therefore, the number of unknown parameters $p_n$ may vary with the sample size $n$. In particular, the sequence $(p_n)$ could tend
to the infinity with $n$, i.e. $p=p_n\rightarrow \infty$ when $n\rightarrow \infty$, but it is not mandatory.

\mds

As a consequence, in this section only, we introduce a sequence of loss functions $(\ell_n)_{n\geq 1}$, which hereafter enters in an associated map $\Lb_n$ 
(whose notation remains unchanged, to simplify): for every $n$, the map $\ell_n : \Theta_n \times (0,1)^{d} \rightarrow \Rb$ defines the ``global loss'' map
 \begin{equation}
 \Lb_n(\theta_n;\uu_1,\ldots,\uu_n) := \overset{n}{\underset{i=1}{\sum}} \ell_n(\theta_n;\uu_i),
\label{loss_as_a_sum_diverging}
 \end{equation}
for every $\theta_n\in\Theta_n$ and every $(\uu_1,\ldots,\uu_n)$ in $(0,1)^{dn}$.
\begin{assumption}
\textcolor{black}{For every $n$, the parameter space $\Theta_n$ is a borelian subset of $\Rb^{p_n}$. 
The function $\theta_n \mapsto \Eb[\ell_n(\theta_n;\UU)]$ is uniquely minimized on $\Theta_n$ at $\theta_{n,0}$, and an open neighborhood of $\theta_{n,0}$ is contained in $\Theta_n$.}
\label{assump_theta0_diverging}
\end{assumption}
\textcolor{black}{Exactly as detailed in the core of the main text (see the discussion after Assumption \ref{assump_theta0}), our theory applies when $\theta_{n,0}$ belongs to the frontier of $\Theta_n$. Technical details are left to the reader.} Let us illustrate the relevance of the diverging dimension case for copula selection. 
\begin{example}[Example~\ref{ex_copula_mixture} cont'd]
\label{mixture_cop_ex}
Consider an infinite sequence of given copulas $(C^{(n)})_{k\geq 1}$ in dimension $d$.
The true underlying copula will be estimated by a sequence of finite mixtures, given by the weighted sum of the first $p_n+1$ copulas $C^{(j)}$, $j\in\{1,\ldots,p_n+1\}$. If we choose the CML method, the loss function $\ell_n(\theta_n;\cdot)$ is the log-copula density associated with $\sum_{k=1}^{p_n+1} \theta_{n,k} C^{(k)}$. 
In theory, we could extend the latter framework by assuming that every $C^{(k)}$ depends on an unknown parameter that has to be estimated in addition to the weights, but the numerical estimation of the enlarged unknown vector of parameters will surely become numerically challenging.
\end{example}
\begin{example}
A probably more relevant application is related to single-index copulas with a diverging number of underlying factors and a known link function $\zeta$: in the same spirit as Example~\ref{ex_single_index_sparsity}, the conditional copula of $\XX\in \Rb^d$ given $\ZZ=\zz\in \Rb^{m_n}$ would be
the $d$-dimension copula $C_{\zeta(\zz^\top\beta_n)}$ for some parameter $\beta_n\in \Rb^{m_n}$ to be estimated and a given parametric copula family 
$\Cc:=\{ C_\theta; \theta \in \Theta\subset \Rb^p\}$.  
\end{example}
We have now to consider a sequence of estimators $(\widehat \theta_n)_{n\geq 1}$ defined as
\begin{equation} \label{obj_crit_n}
\widehat{\theta}_n \, \textcolor{black}{\in} \, \underset{\theta_n \in \Theta_n}{\arg \; \min} \; \Big\{ \Lb_n(\theta_n;\widehat{\Uc}_n) + n
\overset{p_n}{\underset{k=1}{\sum}}\pp(\lambda_n,|\theta_{n,k}|)\Big\}.
\end{equation}
We will focus on the distance between $\theta_{n,0}$ and $\widehat{\theta}_n$.
\begin{remark}
\label{rem_diverging_nbr_param}
Note that we do not evaluate the distance between $\theta_{n,0}$ (or $\widehat{\theta}_n$) and an hypothetical ``true parameter'' $\theta_0$, because they generally do not live in the same spaces. 
In some cases, it is possible to ``dive'' $\Theta_n$ into $\Theta$ and/or $\Pc_n$ into a (correctly specified) parametric family of copulas. 
To illustrate, in Example~\ref{mixture_cop_ex}, assume the true copula is an infinite sum of known copulas, i.e. $C_0=\sum_{k=1}^{+\infty} \pi_k C^{(k)}$. 
Setting $\theta_0:=(\pi_1,\pi_2,\ldots)$, some identification constraints have to be found and they depend on the selected copulas $C^{(k)}$.  
Therefore, it would be possible to compare the two infinite sequences $\theta_0$ and $\bar\theta_{n,0}:=(\theta_{n,0},0,0,\ldots)$. Nonetheless, in such cases, 
the distance between $\theta_{0}$ and $\bar\theta_{n,0}$ is strongly model-dependent. Since this is no longer an inference problem but rather a problem of model specification, we will not go further in this direction.
\end{remark}
In terms of notations, for every $n$, the sparse subset of parameters is
$\Ac_n:=\big\{k: \theta_{n,0,k}\neq 0, k = 1,\ldots,p_n\big\}$, and \textcolor{black}{$s_n$ will denote} the cardinality of $\Ac_n$.
Let us rewrite some of our latter assumptions that have to be adapted to the new framework.
\begin{assumption}\label{regularity_assumption_n}
The map $\theta_n\mapsto \ell_n(\theta_n;\uu)$ is thrice differentiable on $\Theta_n$, for every $\uu \in (0,1)^d$.
\textcolor{black}{Any pseudo-true value $\theta_{n,0}$ satisfies  
the first-order condition, i.e. $\Eb[ \nabla_{\theta_n}\ell_n(\theta_{n,0};\UU)] = 0$.} 
Moreover, 
$\Hb_n := \Eb[\nabla^2_{\theta_n\theta_n^\top}\ell_n(\theta_{n,0};\UU)]$ and $\Mb_n := \Eb[\nabla_{\theta_n}\ell_n(\theta_{n,0};\UU)
\nabla_{\theta_n^\top}\ell_n(\theta_{n,0};\UU)]$ exist, are positive definite and $\sup_n \|\Hb_n\|_\infty < \infty$. 
Denoting by $\lambda_1(\Hb_n)$ the smallest eigenvalue of $\Hb_n$,
there exists a positive constant $\underline{\lambda}$ such that $\lambda_1(\Hb_n)\geq \underline{\lambda}>0$ for every $n$. Finally, for every $\epsilon>0$, 
there exists a positive constant $K_\epsilon$ such that
$$ \sup_n \sup_{\{\theta_n; \|\theta_n - \theta_{n,0}\|<\eps\}} \;\sup_{j,l,m}\big|\Eb[\partial^3_{\theta_{n,j}\theta_{n,l}\theta_{n,m}}
\ell_n( \theta_n;\UU)] \big| \leq K_\epsilon.$$
\end{assumption}
\begin{assumption}
\textcolor{black}{For some $\omega$,} the family of maps $\Fc:=\bigcup_{n\geq 1} \Fc_n$, $\Fc_n:=\Fc_{1,n}\cup\Fc_{2,n}\cup\Fc_{3,n}$, from $(0,1)^d$ to $\Rb$ is $g_\omega$-regular, with
$$\Fc_{n,1}:=\{ f:\uu\mapsto \partial_{\theta_{n,k}}\ell_n(\theta_{n,0};\uu); k=1,\ldots,p_n\},$$
$$\Fc_{n,2}:=\{ f:\uu\mapsto \partial^2_{\theta_{n,k},\theta_{n,l}}\ell_n(\theta_{n,0};\uu); k,l=1,\ldots,p_n\},$$
$$\Fc_{n,3}:=\{ f:\uu\mapsto \partial^3_{\theta_{n,k},\theta_{n,l},\theta_{n,j}}\ell_n(\theta_n;\uu); k,l,j=1,\ldots,p_n, \; \|\theta_n-\theta_{n,0}\|<K\},$$
for some constant $K>0$.
\label{assumption_regularity_n}
\end{assumption}
\begin{assumption}\label{assumption_regularizer_n}
Define
$$a_n := \max_{1\leq j\leq p_n} \big\{ \partial_2  \pp(\lambda_n,|\theta_{n,0,j}|),\theta_{n,0,j}\neq 0 \big\}\;\text{and}
\;b_n := \max_{1\leq j\leq p_n} \big\{ \partial^2_{2,2} \pp(\lambda_n,|\theta_{n,0,j}|),\theta_{n,0,j}\neq 0 \big\}.$$
We assume that $a_n\rightarrow 0$ and $b_n\rightarrow 0$ when $n\rightarrow \infty$. Moreover, there exist some constants $M$ and $D$ such that
$ | \partial^2_{2,2} \pp(\lambda_n,\theta_1) - \partial^2_{2,2} \pp(\lambda_n,\theta_2) |\leq D |\theta_1 -\theta_2 |,$
for any real numbers $\theta_1,\theta_2$ such that $\theta_1,\theta_2 > M\lambda_n  $.
\end{assumption}
This set of assumptions extends the regularity conditions of Section~\ref{asym_prop} to the diverging dimension case. In particular, our assumptions~\ref{regularity_assumption_n} and~\ref{assumption_regularizer_n} are in the same vein as in~\citep{fan2004nonconcave}, assumption (F) and condition 3.1.1. respectively. Note that our Assumption~\ref{cond_reg_copula} in the main text does not need to be altered as $d$ remains fixed.

\begin{theorem}\label{bound_proba_n}
\textcolor{black}{Suppose Assumptions \ref{oscillation_modulus_assump}-\ref{tail_empir_processes} given in Appendix \ref{technicalities} are satisfied. Let some $$\omega \in \Big(0,\min\big\{\frac{\textcolor{black}{\kappa}_1}{2(1-\textcolor{black}{\kappa}_1)},\frac{\textcolor{black}{\kappa}_2}{2(1-\textcolor{black}{\kappa}_2)},\textcolor{black}{\kappa}_3 - \frac{1}{2}\big\}\Big),$$ 
and suppose Assumptions~\ref{cond_reg_copula},~\ref{assump_theta0_diverging}-\ref{assumption_regularizer_n} hold for this $\omega$.
} 
Finally, assume that $p_n^2 \ln(\ln n)/\sqrt{n}\rightarrow 0$ and $p_n^2 a_n\rightarrow 0$ when $n\rightarrow \infty$.
Then, there exists a sequence $(\widehat{\theta}_n)_{n\geq 1}$ of solutions of (\ref{obj_crit_n}) that satisfies the bound
\begin{equation*}
\|\widehat{\theta}_n-\theta_{n,0}\|_2 = O_p\Big( \sqrt{p_n} \big(  n^{-1/2}\ln(\ln n) + a_{n} \big)  \Big).
\end{equation*}
\end{theorem}
Note that $p_n$ is allowed to diverge to the infinity, but not too fast, since $p_n=o\big(n^{1/4}\big)$ by assumption.
\begin{proof}
The arguments are exactly the same as those given for the proof of Theorem \ref{bound_proba}.
With the same notations, the dominant term in the expansion comes from $T_2$ and is larger than $n\nu_n^2 \underline{\lambda} \| \mathbf{v} \|_2^2/2$, for some vector $\mathbf{v}\in \Rb^{p_n}$, $\|\mathbf{v}\|_2=L_\epsilon$. By a careful inspection of the previous proof, the result follows if we satisfy
$$ \sqrt{n p_n} \ln(\ln n) \nu_n \| \mathbf{v} \|_2 << n\nu_n^2 \| \mathbf{v} \|_2^2,$$
$$ \sqrt{n} p_n \ln(\ln n) \nu_n^2 \| \mathbf{v} \|_2^2 << n\nu_n^2 \| \mathbf{v} \|_2^2,$$
$$ n p_n^{3/2} \nu_n^3 \| \mathbf{v} \|_2^3 << n\nu_n^2 \| \mathbf{v} \|_2^2, \; \text{and}$$
$$ \sqrt{\|\theta_{n,0}\|_0} n  \nu_n a_n \| \mathbf{v} \|_2 + n b_n \nu_n^2 \| \mathbf{v} \|_2^2 << n\nu_n^2 \| \mathbf{v} \|_2^2.$$
The latter conditions will be satisfied under our assumptions, choosing some vectors $\mathbf{v}$ whose norm
$L_\epsilon$ is sufficiently large and setting $\nu_n := \sqrt{p_n} \big(  n^{-1/2} \ln(\ln n)+ a_{n} \big)  $.
\end{proof}
As in the fixed dimension case, we establish the asymptotic oracle property, i.e. the conditions for which the true support is recovered and the non-zero coefficients are asymptotically normal. We denote by 
$\widehat{\Ac}_n$ the estimated support of our estimator for the $n$-th model, i.e. $\widehat{\Ac}_n:=\big\{k: \widehat\theta_{n,k}\neq 0; k = 1,\ldots,p_n\big\}$. For convenience and w.l.o.g., we assume that the supports are related to the first components of the true parameters, i.e. $\Ac_n=\{1,\ldots,s_n\}$ and $\Ac_n^c=\{s_n+1,\ldots,p_n\}$. Therefore, every (true or estimated) parameter will be split as $\theta_n=:(\theta_{n}^{(1)},\theta_{n}^{(2)})$, where $\theta_{n}^{(1)}$ (resp. $\theta_{n}^{(2)}$) is related to the $\Ac_n$ (resp. $\Ac_n^c$). components.
The statement of the asymptotic distribution with a diverging dimension requires the introduction of a sequence of deterministic real matrices $(Q_n)_{n\geq 1}$, $Q_n$ being of size $q\times s_n$, for some fixed $q>0$.
Denote $ Q_n:=[q_{n,i,j}]_{1\leq l \leq q,1\leq r\leq s_n}$. Define the $q$ sequences of maps $(w_n^{(l)})_{n\geq 1}$, $l\in\{1,\ldots,q\}$, from $\Theta_n\times (0,1)^d$ to $\Rb$ by
\begin{equation*}
w_{n}^{(l)}(\theta_{n};\uu) := \sum_{r=1}^{s_n} q_{n,l,r} \partial_{\theta_{n,r}} \ell_n (\theta_{n};\uu), \; \uu \in (0,1)^d.
\end{equation*}
In addition to Condition~\ref{assumption_regularity_n}, the next assumption allows to obtain the $g_\omega$-regularity of the latter maps.
\begin{assumption}
$ \sup_n \sup_{1\leq l \leq q} \sum_{r=1}^{s_n} |q_{n,l,r}| < \infty.$
\label{Qn_regular_sum}
\end{assumption}
Moreover, we need to introduce a limit for the sequences of maps $(w_n^{(l)}(\theta_{n,0},\cdot))$.
\begin{assumption}
There exist $q$ maps $w_\infty^{(l)}:(0,1)^d\rightarrow \Rb$ that are $g_\omega$-regular and such that
$$ \sup_{1\leq l \leq q} \Eb\Big[ \big(w_n^{(l)}(\theta_{n,0},\UU) - w_\infty^{(l)}(\UU) \big)^2 \Big] \rightarrow 0 \; \text{as} \; n \rightarrow \infty.$$
\label{regularity_AN_Qn} 
\end{assumption}
Denote $\Wc_\infty:= \{ w_\infty^{(l)}, l=1,\dots,q\}$. The use of the sequence of matrices $(Q_n)$ is classical and inspired here by Theorem 2 of~\citep{fan2004nonconcave}. This technicality allows to obtain convergence towards a finite $q$-dimensional distribution. A similar technique was employed in Theorem 3.2 of~\citep{portnoy1985}, which established the normal distribution of the least squares based M-estimator when the dimension diverges. Our regularity assumptions differ from the latter ones, due to different techniques of proofs. For instance, Theorem 2 of~\citep{fan2004nonconcave} assume $p_n^5/n=o(1)$ and imposes the convergence of $Q_n Q_n^\top$. In our case, we need $p_n^4/n=o(1)$ and the boundedness of the sequence of row-sum norms $\| Q_n \|_{\text{row}}$ (Assumption~\ref{Qn_regular_sum}).
In light of our criterion, since Theorem~\ref{Th_fondam_multiv_rank_stat} will be applied to $n^{-1/2}\sum^{n}_{i=1}\sum_{r\in \Ac_n} q_{n,l,r} \partial_{\theta_{n,r}} \ell_n (\theta_{n,0},\widehat{\UU}_i)$, $l \in \{1,\ldots,q\}$, contrary to our Theorem~\ref{oracle_property} that applied the same corollary to $n^{-1/2}\sum^{n}_{i=1}\partial_{\theta_{j}} \ell(\theta_{0},\widehat{\UU}_i)$, $j \in \Ac$, this motivates Assumptions~\ref{Qn_regular_sum} and~\ref{regularity_AN_Qn}. 


\begin{theorem}\label{oracle_property_n}
In addition to the conditions of Theorem~\ref{bound_proba_n},
assume that $\lambda_n \rightarrow 0$, $p_n a_n=o(\lambda_n)$, $\sqrt{n} \lambda_n / (p_n\ln(\ln n)) \rightarrow \infty$ and 
$\underset{n \rightarrow \infty}{\lim \, \inf} \; \underset{x \rightarrow 0^+}{\lim \; \inf} \, \lambda^{-1}_n \partial_2\pp(\lambda_n,x) > 0$.
Then, the consistent estimator $\widehat{\theta}_n$ given by~(\ref{obj_crit_n}) satisfies the following properties.
\begin{itemize}
    \item[(i)] \textcolor{black}{Sparsity}: $\underset{n \rightarrow \infty}{\lim} \; \Pb(\widehat\theta_{n}^{(2)} = \theta_{n,0}^{(2)}) = 1$.
    \item[(ii)] Asymptotic normality: in addition, assume $\sqrt{n}\lambda_n^2=o(1)$ and 
    Conditions~\ref{Qn_regular_sum}-\ref{regularity_AN_Qn} hold. Moreover,~\ref{cond_wc_empir_process}-\ref{regularity_cond_wc_stronger} in Appendix \ref{technicalities} of the main \textcolor{black}{text} are met, replacing $\Fc$ with $\bigcup_n\Fc_{1n}\cup \Wc_\infty$. 
    Then, we have
    \begin{equation*}
    \sqrt{n}Q_n\Big[\Hb_{n,\Ac_n\Ac_n}+\mathbf{B}_n(\theta_{n,0})\Big]\Big\{\big(\widehat{\theta}^{(1)}_n-\theta_{n,0}^{(1)}\big) + \big[\Hb_{n,\Ac_n\Ac_n}+\mathbf{B}_n(\theta_{n,0})\big]^{-1}\mathbf{A}_n(\theta_{n,0}) \Big\} \overset{d}{\underset{n \rightarrow \infty}{\longrightarrow}} \YY,
    \end{equation*}
    where $\Hb_{n,\Ac_n\Ac_n} := \Big[\Eb\big[ \partial^2_{\theta_{n,k}\theta_{n,l}}\ell(\theta_{n,0};\UU)\big] \Big]_{k,l \in \Ac_n}$,
    $\normalfont\mathbf{A}_n(\theta_{n}) = \big[\partial_2\pp(\lambda_n,|\theta_{n,k}|)\text{sgn}(\theta_{n,k})\big]_{k \in \Ac_n}$,
    $\mathbf{B}_n(\theta_n) =\normalfont \text{diag}(\partial^2_{2,2}\pp(\lambda_n,|\theta_{n,k}|),\,k\in\Ac_n)$ and $\YY$ a $q$-dimensional random vector whose $j$-th component, $j\in \{1,\ldots,q\}$, is
    \begin{equation*}
    Y_j := (-1)^d\int_{(0,1]^d} \Cb(\uu) \, w_\infty^{(j)}(d\uu)+
\sum_{ \substack{I \subset \{1,\ldots,d\} \\ I\neq \emptyset, I\neq \{1,\ldots,d\} } } (-1)^{|I|}
  \int_{(0,1]^{|I|}} \Cb(\uu_{I};\1_{-I}) \, w_\infty^{(j)}(d\uu_{I};\1_{-I}).
    \end{equation*}
\end{itemize}
\end{theorem}
\textcolor{black}{Note that the property (i) is related to the zero coefficients of the true parameters $\theta_{n,0}$, as in~\citep{fan2004nonconcave}. 
There is a subtle difference with the oracle property established for a fixed dimension, where both true zero and non-zero coefficients are correctly identified, a property termed ``support recovery''. Herer, the so-called ``sparsity property'' actually does not preclude the possibility that, for every $n$, some components of $\theta_{n,0}$'s support may be estimated as zero.}

\begin{remark}
Assume the minimum signal condition (H) of~\citep{fan2004nonconcave}, i.e., $\min_{k \in \Ac_n}|\theta_{n,0,k}|/\lambda_n \rightarrow \infty$ as $n \rightarrow \infty$ - a condition standard for sparse estimation with diverging dimensions - is satisfied. 
Such a property is closely related to the unbiasedness property for non-convex penalization: see, e.g., condition 2.2.-(vii) of~\citep{loh2017}.  
Then, if the penalty is SCAD or MCP, the quantities $\mathbf{A}_n(\theta_{n,0}), \mathbf{B}_{n}(\theta_{n,0})$ are zero when $n$ is sufficiently large. Therefore, the conclusion of Theorem~\ref{oracle_property_n} (ii) becomes 
    \begin{equation*}
    \sqrt{n}Q_n \Hb_{n,\Ac_n\Ac_n} \big(\widehat{\theta}^{(1)}_n-\theta_{n,0}^{(1)}\big) 
    \overset{d}{\underset{n \rightarrow \infty}{\longrightarrow}} \YY.
    \end{equation*}
\end{remark}
\begin{remark}
\textcolor{black}{It can be checked that our assumptions in Theorem~\ref{oracle_property_n} can be satisfied by some sequences $(p_n,\lambda_n,a_n)$. 
Restricting ourselves to some powers of $n$, set $p_n=[ n^a]$, $\lambda_n=n^{-b}$ and $a_n=n^{-c}$, for some positive constants $a$, $b$ and $c$.
The subset $$\{ (a,b,c) \in \Rb^3 \,|\, b>\frac{1}{4}, 0<a <\frac{1}{4}, a+b<\min(\frac{1}{2},c) \}$$
yields an acceptable choice for $(p_n,\lambda_n,a_n)$.}
\end{remark}
\begin{proof}
\emph{Point (i):} The proof of (i) follows exactly the same lines as the proof of the first point in Theorem~\ref{oracle_property}.
Due to the diverging number of parameters and Theorem~\ref{bound_proba_n}, we now consider $\nu_n:=\sqrt{p_n}\big\{n^{-1/2}\ln(\ln n)+a_n\}$.
By the same reasoning as above, particularly Equation~(\ref{bound_oracle}), the result is proved if we satisfy
\begin{equation}
\ln(\ln n)\sqrt{n}  + n\sqrt{p_n}\nu_n + np_n\nu_n^2  <<
n \inf_{j \in \Ac^c_n}  \partial_2 \pp(\lambda_n,|\widehat\theta_{n,j}|)\text{sgn}(\widehat\theta_{n,j}),
\label{oracle_cond_n}
\end{equation}
keeping in mind that the estimated parameters $\widehat\theta_n$ we consider satisfies $\max\{ |\widehat\theta_{n,j}|; j \in \Ac^c_n \}=O_P( \nu_n)=o_P(1)$.
Note that we have invoked an uniform upper bound for the second and third order partial derivatives of the loss (Assumption~\ref{regularity_assumption_n}).
It can be easily checked that~(\ref{oracle_cond_n}) is satisfied under our assumptions on the sequence $(p_n,\lambda_n,a_n)$.
\textcolor{black}{Indeed,~(\ref{oracle_cond_n}) if satisfied if 
$$\frac{\ln(\ln n)}{\sqrt{n}}  + \sqrt{p_n}\nu_n + p_n\nu_n^2  << \lambda_n  .$$
Since $p_n \ln (\ln n)/\sqrt{n}=o(1)$ and $p_n a_n=o(1)$, then $p_n\nu_n^2 = o(\sqrt{p_n}\nu_n)$. 
Is is then sufficient to check that
$$ p_n \big( \frac{\ln (\ln n)}{\sqrt{n}} + a_n \big) << \lambda_n ,$$
that is a direct consequence of our assumptions.
}
\mds

\noindent \emph{Point (ii):} We proved \textcolor{black}{the sparsity property $\lim_{n \rightarrow \infty} \; \Pb(\widehat\theta_{n}^{(2)} = \theta_{n,0}^{(2)}) = 1$}. Therefore, for any $\epsilon>0$,
the event $\widehat{\theta}_{n,\Ac_n^c}=\mathbf{0} \in \Rb^{|\Ac_n^c|}$ occurs with a probability larger than $1-\epsilon$ for $n$ large enough.
Since we want to state a convergence in law result, we can consider that the latter event is satisfied everywhere.
By a Taylor expansion around the true parameter, as in the proof of Theorem~\ref{oracle_property}, and after multiplying by the matrix $Q_n$,
we get
\begin{eqnarray*}
\lefteqn{\sqrt{n}Q_n\Kb_n(\theta_{n,0})\Big\{\big(\widehat{\theta}_n-\theta_{n,0}\big)_{\Ac_n} + \mathbf{A}_n(\theta_{n,0})\Big\} =  - \frac{1}{\sqrt{n}}Q_n\nabla_{\theta_{n,\Ac_n}} \Lb_n(\theta_{n,0};\widehat{\Uc}_n) }\\
&-&  \frac{1}{2\sqrt{n}}Q_n\nabla_{\theta_{n,\Ac_n}}
\Big\{(\widehat\theta_n-\theta_{n,0})^{\top}_{\Ac_n} \nabla^2_{\theta_{\Ac_n}\theta^{\top}_{\Ac_n}} \Lb_n(\overline{\theta}_n;\widehat{\Uc}_n) \Big\}(\widehat\theta_n-\theta_{n,0})_{\Ac_n}+o_p(1) \\
& =: & R_1+R_2+o_p(1),
\end{eqnarray*}
where $\Kb_n(\theta_{n,0}) := n^{-1}\nabla^2_{\theta_{n,\Ac_n}\theta^\top_{n,\Ac_n}} \Lb_n(\theta_{n,0};\widehat{\Uc}_n)+\mathbf{B}_n(\theta_0)$.
Obviously, $\overline{\theta}_n$ is a random parameter such that $\|\overline{\theta}_{n,\Ac_n} - \theta_{n,0,\Ac_n}\|_2 < \|\widehat{\theta}_{n,\Ac_n} - \theta_{n,0,\Ac_n}\|_2$.
Due to Assumptions~\ref{assumption_regularity_n} and~\ref{Qn_regular_sum}, the family $\bigcup_{n\geq 1} \Gc_n$ of maps from $(0,1)^d$ to $\Rb$ defined as
$$\Gc_{n}:=\{ f:\uu\mapsto \sum_{r\in \Ac_n} q_{n,l,r} \partial^2_{\theta_{n,r}\theta_{n,k}}\ell_n(\theta_{n,0};\uu); l=1,\ldots,q; k\in \Ac_n\},\, n\geq 1,$$
is $g_\omega$-regular, with the same $\omega$ as in Assumption~\ref{assumption_regularity_n}.
Invoking Corollary~\ref{cor_GC}, we obtain
\begin{equation*}
\| Q_n\Kb_n(\theta_{n,0})- Q_n\Hb_{\Ac_n\Ac_n^\top}- Q_n \mathbf{B}_n(\theta_0)\|_\infty= o_p(1).
\label{approx_QnKn}
\end{equation*} 
Second, the third order term $R_2$ is a vector as size $q$ whose $i$-th component, $i\in \{1,\ldots,q\}$, is
$$ R_{2,i}:=-\frac{1}{2\sqrt{n}}\underset{j\in \Ac_n}{\sum} q_{n,i,j} \underset{l,m \in \Ac_n}{\sum}\partial^3_{\theta_j \theta_l\theta_m} \Lb_n(\overline{\theta}_n;\widehat{\Uc}_n) (\widehat\theta_{n,l}-\theta_{n,0,l})(\widehat\theta_{n,m}-\theta_{n,0,m}). $$
Conditions~\ref{assumption_regularity_n} and~\ref{Qn_regular_sum} imply that the maps 
$$ \uu \rightarrow 
\underset{j\in \Ac_n}{\sum} q_{n,i,j} \partial^3_{\theta_j \theta_l\theta_m} \pp_n(\overline{\theta}_n;\uu), \; 
i\in \{1,\ldots,q\}, l,m \in \Ac_n,$$
are $g_\omega$-regular. 
Then, apply Corollary~\ref{cor_GC} to the latter family. This yields $R_{2} - \bar R_{2}=O_P(\ln (\ln n)/\sqrt{n})$, by setting 
$$ \bar R_{2,i} := -\frac{n}{2\sqrt{n}}\underset{j\in \Ac_n}{\sum} q_{n,i,j} \underset{l,m \in \Ac_n}{\sum} 
\Eb\big[ \partial^3_{\theta_j \theta_l\theta_m}\ell_n(\overline{\theta}_n;\UU)\big] (\widehat\theta_{n,l}-\theta_{n,0,l})(\widehat\theta_{n,m}-\theta_{n,0,m}),$$
for every $i\in \{1,\ldots,q\}$.
By Assumption~\ref{regularity_assumption_n}, we get 
$$ \bar R_{2} =O_P\Big(  \sqrt{n}\| \widehat \theta_n - \theta_{n,0} \|_1^2 \Big) = O_P\Big(  \sqrt{n}p_n\| \widehat \theta_n - \theta_{n,0} \|_2^2 \Big)= O_P\Big( \sqrt{n}p^2_n \big(\frac{\ln(\ln n)}{n} + a_n^2 \big) \Big).$$
With our assumptions about $(a_n,\lambda_n,p_n)$, we obtain $R_2=o_P(1)$.

It remains to show that $R_1$ is asymptotically normal. To this end, note that 
$$  R_{1,j}=- \frac{1}{\sqrt{n}}\sum_{i=1}^n \underset{r \in \Ac_n}{\sum} q_{n,j,r} \partial_{\theta_{n,r}}\ell_n (\theta_{n,0};\widehat{\UU}_i) 
= - \frac{1}{\sqrt{n}}\sum_{i=1}^n w_n^{(j)}(\theta_{n,0}; \widehat \UU_i), $$
for any $j\in \{1,\ldots,q\}$.
Setting 
$w_{n,0}(\cdot):=\big[w_n^{(1)}(\theta_{n,0};\cdot),\ldots ,w_n^{(1)}(\theta_{n,0};\cdot)\big],$
this implies that $R_1=-\sqrt{n}\int w_{n,0} \, d\widehat \Cb_n $, due the first order conditions. 

\mds 

Now, denote $\Wc:=\bigcup_{n\geq 1}\{ w_n^{(l)}(\theta_{n,0},\cdot), l=1,\dots,q\} \cup \{ w_\infty^{(l)}, l=1,\dots,q\}$.
Remind that $\Wc_\infty$ is $g_\omega$-regular (Assumption~\ref{regularity_AN_Qn}). 
Thus, the family of maps $\Wc$ is $g_\omega $-regular, due to the $g_\omega$-regularity of $\bigcup_n \Fc_{n,1}$ (Assumption~\ref{assumption_regularity_n}) and Assumption~\ref{Qn_regular_sum}. 
Moreover, $\Wc$ satisfies Assumption~\ref{regularity_cond_wc_stronger} (replacing $\Fc$ with $\Wc$). 
Therefore, all the conditions of application of Theorem~\ref{Th_fondam_multiv_rank_stat} (iii) are fullfilled and
$\widehat \Cb_n$ is weakly convergent in $\ell^\infty(\Wc)$.
By the stochastic equicontinuity of the process $\widehat\Cb_n$ defined on $\Wc$, 
the random vector $ \int w_n \, d\widehat \Cb_n $ weakly tends to $ \int w_\infty \, d\widehat \Cb_n $, with obvious notations.

\mds

We then conclude with an application of Slutsky's Theorem to deduce the following asymptotic distribution:
\begin{eqnarray*}
\sqrt{n}Q_n\Big[\Hb_{n,\Ac_n\Ac_n}+\mathbf{B}_n(\theta_0)\Big]\Big\{\big(\widehat{\theta}_n-\theta_{n,0}\big)_{\Ac_n} + \Big[\Hb_{n,\Ac_n\Ac_n}+\mathbf{B}_n(\theta_{n,0})\Big]^{-1}\mathbf{A}_n(\theta_{n,0})\Big\}
& \overset{d}{\underset{n \rightarrow \infty}{\longrightarrow}} & \mathbf{Y},
\end{eqnarray*}
where $\mathbf{Y}$ is the $q$-dimensional random vector defined in the statement of the theorem.
\end{proof}

\section{Additional simulated experiment}\label{additional_experiment_gaussian_cop}

In this subsection, we investigate how the calibration of $a_{\text{scad}}$ and $b_{\text{mcp}}$ alters the performances of the SCAD and MCP penalty functions, respectively. Following the experiment on sparse Gaussian copula performed in the main \textcolor{black}{text}, we set $d=10$ and specify \textcolor{black}{two true sparse correlation matrices $\Sigma_{0,1}, \Sigma_{0,2}$. The true parameters $\theta_{0,1}, \theta_{0,2}$, which stack the lower triangular part of $\Sigma_{0,1}, \Sigma_{0,2}$, respectively,} excluding the diagonal terms, belong to $\Rb^p$ with $p=d(d-1)/2=45$. The number of zero coefficients in \textcolor{black}{$\theta_{0,1}, \theta_{0,2}$ is $38$}, i.e., approximately $85\%$ of their total number of entries \textcolor{black}{are zero coefficients. The non-zero entries are generated from the uniform distribution $\Uc([-0.7,-0.05]\cup[0.05,0.7])$. For them, we have $\max_{k \in \Ac}\,|\theta_{0,1,k}| = 0.4417$, $\min_{k \in \Ac}\,|\theta_{0,1,k}| = 0.0631$; $\max_{k \in \Ac}\,|\theta_{0,2,k}| = 0.6518$, $\min_{k \in \Ac}\,|\theta_{0,2,k}| = 0.1041$. Moreover, the true vector $\theta_{0,1}$ contains three non-zero entries smaller (in absolute value) than $0.1$. The latter parameters $\theta_{0,1}$ and $\theta_{0,2}$ will be left fixed hereafter.} 

Then, for the sample size $n=500$, we draw $\UU_i \in \Rb^d$, $i\in \{1,\ldots,n\}$, from the sparse Gaussian copula with parameter \textcolor{black}{$\Sigma_{0,1}$} and apply the rank-based transformation to obtain the $\widehat{\UU}_i$. \textcolor{black}{Equipped with} the pseudo-sample $\widehat{\UU}_i$, $i\in \{1,\ldots,n\}$, we solve the same penalized criteria as in the main \textcolor{black}{text} for the Gaussian copula (Gaussian loss and least squares loss) with SCAD and MCP penalty functions. \textcolor{black}{We will use a} grid of different $a_{\text{scad}}$ and $b_{\text{mcp}}$ values: $a_{\text{scad}} \in\{ 2.1, 2.5, 3, 3.5, \ldots,\textcolor{black}{25}\}$ and $b_{\text{mcp}} \in \{ 0.1, 0.5, 1, 1.5,\ldots,\textcolor{black}{25}\}$. We repeat this procedure for $100$ independent batches. \textcolor{black}{The same experiment is performed with the true Gaussian copula parameter $\Sigma_{0,2}$}. 

\textcolor{black}{Figure~\ref{c1_c2_case1} and Figure~\ref{mse_case1} display the metrics C1/C2 and MSE, respectively, averaged over these $100$ batches, when the true parameter is $\Sigma_{0,1}$. Figure~\ref{c1_c2_case2} and Figure~\ref{mse_case2} display the same metrics when the true parameter is $\Sigma_{0,2}$.} 
For instance, on Figure~\ref{scad_c1_c2_case1}, the red solid line represents the percentage of the true zero coefficients in \textcolor{black}{$\theta_{0,1}$} that are correctly recovered by $\widehat{\theta}$ when deduced from the Gaussian loss penalized by the SCAD penalty for different values of $a_{\text{scad}}$ and averaged over the $100$ batches. Similarly, on Figure~\ref{mcp_c1_c2_case1}, the blue dash-dotted line represents the percentage of the true non-zero coefficients in \textcolor{black}{$\theta_{0,1}$} correctly recovered by $\widehat{\theta}$ when deduced from the least squares loss penalized by the MCP penalty for some values of $b_{\text{mcp}}$ and averaged over the $100$ batches. 

Figures \ref{scad_c1_c2_case1} and \ref{mcp_c1_c2_case1} highlight the existence of a trade-off between the recovery of the zero and non-zero coefficients, \textcolor{black}{particularly} for the Gaussian loss: smaller $a_{\text{scad}}$ and $b_{\text{mcp}}$ provide better C1 but to the detriment of C2, whereas larger values imply worse C1, but better C2. \textcolor{black}{Interestingly, a different recovery pattern of the true non-zero coefficients is displayed in Figure~\ref{scad_c1_c2_case2} and Figure~\ref{mcp_c1_c2_case2}: the metric C2 is close to $100\%$ for all $a_{\text{scad}}, b_{\text{mcp}}$ (excluding the values $b_{\text{mcp}}<1.5$ in the MCP case). Since $\theta_{0,2}$ contains a true minimum signal $\min_{k \in \Ac}\,|\theta_{0,2,k}|$ sufficiently large, the penalization procedure can correctly recover all its non-zero entries. On the other hand, $\theta_{0,1}$ includes three values smaller than $0.1$. In that case, the penalization tends to estimate the true small non-zero entries as zero entries, thus worsening the C2 metric. The pattern of C1 for both penalty functions and loss functions is not much affected by $\theta_{0,1}, \theta_{0,2}$. For both $\theta_{0,1}, \theta_{0,2}$ cases, and for both SCAD and MCP,} the Gaussian loss-based criterion tends to generate more zero coefficients than the least squares-based criterion. Furthermore, the performances in terms of estimation accuracy (MSE) are in favor of the Gaussian loss-based criterion. Larger $a_{\text{scad}}$ and $b_{\text{mcp}}$ values result in larger MSE, which is in line with the findings in the main \textcolor{black}{text: large $a_{\text{scad}}$ and $b_{\text{mcp}}$ values result in a LASSO penalization, a situation where the penalty increases linearly with respect to the absolute value of the coefficient, thus generating more bias}. \textcolor{black}{Moreover, the MSE metric is smaller when the true parameter is $\theta_{0,2}$: compare Figure~\ref{mse_case1} and Figure~\ref{mse_case2}. This is in line with the fact that the support recovery is better for $\theta_{0,2}$ than for $\theta_{0,1}$.
}

A point worth mentioning are the poor performances of the MCP method for too small values of $b_{\text{mcp}}$. In the latter case, the set $\{|\theta|>b_{\text{mcp}}\lambda_n\}$ is likely to be active, and the penalty becomes $\lambda^2_n\,b_{\text{mcp}}/2$, so that the penalization is almost vanishing, which results in lower C1 and larger C2, as depicted in Figure~\ref{mcp_c1_c2_case1} when $0<b_{\text{mcp}}<1$. But a low C1 implies that the penalization misses a large number of true zero entries, which prompts large MSE patterns. 

\mds 

\textcolor{black}{This simulated experiment on sparse Gaussian copula} suggests the existence of an interval of ``optimal'' values of $a_{\text{scad}}$ and $b_{\text{mcp}}$, where optimality refers to the ideal situation of perfect recovery of the zero and non-zero coefficients with a low MSE. Such coeffficients should not be too small and too large: for the SCAD, $a_{\text{scad}} \in (3,\textcolor{black}{9})$ provides an optimal compromise between C1 and C2 with low MSE; this is the case for the MCP when $b_{\text{mcp}} \in (2,9)$.

\begin{figure}[!ht]
\centering
\subfloat[SCAD \label{scad_c1_c2_case1}]{\includegraphics[width=7.1cm,height=5.2cm]{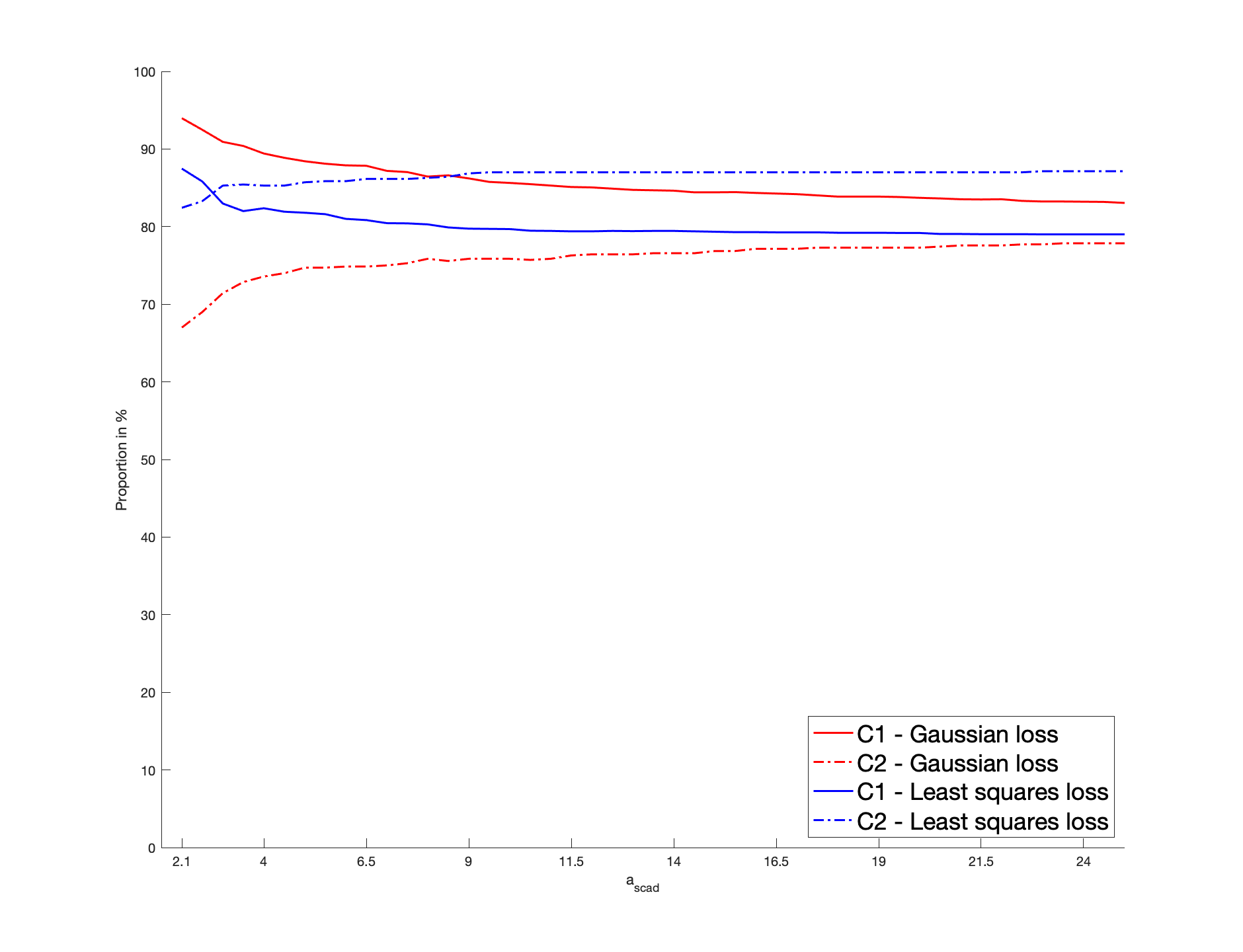}}
\subfloat[MCP \label{mcp_c1_c2_case1}]{\includegraphics[width=7.1cm,height=5.2cm]{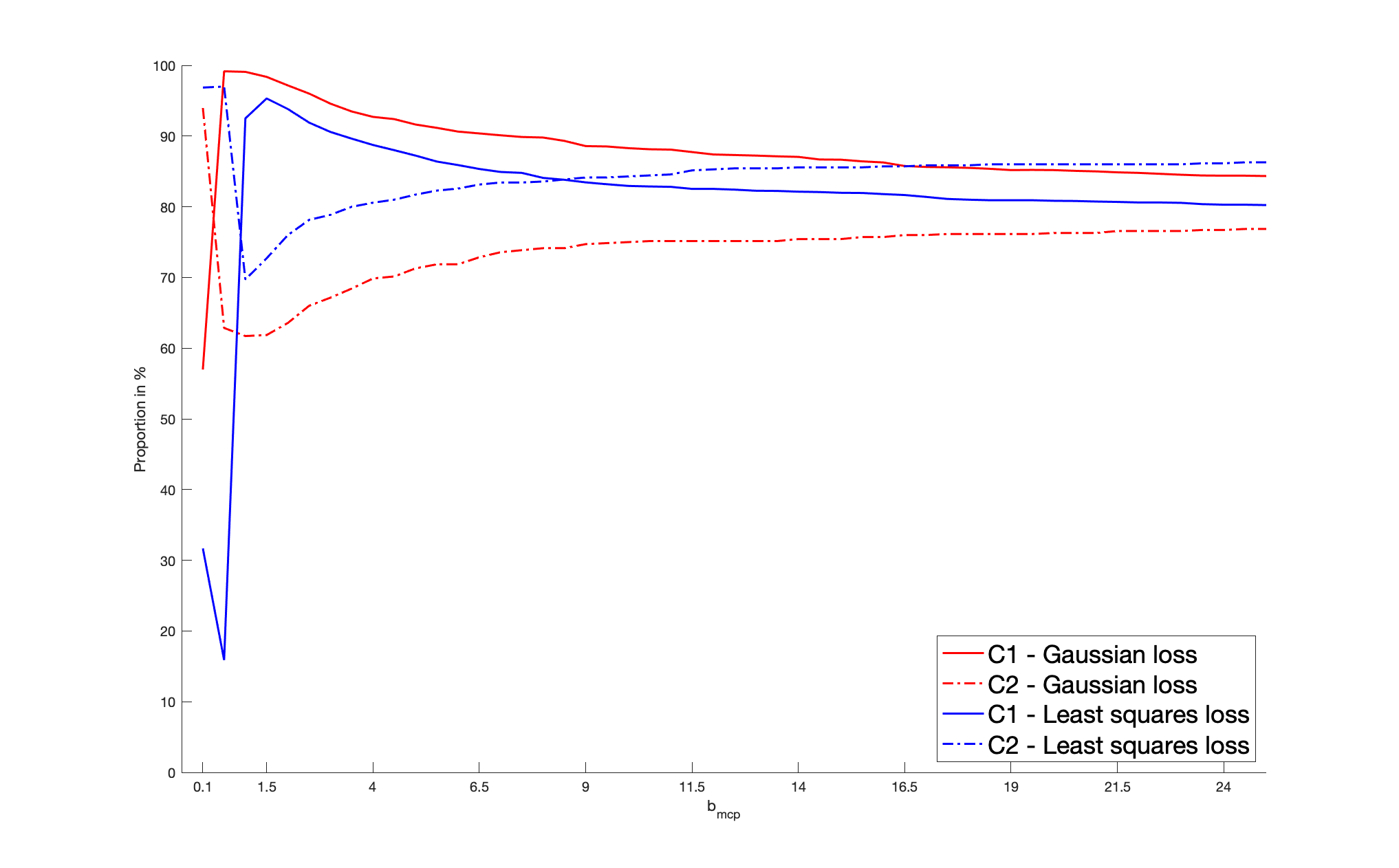}} \caption{Percentage of zero (C1) and non-zero (C2) coefficients that are correctly estimated. Each point represents an average over $100$ batches. \textcolor{black}{The experiment based on $\Sigma_{0,1}$}.} \label{c1_c2_case1}
\end{figure}
\begin{figure}[!ht]
\centering
\subfloat[SCAD]{\includegraphics[width=7.1cm,height=5.2cm]{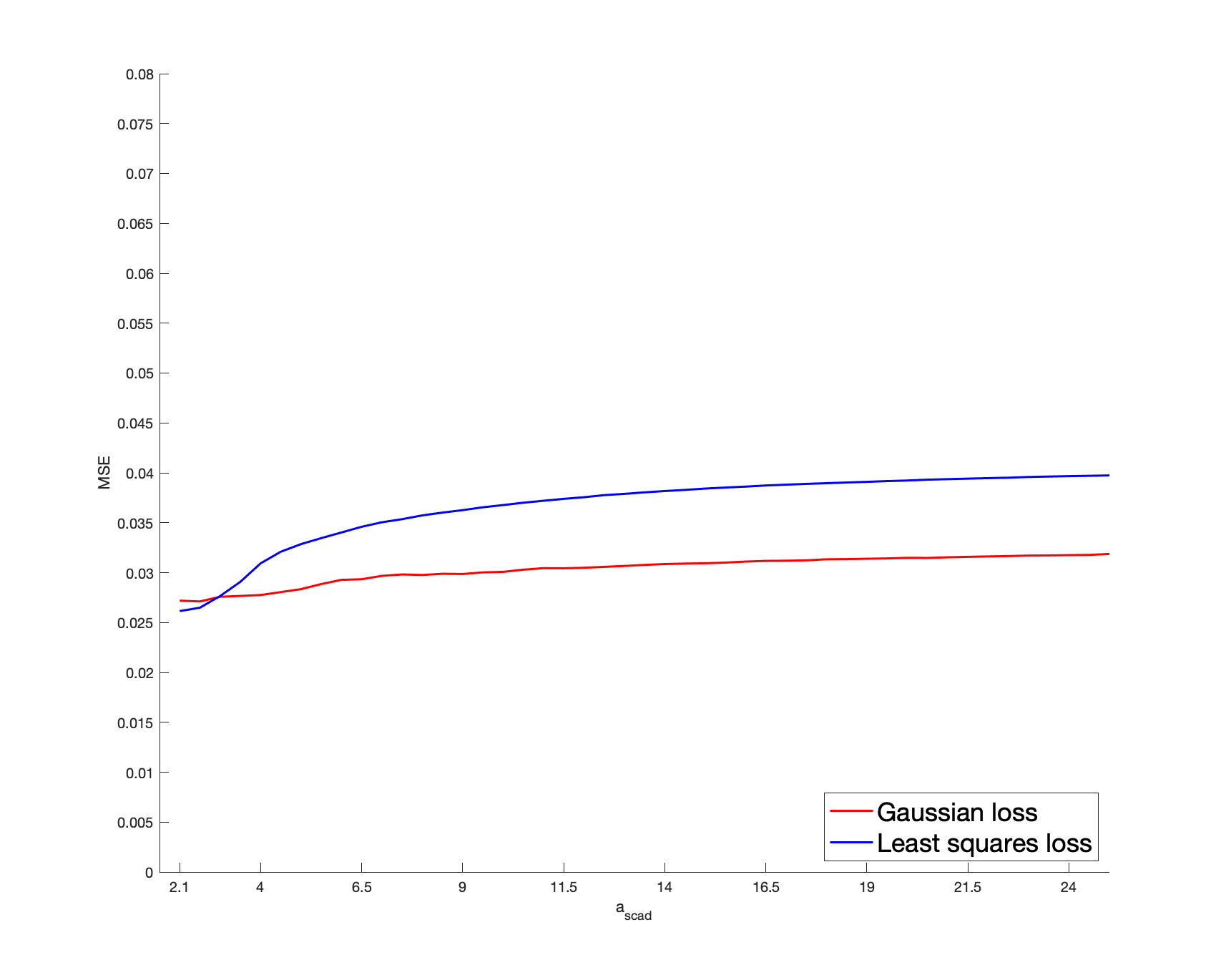}}
\subfloat[MCP]{\includegraphics[width=7.1cm,height=5.2cm]{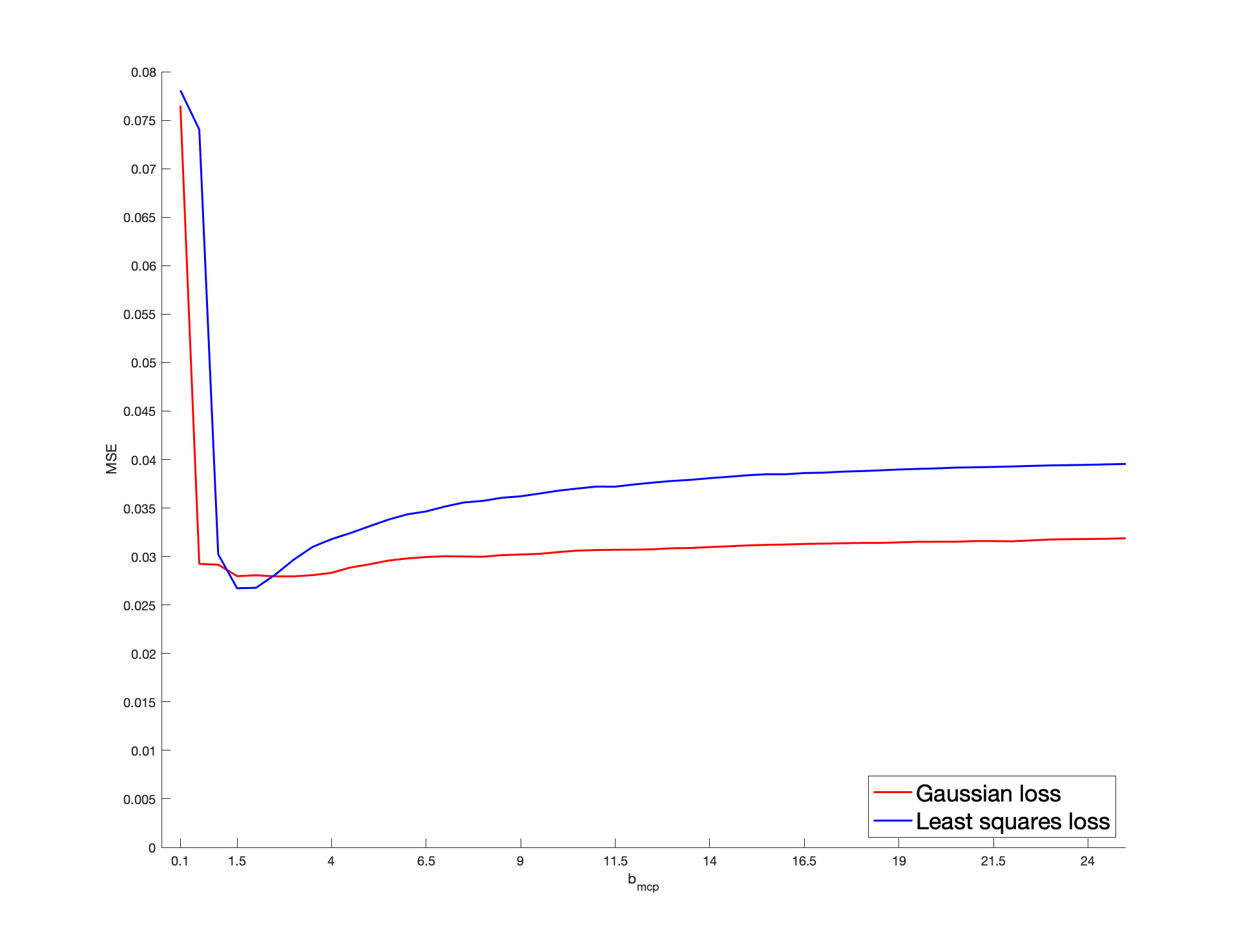}} \caption{Mean squared errors. Each point represents an average over $100$ batches. \textcolor{black}{The experiment based on $\Sigma_{0,1}$}.} \label{mse_case1}
\end{figure}
\begin{figure}[!ht]
\centering
\subfloat[SCAD \label{scad_c1_c2_case2}]{\includegraphics[width=7.1cm,height=5.2cm]{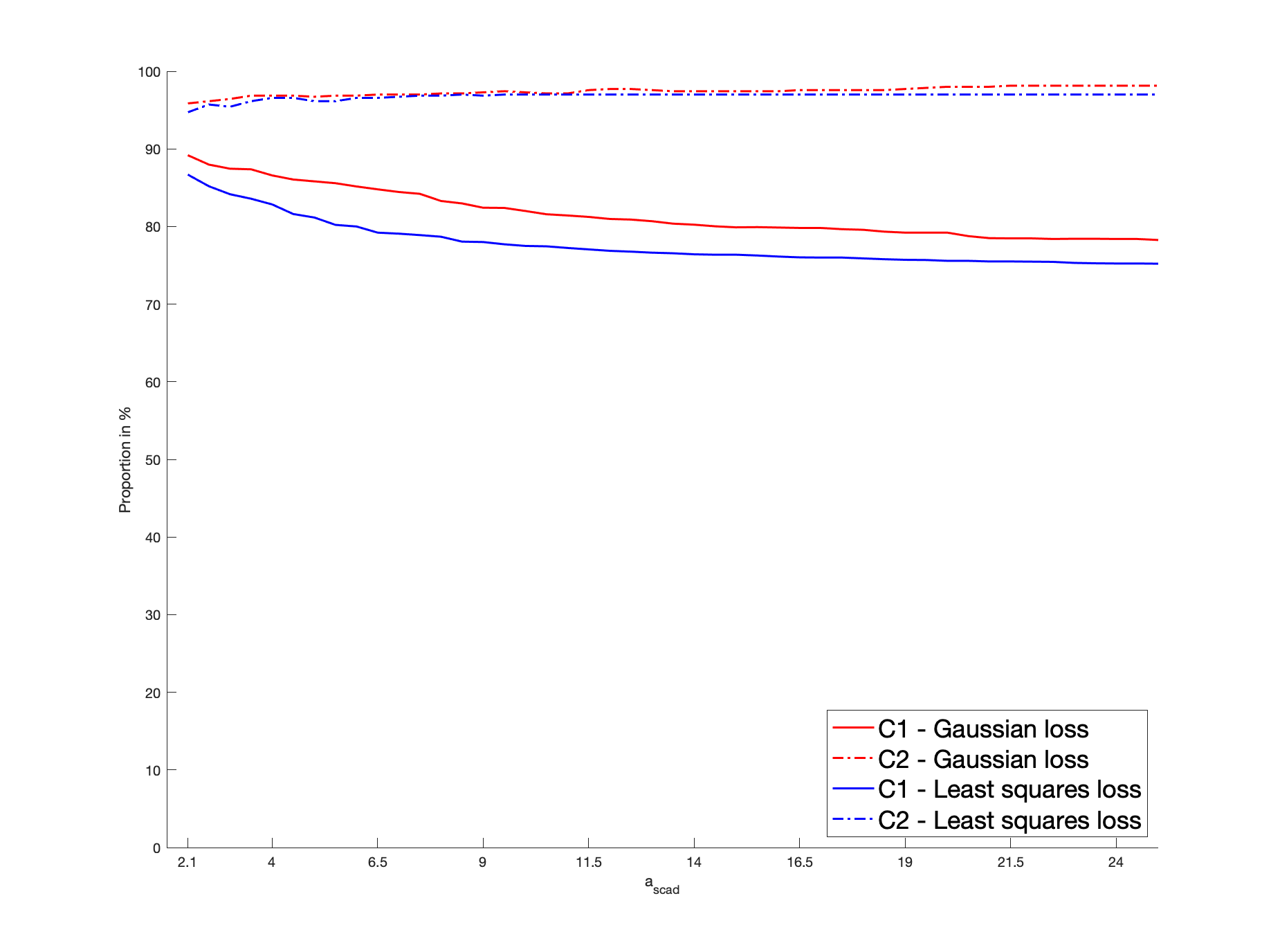}}
\subfloat[MCP \label{mcp_c1_c2_case2}]{\includegraphics[width=7.1cm,height=5.2cm]{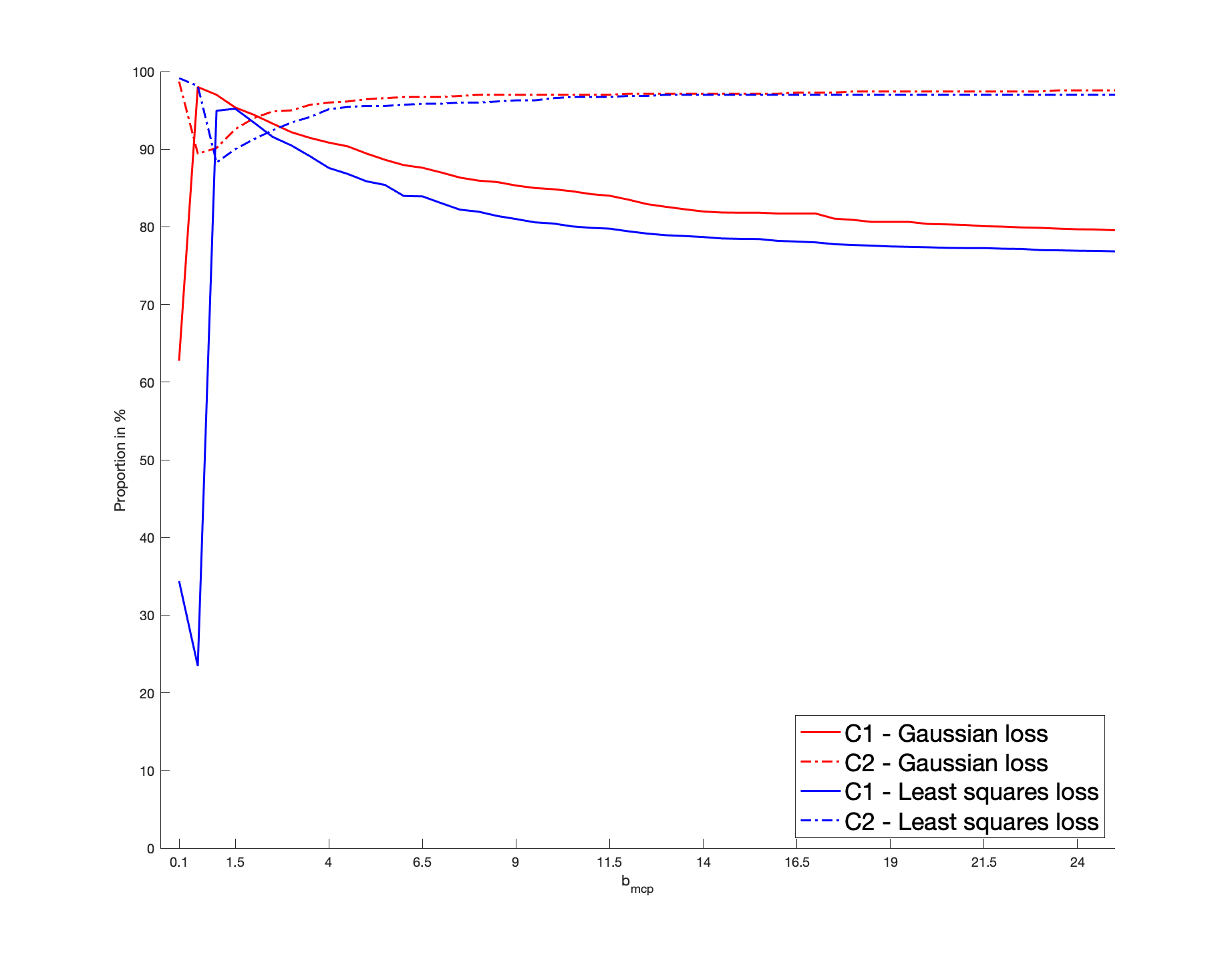}} \caption{Percentage of zero (C1) and non-zero (C2) coefficients that are correctly estimated. Each point represents an average over $100$ batches. \textcolor{black}{The experiment based on $\Sigma_{0,2}$}.} \label{c1_c2_case2}
\end{figure}
\begin{figure}[!ht]
\centering
\subfloat[SCAD]{\includegraphics[width=7.1cm,height=5.2cm]{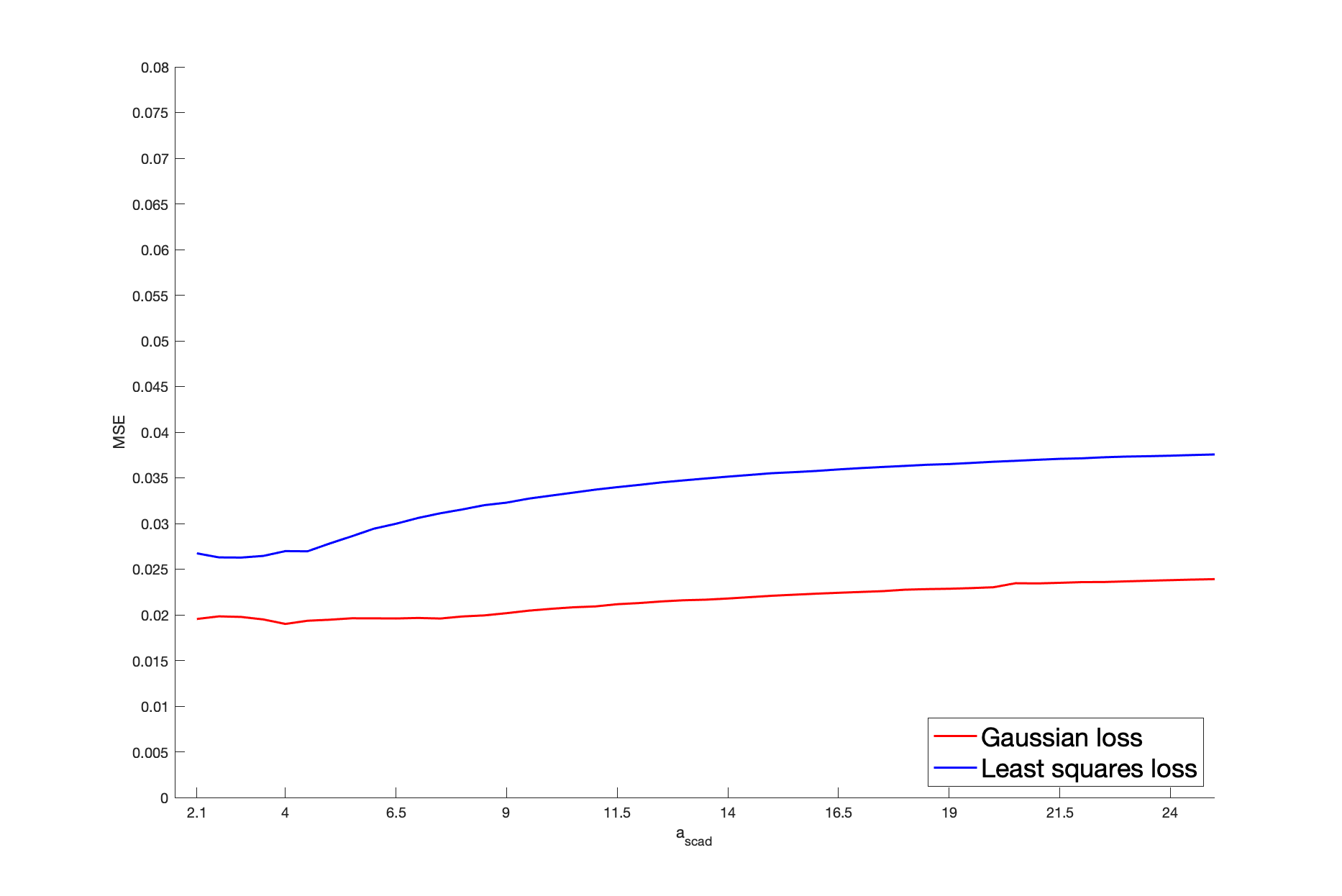}}
\subfloat[MCP]{\includegraphics[width=7.1cm,height=5.2cm]{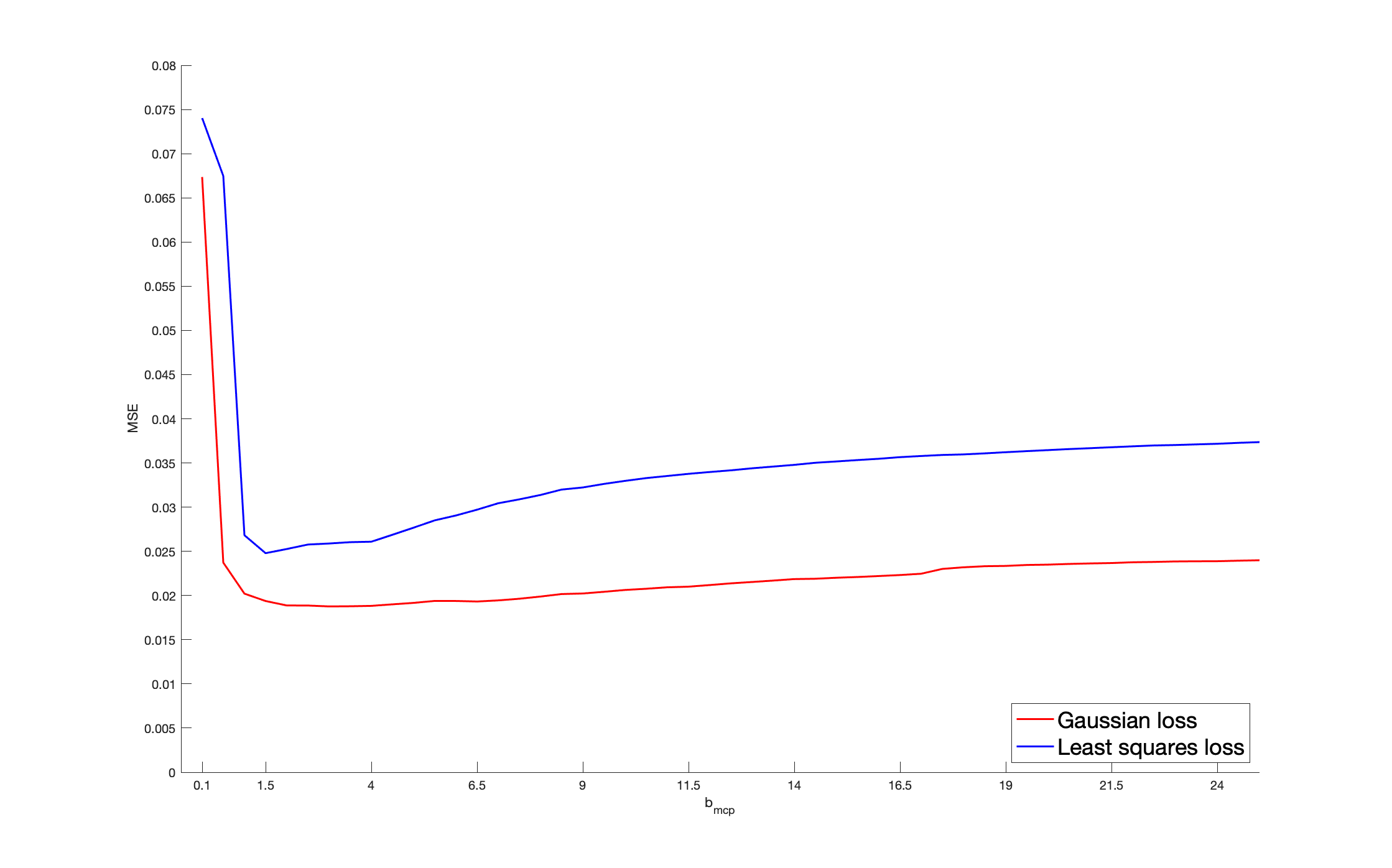}} \caption{Mean squared errors. Each point represents an average over $100$ batches. \textcolor{black}{The experiment is based on $\Sigma_{0,2}$}.} \label{mse_case2}
\end{figure}

\end{document}